\documentclass[12pt]{article}
\usepackage[english]{babel}
\usepackage[letterpaper,top=2cm,bottom=2cm,left=3cm,right=3cm,marginparwidth=1.75cm]{geometry}
\usepackage{amsmath}
\usepackage{graphicx}
\usepackage[colorlinks=true, allcolors=blue]{hyperref}
\usepackage{booktabs} 
\usepackage{multirow} 
\usepackage{siunitx}
\usepackage{array} 
\usepackage{amsthm,amsmath,amssymb}
\usepackage{mathrsfs}
\usepackage{cite}
\usepackage{subcaption}
\usepackage{caption}
\usepackage{algorithmic}
\usepackage{algorithm}
\usepackage{ulem}
\usepackage{tikz}
\usetikzlibrary{fit}
\usetikzlibrary{arrows.meta, positioning}
\usepackage{xcolor}

\numberwithin{equation}{section}
\newtheorem{definition}{Definition} [section]
\newtheorem{assumption}[definition]{Assumption}
\newtheorem{theorem}[definition] {Theorem}
\newtheorem{lemma}[definition]{Lemma}

\title{A model reduction method based on nonlinear optimization for stochastic multiscale  optimal control problems}
\author{Lingling Ma, Jingyi Zhang, Qiuqi Li}

\begin{document}
\maketitle
\begin{abstract}

This paper proposes a non-intrusive, data-driven reduced-order modeling framework for stochastic optimal control problems governed by partial differential equations. The control problem is formulated with a quadratic cost functional and stochastic PDE constraints, and an $\mathcal{L}_2$-optimal reduced-order model is constructed to directly approximate the parameter-to-output mapping. The model is obtained by minimizing the $\mathcal{L}_2$ norm of the output error via gradient-based optimization, requiring only input–output data without access to the full-order system matrices or state variables.
To efficiently generate high-fidelity training data for multiscale problems, the Generalized Multiscale Finite Element Method (GMsFEM) is employed as an offline solver. The proposed framework ensures accuracy in control-relevant outputs while maintaining computational complexity independent of the original PDE dimension, making it suitable for real-time applications.
Numerical experiments on stochastic diffusion and advection–diffusion equations demonstrate the accuracy, efficiency, and robustness of the method.

\end{abstract}

 \section{Introduction}
Optimal control problems play a crucial role in various fields of science and engineering, where the goal is to determine the optimal control strategy that minimizes or maximizes a given objective function while satisfying certain constraints. When modeling physical processes with partial differential equations (PDEs), these problems often involve complex interactions between the control inputs and the underlying dynamics of the system. However, in real-world applications, uncertainties are inevitably present in the system parameters, boundary conditions, external loadings, or even the shape of the physical domain. These uncertainties can significantly impact the performance of the optimal control strategy, making it essential to incorporate stochastic information into the control problem formulation.

In recent years, stochastic optimal control problems governed by stochastic PDEs have attracted increasing attention due to their ability to account for uncertainties in a rigorous and systematic manner. By introducing random variables to parameterize the stochastic functions, these problems provide a more realistic framework for modeling and optimizing control strategies under uncertain conditions. For deterministic optimal control problems, extensive research has been conducted over the past decades, resulting in well-developed mathematical theories and efficient computational methods \cite{lions1971optimal,glowinski1996exact,rees2010optimal,manzoni2021optimal}. In contrast, the development of stochastic optimal control problems, especially those involving stochastic PDEs, has only gained substantial progress in the last few decades \cite{gunzburger2011error,hou2011finite,rosseel2012optimal,lee2013stochastic,chen2013stochastic}. Despite these advancements, there remain significant challenges in both the theoretical and computational aspects of stochastic optimal control, particularly in handling high-dimensional uncertainties and developing efficient numerical algorithms.

In this study, we focus on stochastic optimal control problems with a quadratic cost functional that are subject to constraints imposed by stochastic partial differential equations. Within the realm of PDE-constrained optimization, there exists a fundamental decision point: whether to adopt a discretize-then-optimize approach or an optimize-then-discretize approach. The literature presents divergent views on the merits of each strategy \cite{Collis2002Analysis}. After careful consideration, we have elected to pursue the optimize-then-discretize methodology in the present paper. For the simulation of stochastic optimal control problems, several efficient methods exist, including the Monte Carlo method\cite{ali2017multilevel,guth2021quasi,oz2022efficient},  Stochastic Galerkin method \cite{hou2011finite,babuska2004galerkin}, stochastic collocation method\cite{rosseel2012optimal,kouri2013trust} and so on. However, when the optimality system involves high-dimensional stochastic spaces, these techniques often face significant challenges. Specifically, they can suffer from a low convergence rate and require repeated computation of the optimality system for a large number of sample or collocation nodes, leading to considerable computational cost.

In solving stochastic optimal control problems, particularly those involving PDE-constrained optimization, a range of significant challenges arise\cite{antil2018frontiers,troltzsch2010optimal}. Firstly, the computational cost of high-dimensional integration increases exponentially with the dimensionality, leading to the "curse of dimensionality". Secondly, solving PDE-constrained optimization problems requires solving multiple partial differential equations simultaneously, including the state equation, the adjoint equation, and a set of equations ensuring the optimality of the solution, making the computational process very time-consuming and complex. Furthermore, in the context of many queries for parameterized PDEs, the need to solve for multiple parameter values significantly increases the computational load, especially when the parameter space is large or infinite. Finally, in high-dimensional stochastic spaces, the exponential growth of computational complexity makes numerical computation extremely difficult and time-consuming. These challenges necessitate the selection and development of efficient numerical methods, such as reduced-order models, adaptive mesh methods, or parallel computing techniques, to improve computational efficiency and the feasibility of practical applications.

Model reduction methods have become indispensable tools for tackling high dimensional control problems governed by PDEs. Classical approaches such as Proper Orthogonal Decomposition (POD)\cite{POD1,POD2,POD3} and Galerkin projection\cite{Galerkin1,Galerkin2} construct reduced-order model (ROM) by projecting the full-order model (FOM) onto carefully chosen subspaces. While these methods have achieved success in deterministic scenarios, they exhibit significant limitations when dealing with stochastic optimal control problems: First, such methods typically require direct manipulation of the full order model’s system matrices. Second, when handling high-dimensional stochastic parameter spaces, the combination of traditional Monte Carlo sampling and POD faces the challenge of the "curse of dimensionality." Finally, existing methods mainly focus on the approximation accuracy of state variables while neglecting the optimization of output errors directly related to control objectives.

In order to address the above issues, this paper presents a non-intrusive $\mathcal{L}_2$-optimal model order reduction approach for stochastic optimal control problems. The method innovatively employs a parameter-separable form to handle the dependence on stochastic parameters, and directly minimizes the $\mathcal{L}_2$ norm of output errors through gradient optimization techniques\cite{L2}. This method was originally motivated by the work on $\mathcal{H}_2$-optimal model order reduction\cite{H2,H} for nonparametric LTI systems. Compared with existing methods, this approach has three significant advantages: Firstly, it is completely data-driven based on output measurements, eliminating the need to access the internal system matrices. Secondly, it guarantees the approximation accuracy of control outputs, directly aligning with the optimization of control objectives. In addition, its computational complexity is independent of the original PDE dimensionality, ensuring feasibility for real-time control applications.

However, even with a data-driven reduced order framework, constructing a high-fidelity ROM still relies on the full order model to generate a sufficient number of training samples (snapshots) covering the parameter space. For complex PDE constraints with multiscale features, performing large-scale full order solutions on a fine mesh remains computationally expensive, which presents a front-end bottleneck for the practical application of data-driven methods. To overcome this bottleneck, various multiscale methods have been developed \cite{engquist2002heterogeneous,hou1997multiscale,jiang2010mixed}. The main idea of multiscale methods is to decompose the fine-scale problem into a set of localized subproblems, then use their individual solutions to systematically construct an accurate coarse-scale governing equation \cite{jiang2010mixed}. Among these, the Multiscale Finite Element Method \cite{efendiev2009multiscale,efendiev2004multiscale} and its generalization, the Generalized Multiscale Finite Element Method (GMsFEM)\cite{GMs,GMsFEM}, have emerged as effective local model reduction techniques, providing a robust framework for this multiscale computational approach. By computing local multiscale basis functions that capture fine-scale information during the offline stage, GMsFEM can rapidly and accurately approximate full order solutions on a coarse grid. This makes the generation of large-scale training datasets, as well as providing high-quality initial guesses for data-driven reduced-order optimization, highly efficient. Consequently, the computational burden is shifted from expensive online fine mesh solving to parallelizable offline basis function computations. Furthermore, active research efforts continue to develop novel model reduction techniques, such as the constraint energy minimizing generalized multiscale finite element method (CEM-GMsFEM)\cite{chung2018constraint,chung2018fast,fu2020constraint}.
In this paper, GMsFEM is employed solely as an offline solver to generate high-fidelity realizations of the optimal control for sampled random parameters. Based on these data, a non-intrusive $\mathcal{L}_2$-optimal reduced-order model with parameter-separable structure is constructed to directly approximate the parameter-to-control map, without projecting the governing equations or the optimality system.

The paper is structured as follows. Section 2 presents the definition of stochastic optimal control problem and some relevant preliminaries. In Section 3, we elaborate in detail on the construction of the $\mathcal{L}_2$-optimal reduced-order model for these problems. Section 4 introduces the generalized multiscale finite element method. The effectiveness of the proposed approach is then verified in Section 5 through numerical experiments on PDE-constrained problems with stochastic coefficients.

The flowchart of the proposed method in this paper is given as follows:
\begin{figure}[H]
\centering
\resizebox{\textwidth}{!}{
\begin{tikzpicture}[
    node distance=0.6cm and 1.0cm,
    every node/.style={font=\small},
    block/.style={
        rectangle, draw, rounded corners=2pt,
        align=center,
        minimum width=2.8cm, 
        minimum height=1.0cm,
        inner sep=2pt,
        line width=1pt
    },
    arrow/.style={->, thick, >=stealth},
    group/.style={
        rectangle, dashed, rounded corners,
        draw, inner sep=6pt, line width=1pt
    }
]

\usetikzlibrary{fit}

\definecolor{fomcolor}{RGB}{230,230,230}
\definecolor{rbcolor}{RGB}{180,220,250}
\definecolor{optcolor}{RGB}{250,190,190}

\node[block, fill=fomcolor] (fom)
{FOM\\(e.g. FEM or GMsFEM)};

\node[block, fill=rbcolor, right=0.4cm of fom] (snap)
{Snapshot collection\\$u(\mu_1), \dots, u(\mu_M)$};

\node[block, fill=rbcolor, right=0.4cm of snap] (rb)
{RB basis (POD/Greedy)\\$V \in \mathbb{R}^{N \times r}$};

\node[block, fill=rbcolor, below=0.8cm of rb] (rom0)
{ROM\\$(\hat{A}_r^0,\hat{B}_r^0,\hat{C}_r^0)$};

\node[block, fill=optcolor, below=0.8cm of rom0] (l2)
{$\mathcal{L}_2$ data-driven optimization\\$\min \| y - \hat{y} \|_{L_2}$};

\node[block, fill=optcolor, right=0.4cm of l2] (rom)
{DDROM\\$(\hat{A}_r^*,\hat{B}_r^*,\hat{C}_r^*)$};

\draw[arrow] (fom) -- (snap);
\draw[arrow] (snap) -- (rb);
\draw[arrow] (rb) -- (rom0);
\draw[arrow] (rom0) -- (l2);
\draw[arrow] (l2) -- (rom);

\node[group, draw=blue!60!black, fit=(snap)(rb)(rom0)] (groupRB) {};
\node[above=0.25cm of groupRB, text=blue!70!black, font=\small]
{Proper initialization (RB method)};

\node[group, draw=red!60!black, fit=(l2)(rom)] (groupDD) {};
\node[below=0.25cm of groupDD, text=red!70!black, font=\small]
{Data-driven correction ($\mathcal{L}_2$ optimization)};

\end{tikzpicture}
}
\caption{Framework of $\mathcal{L}_2$ data-driven reduced-order modeling initialized by RB}
\label{fig:rb_l2_framework}
\end{figure}

As shown in Figure \ref{fig:rb_l2_framework}, the proposed method consists of two stages: physics-based initialization and data-driven correction. First, the RB method is employed to obtain an initial reduced-order model with physical consistency. Subsequently, $\mathcal{L}_2$ gradient optimization is applied to further enhance the model accuracy, thereby achieving a favorable balance between computational efficiency and approximation accuracy.

\section{Preliminaries}
This paper proposes a data-driven, nonlinear optimization method for stochastic control problem constrained by partial differential equations. In this section, we first introduce the notations used throughout the paper and  present the definition of stochastic optimal control problem. We then provides the discrete approximation for the optimal control problems using the finite element method.

\subsection{Stochastic optimal control problem}
Let $(\mathcal{D},\mathscr{F})$ be a measurable space, where $\mathcal{D}$ is the set of possible outcomes $\omega \in \mathcal{D}$ and $\mathscr{F} \subset 2^{\mathcal{D}}$ is a $\sigma$-algebra of events. We denote the expected value of a random variable $X : \mathcal{D} \rightarrow \mathbb{R}$ with respect to a probability measure $\mathcal{P}:\mathscr{F} \rightarrow[0, 1]$ defined on the measurable space $(\mathcal{D}, \mathscr{F} )$ by
$$
\mathbb{E}_P\left[X\right]=\int_{\mathcal{D}}X(\omega)dP(\omega).
$$
We denote the usual Lebesgue space of $r \in[1, \infty)$ integrable real-valued functions by
$$
L^r(\mathcal{D},\mathscr{F},P):=\left\{\theta:\mathcal{D}\to\mathbb{R}:\theta\mathrm{~is~}\mathscr{F}\text{-measurable, }\mathbb{E}_P[|\theta|^r]<\infty\right\}.
$$
The Lebesgue spaces defined on $(\mathcal{D}, \mathscr{F}, \mathcal{P} )$ are Banach spaces and serve as natural spaces for real-valued random variables, i.e., $\mathscr{F}$-measurable functions. 

We assume that $\Omega$ is a convex bounded polygonal domain in $\mathbb{R}^d(d \geq 1)$  with Lipschitz continuous boundary $\partial\Omega$. Given a real Hilbert space $H^s(\Omega)$, the tensor-product vector space associated with $L^2(\mathcal{D}, \mathscr{F} ,\mathcal{P})$ and $H^s(\Omega)$ is 
$$
L^2(\mathcal{D}, \mathscr{F}, \mathcal{P}) \otimes H^s(\Omega) := \text{span}\left\{\theta v : \theta \in L^2 (\mathcal{D}, \mathscr{F}, \mathcal{P}), v \in H^s(\Omega) \right\}.
$$ 
We denote $\mathscr{H}^s(\Omega):=L^2(\mathcal{D}, \mathscr{F}, \mathcal{P})\otimes  H ^s(\Omega)  $ to shorten the notation and equip it with the following norm
$$
\|u\|_{\mathscr{H}^s(\Omega)}=\mathbb{E}_P[\|u\|_{H^s}^2]^{\frac12}.
$$

In particular, $\mathscr{H}^s_0(\Omega) = \left\{u\in \mathscr{H}^s(\Omega) : u|_{\partial\Omega} = 0\right\}$. When $s = 0$, we employ the abbreviated notion $\mathscr{L}^2(\Omega)$ to denote $\mathscr{H}^0(\Omega)$ by convention.
 
The stochastic optimal control problem is to minimize the objective functional under some constraints, which is mainly divided into distributed control problems and boundary control problems. The main research object of this paper is the former, and 
now we give the form of the stochastic distributed control problem. The objective functional is 
\begin{equation}
\label{obj_fun}
\min_{\substack{u\in\mathscr{H}^{1}(\Omega)\\f\in\mathscr{L}^{2}(\Omega)}}J(u,f):=\frac{1}{2}\|u(x,\mu(\omega))-\hat{u}(x,\mu(\omega))\|_{\mathscr{L}^{2}(\Omega)}^{2}+\beta\|f(x,\mu(\omega))\|_{\mathscr{L}^{2}(\Omega)}^{2}
\end{equation}
which is constrained by a stochastic PDE with the following variational from
\begin{equation}
\label{PDE_weak}
a(u,v;\mu(\omega))=(f,v;\mu(\omega)), \forall v\in\mathscr{H}_0^1(\Omega),
\end{equation}
subject to the Dirichlet boundary condition $u|_{\partial\Omega} = g(x)$. Here, $J(u, f):\Omega\times\mathcal{D}\rightarrow\mathbb{R}$ is the cost functional. 

In the model problem, we assume that both the control function $f(x, \mu(\omega))$ and the state function $u(x, \mu (\omega))$ are random fields represented by a random vector $\mu(\omega)$. The aim of the stochastic control problem is to select a suitable $f(x, \mu(\omega))$ so that the corresponding $u(x,\mu(\omega))$ is the best possible approximation of the expected $\hat{u}(x,\mu(\omega))$. The second term in \eqref{obj_fun} serves as a regularizer, known as Tikhonov regularization. Here, $\beta$ is referred to as a regularization parameter. This term prevents the control from becoming locally unbounded and ensures that the cost functional
functional $J(u, f)$ does not approach its minimum in an ill-posed or numerically unstable manner.

Finally,  according to \eqref{obj_fun}-\eqref{PDE_weak} and the boundary condition, we rewrite the stochastic control problem in the following form:
\begin{equation}
\label{sto_DCP}
\left\{  
    \begin{aligned}
     \min_{\substack{u\in\mathscr{H}^{1}(\Omega)\\f\in\mathscr{L}^{2}(\Omega)}}J(u,f)&=\frac{1}{2}\|u-\hat{u}\|_{\mathscr{L}^{2}(\Omega)}^{2}+\beta\|f\|_{\mathscr{L}^{2}(\Omega)}^{2},\\
     s.t.\quad a(u,v;\mu(\omega))&=(f,v;\mu(\omega))+b(v;\mu(\omega)), \forall v\in\mathscr{H}_0^1(\Omega),\; in\;\Omega
    \end{aligned}
\right.  
\end{equation}

\subsection{Formulation and structure}
\label{section_sto_form}
To numerically solve \eqref{sto_DCP}, two principal approaches are commonly employed: discretize-then-optimize and optimize-then-discretize. While these two methods may, in some cases, yield identical numerical results, they differ fundamentally in their procedural structure. This subsection focuses on a detailed exposition of the latter approach.

Firstly, we derive the stochastic optimality system of the optimal control problem by using Lagrangian approach, and then give the framework of FE approximation for stochastic optimal control problems.

We define the following stochastic Lagrangian functional as
\begin{equation}
\label{Lag_fun}
    \mathcal{L}(u,f,\lambda):=J(u,f)+a(u,\lambda;\mu)-(f,\lambda;\mu),
\end{equation}
where $\lambda\in\mathscr{H}_0^1(\Omega)$ is the Lagrangian parameter or adjoint variable. By taking the Fr\'echet derivative of Lagrangian functional \eqref{Lag_fun} for the variables $\lambda$, $u$ and $f$, and evaluating them at $\widetilde{u}$, $\widetilde{f}$ and $\widetilde{\lambda}$, we can obtain the first-order necessary optimality conditions of stochastic control problems \eqref{sto_DCP}, i.e.,
\begin{equation}
\label{opt_sys}
\left\{
    \begin{aligned}
    a(u,\widetilde{u};\mu)&=(f,\widetilde{u};\mu),\;\forall\widetilde{u}\in\mathcal{H}_0^1(\Omega),\\
    a^{\prime}(\lambda,\widetilde{\lambda};\mu)&=-(u-\hat{u},\widetilde{\lambda};\mu),\;\forall\widetilde{\lambda}\in\mathcal{H}_0^1(\Omega),\\
    2\beta(f,\widetilde{f};\mu)&=(\widetilde{f},\lambda;\mu),\;\forall\widetilde{f}\in\mathscr{L}^2(\Omega),
    \end{aligned}
\right.
\end{equation}
where $(\cdot,\cdot;\mu)$ represents the $L^2$ general inner product and $a^{'}(\lambda,\widetilde{\lambda};\mu)=a(\widetilde{\lambda},\lambda; \mu)$ is the adjoint bilinear form. These three equations are, respectively, the state equation, adjoint equation and gradient equation.

The optimality system \eqref{opt_sys} only admits local optimal solutions. To establish the global existence and uniqueness of the optimal solution, it is necessary to derive the stochastic saddle point formula for the optimal control problem \eqref{sto_DCP}.

We first introduce some new definitions to obtain a minimization problem equivalent to the original optimal control problem \eqref{sto_DCP}. Let $\underline{u}=(u,f)\in\mathscr{U}$ and $\underline{v}=(v,h)\in\mathscr{U}$ belong to the tensor space $\mathscr{U}=\mathscr{H}^1(\Omega)\times\mathscr{L}^2(\Omega)$, which is equipped with the norm $\|\underline{u}\|_{\mathscr{U}}=\|u\|_{\mathscr{H}^1(\Omega)}+\|f\|_{\mathscr{L}^2(\Omega)}$. We also define bilinear forms 
\begin{equation}
\label{new_bilinear}
    \begin{cases}
    \mathcal{A}(\underline{u},\underline{v}):=(u,v)+2\beta(f,h;\mu),\\[2ex]\mathcal{B}(\underline{u},q):=a(u,q;\mu)-(f,q;\mu),
    \end{cases}
\end{equation}
where $\mathcal{A}(\cdot,\cdot):\mathscr{U}\times\mathscr{U}\rightarrow\mathbb{R}$ and $\mathcal{B}(\cdot,\cdot):\mathscr{U}\times\mathscr{H}_0^1(\Omega)\rightarrow\mathbb{R}$. Giving $\hat{\underline{u}}=(\hat{u},0)$, $(\hat{\underline{u}},\underline{u})=(\hat{u},u)$ and $\mathscr{U}_{ad}\subset\mathscr{U}$, the equivalent form is
\begin{equation}
\label{equ_form}  
\left\{
    \begin{aligned}
    \min_{\underline{u}\in\mathscr{U}_{ad}}\mathcal{J}(\underline{u})&=\frac12\mathcal{A}(\underline{u},\underline{u})-(\underline{\hat{u}},\underline{u}),\\
    s.t.\;\mathcal{B}(\underline{u},\widetilde{u})&=(g,\widetilde{u})_{\partial\Omega},\;\forall\widetilde{u}\in\mathscr{H}_0^1(\Omega).
     \end{aligned}
\right.
\end{equation}
Moreover, combining \eqref{opt_sys} with \eqref{new_bilinear}, the equivalent saddle point problem for \eqref{equ_form} is to find $(\underline{u},\lambda)\in\mathscr{U}\times\mathscr{H}_0^1(\Omega)$ such that
\begin{equation}
\label{saddle point}
    \begin{cases}
    \mathcal{A}(\underline{u},\underline{v})+&\mathcal{B}(\underline{v},\lambda)=(\hat{\underline{u}},\underline{v}),\quad\;\forall\underline{v}\in\mathscr{U},\\
    &\mathcal{B}(\underline{u},\widetilde{u})=(g,\widetilde{u})_{\partial\Omega},\;\forall\widetilde{u}\in\mathscr{H}_0^1(\Omega).
    \end{cases}
\end{equation}
Please refer to \cite{ma2018local} for the equivalent relationship between the saddle point form of stochastic control problem and other forms, and we will not repeat it here.
\begin{theorem}
    Let $\mathscr{U}_0:=\left\{\underline{u}\in\mathscr{U}:\mathcal{B}(\underline{u},\widetilde{u})=0,\forall\widetilde{u}\in\mathcal{H}_0^1(\Omega)\right\}$ be the kernel space of bilinear form $\mathcal{B}(\cdot,\cdot)$, we can obtain the global existence of a unique solution to the minimization problem \eqref{saddle point}. 
\end{theorem}

The above is to optimize the stochastic control problem \eqref{sto_DCP} and get the saddle point formation, and then discretize it with finite element approximation.

Let $\mathcal{T}_h$ be a uniform partition of the domain $\Omega$, $N_h$ be the number of vertices, $N_e$ be the number of elements in the fine mesh $\mathcal{T}_h$ and the dimension of FE space be $\mathcal{N}$. The FE space defined on the fine mesh is $V^h(\Omega)\subset H^1(\Omega)$, and $V^h_0(\Omega)\subset V^h(\Omega)$ has the vanishing boundary values. Furthermore, we define the tensor spaces $\mathscr{V}_0^h(\Omega):=V_0^h(\Omega)\otimes L^2(\Gamma)$ and $\mathcal{M}_h(\Omega):=M_h(\Omega)\otimes L^2(\Gamma)$, where $M_h(\Omega)$ is the finite dimensional subspace of $L^2(\Omega)$. Consequently, $\mathcal{M}_h(\Omega)$ is also a finite dimensional subspace of $\mathscr{L}^2(\Omega)$. 

Giving any $\mu\in\Gamma$, by applying the Galerkin projection of $\mathscr{U}_h\times\mathscr{V}_0^h(\Omega)\subset\mathscr{U}\times\mathscr{H}_0^1(\Omega)$, where $\mathscr{U}_h:=\mathscr{V}_0^h(\Omega)\times\mathcal{M}_h(\Omega)$, we can get the discrete formulation of the saddle point problem \eqref{saddle point}: find $(\underline{u}_h,\lambda_h)\in\mathscr{U}_h\times\mathscr{V}_0^h$ such that
\begin{equation}
\left\{
\begin{aligned}
\mathcal{A}(\underline{u}_h,\underline{v}_h)+\mathcal{B}(\underline{v}_h,\lambda_h)&=(\underline{\hat{u}},\underline{v}_h),\quad\forall\underline{v}_h\in\mathscr{U}_h,\\
\mathcal{B}(\underline{u}_h,\widetilde{u}_h)&=(g,\widetilde{u}_h)_{\partial\Omega},\;\forall\widetilde{u}_h\in\mathscr{V}_0^h(\Omega).
\end{aligned}
\right.
\end{equation}

The Galerkin formulation of \eqref{opt_sys} is to find $(u_h,f_h,\lambda_h)\in\mathscr{V}_0^h(\Omega)\times\mathcal{M}_h(\Omega)\times\mathscr{V}_0^h(\Omega)$, such that
\begin{equation}
\label{Galerkin}
\left\{
\begin{aligned}
a(u_h,\widetilde{u}_h;\mu)&=(f_h,\widetilde{u}_h;\mu),\;\forall\widetilde{u}_h\in\mathscr{V}_0^h(\Omega),\\
a(\lambda_h,\widetilde{\lambda}_h;\mu)&=-(u_h-\hat{u},\widetilde{\lambda}_h;\mu),\;\forall\widetilde{\lambda}_h\in\mathscr{V}_0^h(\Omega),\\
2\beta(f_h,\widetilde{f}_h;\mu)&=(\widetilde{f}_h,\lambda_h;\mu),\;\forall\widetilde{f_h}\in\mathcal{M}_h(\Omega).
\end{aligned}
\right.
\end{equation}

Assuming that the basis functions of the spaces $\mathcal{M}_h(\Omega)$ and $\mathscr{V}^h_0(\Omega)$ are denoted by $\left\{\phi_k\right\}^{N_e}_{k=1}$ and $\left\{\psi_k\right\}^{N_e}_{k=1}$, respectively, the variables $u_h$, $f_h$, $\lambda_h$ in \eqref{Galerkin} can be represented by the linear combination of the these functions. 
\begin{equation}
\label{linear_com}
u_h=\sum_{i=1}^{N_h}u_{h,i}\psi_i,\quad f_h=\sum_{j=1}^{N_e}f_{h,j}\phi_j,\quad\lambda_h=\sum_{k=1}^{N_h}\lambda_{h,k}\psi_k.
\end{equation}

For the weak form \eqref{PDE_weak}, we assume that both the parameter-dependent bilinear form $a(\cdot,\cdot;\mu)$ and the linear form $(f,\cdot;\mu)$ are affine with respect to $\mu$, i.e.,
\begin{equation}
\label{affine_ass}
\left\{
\begin{aligned}
a(u,v;\mu)&=\sum_{q=1}^{Q_a}Q_a^q(\mu)a^q(u,v)\quad\forall u,v\in H^1(\Omega), \forall\mu\in\Gamma,\\
(f,v;\mu)&=\sum_{q=1}^{Q_f}Q_f^q(\mu)(f^q,v)\quad\forall f^q\in L^2(\Omega), v\in H^1(\Omega), \forall\mu\in\Gamma,
\end{aligned}
\right.
\end{equation}
where $Q_a^q(\mu):\Gamma\rightarrow\mathbb{R}$ and $Q_f^q(\mu):\Gamma\rightarrow\mathbb{R}$ are $\mu$-dependent function, in addition, $a_q(\cdot,\cdot):H^1(\Omega)\times H^1(\Omega)\rightarrow\mathbb{R}$ and $(f^q,\cdot):L^2(\Omega)\times H^1(\Omega)\rightarrow\mathbb{R}$ are independent of $\mu$.

With the assumption \eqref{affine_ass}, the discrete optimality system \eqref{Galerkin} can be rewritten as

\begin{equation*}
\left\{
\begin{aligned}
&\sum_{q=1}^{Q_a}\sum_{i=1}^{N_h}Q_a^q(\mu)u_{h,i}(\mu)a^q(\psi_i,\psi_{i^{\prime}})=\sum_{j=1}^{N_e}f_{h,j}(\mu)(\phi_j,\psi_{i^{\prime}}),\\
&\sum_{q^{\prime}=1}^{Q_a}\sum_{k=1}^{N_h}Q_a^{q^{\prime}}(\mu)\lambda_{h,k}(\mu)a^{q^{\prime}}(\psi_k,\psi_{k^{\prime}})+\sum_{i=1}^{N_h}u_{h,i}(\mu)(\psi_i,\psi_{k^{\prime}})=\sum_{p=1}^{Q_u}\hat{u}_p(\mu)(\overline{\hat{u}}_p,\psi_{k^{\prime}}),\\
&2\beta\sum_{j=1}^{N_e}f_{h,j}(\mu)(\phi_j,\phi_{j^{\prime}})=\sum_{k=1}^{N_h}\lambda_{h,k}(\mu)(\phi_{j^{\prime}},\psi_k).
\end{aligned}
\right.
\end{equation*}
In order to convert the above formula into matrix-vector form, we first assume that
$$\begin{aligned}
    & \left(M_{1}\right)_{j, j^{\prime}}=\left(\phi_j, \phi_{j^{\prime}}\right), 1 \leq j,\;j^{\prime} \leq N_e, \\
    & \left(M_{2}\right)_{k, j^{\prime}}=\left(\phi_{j^{\prime}}, \psi_k\right), 1 \leq j^{\prime} \leq N_e,\;1 \leq k \leq N_h, \\
    & \left(M_{3}\right)_{i, k^{\prime}}=\left(\psi_i, \psi_{k^{\prime}}\right),\;1 \leq i, k, k^{\prime} \leq N_h, \\ 
    &\left(K^q\right)_{k, k^{\prime}}=a^q\left(\psi_k, \psi_{k^{\prime}}\right),\;1 \leq i, k, k^{\prime} \leq N_h,\\ &\left(\widehat{U}_p\right)_{k^{\prime}}=\left(\overline{\hat{u}}_p, \psi_{k^{\prime}}\right), 1 \leq k,k^{\prime} \leq N_h, \\
    & K=\sum_{q=1}^{Q_a} Q_a^q(\mu) K^q,\quad\overline{\widehat{U}}=\sum_{p=1}^{Q_u} \widehat{u}_p(\mu) \widehat{U}_p .
\end{aligned}$$
Here, $\left\{K^q\right\}_{q=1}^{Q_a}$ are stiffness matrices, $\left\{M_{i}\right\}_{i=1}^{3}$ are mass matrices, and $\left\{\widehat{U}_p\right\}_{p=1}^{Q_u}$ are load vectors.
Then, we can get the algebraic formulation of \eqref{Galerkin},
\begin{equation}
\label{alge_form}
\begin{aligned}
&\underbrace{
\begin{bmatrix}
2\beta M_{1}&0&-M_{2}^T\\0&M_{3}&K^T(\mu)\\-M_{2}&K(\mu)&0\end{bmatrix}}_{\mathscr{A}_h(\mu)\in\mathbb{R}^{(2N_h+N_e)\times(2N_h+N_e)}}
\begin{bmatrix}
F(\mu)\\U(\mu)\\\Lambda(\mu)
\end{bmatrix}
=\begin{bmatrix}0\\\overline{\widehat{U}}(\mu)\\d\end{bmatrix},
\end{aligned}
\end{equation}
where $U(\mu)$, $F(\mu)$ and $\Lambda(\mu)$ denote the vectors of the coefficients in the expansion of $u_h(\mu)$, $f_h(\mu)$ and $\lambda_h(\mu)$ and the term coming from the boundary values of $u_h$ is denoted by $d$.

\section{$\mathcal{L}_2$-optimal reduced-order modeling}
In this section, we introduce a non-invasive, data-driven, gradient-based descent algorithm for solving the stochastic control problem formulated in \eqref{alge_form}. Designed for optimization problems with a parameter-separable structure, the proposed method constructs the optimal approximation using only output samples.

\subsection{The establishment of the model}

Given a distributed stochastic control problem in the form of \eqref{sto_DCP}, we can transform it into the algebraic formulation of \eqref{alge_form} by the discretize-then-optimize approach. Defining $x(\mu):=\begin{bmatrix}
F(\mu),U(\mu),\Lambda(\mu)
\end{bmatrix}^T$ in the \eqref{alge_form} as the state quantity, we consider the following problem: giving $\mu \in\Gamma$, evaluate
\begin{equation}\label{interest}
  y(\mu)=\ell(x(\mu)), 
\end{equation}
and $x(\mu)$ satisfy
\begin{equation}\label{FOM}
\begin{aligned}
\begin{bmatrix}
2\beta M_{1}&0&-M_{2}^T\\0&M_{3}&K^T(\mu)\\-M_{}&K(\mu)&0\end{bmatrix}
x(\mu)=\begin{bmatrix}0\\\overline{\widehat{U}}(\mu)\\d\end{bmatrix},
\end{aligned}
\end{equation}
where $\ell$ is a bounded linear functional. Let $x(\mu) \in \mathbb{R}^{(2N_h+N_e)\times1}$ be the state and $y(\mu) \in \mathbb{R}$ be the output. By introducing a matrix $C(\mu)\in \mathbb{R}^{1\times(2N_h+N_e)}$, 
The function defined in problem \eqref{interest} is discretized into 
\begin{equation}\label{int_dis}
    y(\mu)=C(\mu)x(\mu).
\end{equation} We emphasize that our ultimate interest is the output prediction $y(\mu)$: the state variable $x(\mu)$ serves as an intermediary.  The models in \eqref{FOM} and \eqref{int_dis} are referred to as the full-order model (FOM). In addition, we can rewrite the above model into a direct approximation/ interpolation of the input-output mapping
\begin{equation}\label{mapping}
    y:\Gamma\rightarrow \mathbb{R},\quad \mu\mapsto y(\mu).
\end{equation}

Since the full-order matrices $M_1$, $M_2$, $M_3$, $K(\mu)$, $\overline{\widehat{U}}(\mu)$ and $C(\mu)$ are must be accessed, it is very expensive to calculate $y(\mu)$ for a given $\mu$ in FOM. The key question is how to compute $y(\mu)$ without accessing the above internal representations.

 We define the spaces of the state, control and adjoint variables as $X_h^{\mathcal{N}}(\Omega)$, $Y_h^{\mathcal{N}}(\Omega)$ and $Z_h^{\mathcal{N}}(\Omega)$, respectively. Given a positive integer $N_{\text{max}}$, an associated sequence of approximation spaces: for $N = 1,\cdots,N_{\text{max}}$, $X_h^N(\Omega)$ is a $N$-dimensional subspace of $X_h^{\mathcal{N}}(\Omega)$. We further suppose that $X_h^1(\Omega)\subset X_h^2(\Omega) \subset \cdots \subset X_h^{N_{\text{max}}}(\Omega) \subset X_h^{\mathcal{N}}(\Omega)$. Similarly, for the control variable $f$ and adjoint variable $\lambda$, we similarly define $Y_h^1(\Omega)\subset Y_h^2(\Omega) \subset \cdots \subset Y_h^{N_{\text{max}}}(\Omega) \subset Y_h^{\mathcal{N}}(\Omega)$ and $Z_h^1(\Omega)\subset Z_h^2(\Omega) \subset \cdots \subset Z_h^{N_{\text{max}}}(\Omega) \subset Z_h^{\mathcal{N}}(\Omega)$.

 Applying Galerkin projection onto the low-dimensional subspace $X_h^N(\Omega)\times Y_h^N(\Omega)\times Z_h^N(\Omega)$ yields the optimality system: given $\mu\in\Gamma$, find $(u_h^N,f_h^N,\lambda_h^N)\in X_h^N(\Omega)\times Y_h^N(\Omega)\times Z_h^N(\Omega)$ such that
\begin{equation}\label{low_Gal}
\left\{
\begin{aligned}a(u_{h}^{N},\widetilde{u}_{h}^{\prime};\mu)&=(f_{h}^{N},\widetilde{u}_{h}^{\prime};\mu),\quad\forall\widetilde{u}_{h}^{\prime}\in X_{h}^{N}(\Omega),\\
a(\lambda_{h}^{N},\widetilde{\lambda}_{h}^{\prime};\mu)&=-(u_{h}^{N}-\hat{u},\widetilde{\lambda}_{h}^{\prime};\mu),\quad\forall\widetilde{\lambda}_{h}^{\prime}\in Z_{h}^{N}(\Omega),\\
2\beta(f_{h}^{N},\widetilde{f}_{h}^{\prime};\mu)&=(\widetilde{f}_{h}^{\prime},\lambda_{h}^{N};\mu),\quad\forall\widetilde{f}_{h}^{\prime}\in Y_{h}^{N}(\Omega).
\end{aligned}
\right.
\end{equation}
We then choose the set $\left\{\delta_i\right\}_{i=1}^{2N}$ as basis functions for both $X_h^N(\Omega)$ and $Z_h^N(\Omega)$, and the set $\left\{\xi_j\right\}_{j=1}^N$ as basis functions for $Y_h^N(\Omega)$. Because basis function $\delta_i$ and $\xi_j$ belong to the FEM space $V_0^h(\Omega)$ and $\mathcal{M}_h(\Omega)$, they can be expressed as 
$$
\delta_i=\sum_{k=1}^{N_h}W_{i,k}\psi_k,\; i=1,\cdots,2N,
$$
$$
\xi_j=\sum_{s=1}^{N_e}V_{j,s}\phi_s,\; j=1,\cdots,N.
$$

Similar to the derivation in Section \ref{section_sto_form}, the matrix form of \eqref{low_Gal} is given by
\begin{equation}
\left\{
\begin{aligned}&\sum_{q=1}^{Q_{a}}Q_{a}^{q}(\mu)(W^{T}K^{q}W)U_g(\mu)=W^{T}M_{2}^{T}VF_g(\mu),\\&\sum_{q^{\prime}=1}^{Q_{a}}Q_{a}^{q^{\prime}}(\mu)(W^{T}K^{q^{\prime}}W)\Lambda_g(\mu)+W^{T}M_{3}WU_g(\mu)=\sum_{p=1}^{Q_{u}}\widehat{u}_{p}(\mu)(W^{T}\overline{\widehat{U}}_p),\\&2\beta V^{T}M_{1}V^{T}F_g(\mu)=V^{T}M_{2}W\Lambda_g(\mu).
\end{aligned}\right.
\end{equation}
With the following  notations,
$$\begin{aligned}
& K_g(\mu)=\sum_{q=1}^{Q_a} Q_a^q(\mu) (W^TK^qW),\\
&\quad\overline{\widehat{U}}_g=\sum_{p=1}^{Q_u} \widehat{u}_p(\mu)(W^T\overline{\widehat{U}}_p),\\
&M_{1,g}=V^TM_1V,M_{2,g}=W^TM_2V,M_{3,g}=W^TM_1W,\\
\end{aligned}$$
we can get the algebraic formulation of \eqref{low_Gal}
\begin{equation}\label{ROM}
\begin{aligned}
\underbrace{\begin{bmatrix}
2\beta M_{1,g}&0&-M_{2,g}^T\\0&M_{3,g}&K_g^T(\mu)\\-M_{2,g}&K_g(\mu)&0\end{bmatrix}}_{\mathscr{A}_g(\mu)\in\mathbb{R}^{5N\times5N}}\begin{bmatrix}
F_g(\mu)\\U_g(\mu)\\\Lambda_g(\mu)\end{bmatrix}=\begin{bmatrix}0\\\overline{\widehat{U}_g}(\mu)\\d_g\end{bmatrix}.
\end{aligned}
\end{equation}
The dimension of this linear system is obviously lower than that of \eqref{FOM}. 

By defining $\widehat{x}(\mu):=\begin{bmatrix}
F_g(\mu),U_g(\mu),\Lambda_g(\mu)
\end{bmatrix}^T$, we can rewrite the state-interest function \eqref{int_dis} as 
\begin{equation}\label{rom_int}
    \widehat{y}(\mu)=C(\mu)\begin{bmatrix}
    V^T&0&0\\0&W^T&0\\0&0&W^T
    \end{bmatrix}\widehat{x}(\mu):= C_g(\mu)\widehat{x}(\mu),
\end{equation}
and thus obtain the input-output mapping of the system \eqref{ROM}
\begin{equation}
    \widehat{y}:\Gamma\rightarrow\mathbb{R},\quad \mu\mapsto\widehat{y}(\mu).
\end{equation}
We refer to the system described by \eqref{ROM} and \eqref{rom_int} as the reduced-order model (ROM). In contrast to the full-order model (FOM), the dimension $N$ of the ROM is modest, making the computation of $\widehat{y}(\mu)$ significantly less expensive than that of $y(\mu)$ for any parameter $\mu \in \Gamma$.

Our goal is to construct a data-driven reduced-order model (DDROM)
\begin{equation}\label{DDROM}
\begin{aligned}
\underbrace{\begin{bmatrix}
2\beta M_{1,r}&0&-M_{2,r}^T\\0&M_{3,r}&K_r^T(\mu)\\-M_{2,r}&K_r(\mu)&0\end{bmatrix}}_{\mathscr{A}_r(\mu)\in\mathbb{R}^{r\times r}}\widehat{x}(\mu)=\begin{bmatrix}0\\\overline{\widehat{U}}_r(\mu)\\d_r\end{bmatrix},
\end{aligned}
\end{equation}
and 
\begin{equation}\label{ddrom_int}
    \widehat{y}(\mu)=C_r(\mu)\widehat{x}(\mu),
\end{equation}  
whose output $\widehat{y}(\mu)$ is both significantly cheaper to evaluate than $y(\mu)$ and consistently close to it for all $\mu\in\Gamma$. Although the motivation comes from FOM \eqref{FOM}-\eqref{int_dis}, the approximation framework developed below requires only parameter-to-output mapping \eqref{mapping} but not the full-order matrices or state $x$ in the FOM. Thus, we work with a nonintrusive, data-driven formulation based on parameter-output pairs.

To construct a data-driven reduced-order model (DDROM) by measuring the distance between $y$ and $\widehat{y}$, the $\mathcal{L}_2$-optimal reduced-order modeling introduced in this section minimizes the $\mathcal{L}_2$ error.
\begin{equation}\label{L_2 error}
\|y-\widehat{y}\|_{\mathcal{L}_2}=
\left(\int_{\Gamma}\|y(\mu)-\widehat{y}(\mu)\|_{\mathcal{F}}^2\operatorname{d\mu}\right)^{1/2},
\end{equation}
where $\|\cdot\|_{\mathcal{F}}$ is the Frobenius norm. All in all, given the input-to-output mapping in \eqref{mapping}, our goal is to find a DDROM consisting of \eqref{DDROM} and \eqref{ddrom_int} that minimizes the output $\mathcal{L}_2$ error \eqref{L_2 error}.

\subsection{Parameter-separable form of the full order model}
In this section, we express the full order model (FOM) in a parameter-separable form (PSF), where the system matrices and right-hand sides are written as affine linear combinations of parameter-independent components. 

Under the affine assumptions introduced in Section \ref{section_sto_form}, the FOM system \eqref{FOM} and \eqref{int_dis} can be rewritten as follows:
\begin{equation}\label{FOM_EXT}
\begin{aligned}
\left\{
\begin{bmatrix}
2\beta M_{1}&0&-M_{2}^T\\0&M_{3}&0\\-M_{2}&0&0\end{bmatrix}+\sum_{q=1}^{Q_a}Q_a^q(\mu)\begin{bmatrix}
0&0&0\\0&0&(K^q)^T\\0&K^q&0\end{bmatrix}\right\}
x(\mu)\\
=\begin{bmatrix}0\\0\\d\end{bmatrix}+\sum_{p=1}^{Q_u}\widehat{u_p}(\mu)\begin{bmatrix}0\\\widehat{U}_p\\0\end{bmatrix},
\end{aligned}
\end{equation}
and
\begin{equation}
    \label{fom_ext_int}
    y(\mu)=\sum_{k=1}^{Q_l}Q_l^k(\mu)C_kx(\mu).
\end{equation}
This parameter-separable form is preserved in the classical Galerkin projection ROM \eqref{ROM},
\begin{equation}\label{ROM_}
\begin{aligned}
\left\{
\begin{bmatrix}
2\beta V^TM_1V&0&-V^TM_2^TW\\0&W^TM_3W&0\\-W^TM_2V&0&0\end{bmatrix}+\sum_{q=1}^{Q_a}Q_a^q(\mu)\begin{bmatrix}
0&0&0\\0&0&W^T(K^q)^TW\\0&W^TK^qW&0\end{bmatrix}\right\}
\widehat{x}(\mu)\\
=\begin{bmatrix}0\\0\\d_g\end{bmatrix}+\sum_{p=1}^{Q_u}\widehat{u_p}(\mu)\begin{bmatrix}0\\W^T\widehat{U}_p\\0\end{bmatrix},
\end{aligned}
\end{equation}
where $Q_a$, $Q_u$, $Q_l$ are small positive integers, $Q_a^q(\mu)$, $Q_l^k(\mu)$, $\widehat{u}_p(\mu): \Gamma\to\mathbb{R}$ are given functions that are easy to evaluate. In order to construct DDROM, we further simplify the ROM by defining the following matrices:  \[\widetilde{M}_1=V^TM_1V, \quad \widetilde{M}_2=W^TM_2V, \quad \widetilde{M}_3=W^TM_3W,\quad
\widetilde{K}^q=W^TK^qW, \quad \widetilde{\widehat{U}}
_p=W^T\widehat{U}_p.\] 
Then we have
\begin{equation}\label{ROM_EXT}
\begin{aligned}
\left\{
\begin{bmatrix}
2\beta \widetilde{M}_1&0&-\widetilde{M}_2^T\\0&\widetilde{M}_3&0\\-\widetilde{M}_2&0&0\end{bmatrix}+\sum_{q=1}^{Q_{a}}Q_a^q(\mu)\begin{bmatrix}
0&0&0\\0&0&(\widetilde{K}^q
)^T\\0&\widetilde{K}^q&0\end{bmatrix}\right\}
\widehat{x}(\mu)\\
=\begin{bmatrix}0\\0\\ \widetilde{d}\end{bmatrix}+\sum_{p=1}^{Q_{u}}\widehat{u}_p(\mu)\begin{bmatrix}0\\ 
\widetilde{\widehat{U}}_p\\0\end{bmatrix},
\end{aligned}
\end{equation}
and
\begin{equation}\label{rom_ext_int}
    \widehat{y}(\mu)=\sum_{k=1}^{Q_l}Q_l^k(\mu)C_k\begin{bmatrix}
        V^T&0&0\\0&W^T&0\\0&0&W^T
    \end{bmatrix}\widehat{x}(\mu):=\sum_{k=1}^{Q_l}Q_l^k(\mu)\widetilde{C}_k\widehat{x}(\mu).
\end{equation}

Inspired by this formulation, in the $\mathcal{L}_2$-optimal DDROM setting, we search for a structured DDROM with parameter-separable form as follows,
\begin{equation}\label{DDROM_EXT}
\begin{aligned}
\left\{
\begin{bmatrix}
2\beta \widetilde{M}_1&0&-\widetilde{M}_2^T\\0&\widetilde{M}_3&0\\-\widetilde{M}_2&0&0\end{bmatrix}+\sum_{q=1}^{Q_{\widetilde{a}}}\widetilde{Q}_a^q(\mu)\begin{bmatrix}
0&0&0\\0&0&(\widetilde{K}^q
)^T\\0&\widetilde{K}^q&0\end{bmatrix}\right\}
\widehat{x}(\mu)\\
=\begin{bmatrix}0\\0\\ \widetilde{d}\end{bmatrix}+\sum_{p=1}^{Q_{\widetilde{u}}}\widetilde{\widehat{u}}_p(\mu)\begin{bmatrix}0\\ 
\widetilde{\widehat{U}}_p\\0\end{bmatrix},
\end{aligned}
\end{equation}
and
\begin{equation}\label{ddrom_ext_int}
    \widehat{y}(\mu)=\sum_{k=1}^{Q_{\widetilde{l}}}\widetilde{Q}_l^k(\mu)\widetilde{C}_k\widehat{x}(\mu),
\end{equation}
where $Q_{\widetilde{a}}$, $Q_{\widetilde{u}}$, $Q_{\widetilde{l}}$ are small positive integers. $\widetilde{Q}_a^q(\mu)$, $\widetilde{Q}_l^k(\mu)$, and $\widetilde{u}_p(\mu): \Gamma \to \mathbb{R}$ are given functions, $\widetilde{M}_i~(i=1,2,3)$,  $\widetilde{K}^q~(q=1,\cdots,Q_{\widetilde{a}})$, $\widetilde{d}$, $\widetilde{\widehat{U}}_p~(p=1,\cdots,Q_{\widetilde{u}})$ and $\widetilde{C}_k~(k=1,\cdots,Q_{\widetilde{l}})$ are the DDROM matrices to be determined by minimizing the $\mathcal{L}_2$ error \eqref{L_2 error}.


Reviewing the initial idea of this method, it is interested in approximating a input-to-output mapping \eqref{mapping} by a DDROM \eqref{DDROM}. The most important thing is that the framework does not need the full-order matrices and the full-order state $x$, but we only need the output $y$.

For the following explanation, a assumption is given first, and this assumption is easy to establish in most problems.

\begin{assumption}\label{ass_scalar}
Let $(\Gamma,\Sigma,\mathbb{P})$ be a measure space, the functions  $\widetilde{Q}_a^q(\mu)$, $\widetilde{Q}_l^k(\mu)$, $\widetilde{u}_p(\mu)$ are measurable, and satisfy \begin{equation}
    \label{scalar_funs}
    \int_{\Gamma}\left\{\frac{(1+\sum_{p=1}^{Q_{\widetilde{u}}}|\widetilde{\widehat{u}}_p(\mu)|)\sum_{k=1}^{Q_{\widetilde{l}}}|\widetilde{Q}^k_l(\mu)|}{1+\sum_{q=1}^{Q_{\widetilde{a}}}|\widetilde{Q}^q_a(\mu)|}\right\}^2\mathbf{d}\mathbb{P}(\mu)<\infty.
\end{equation}
\end{assumption}

Next, we introduce some new notations for the structured DDROM \eqref{DDROM_EXT} before defining the set of allowable DDROM matrices based on Assumption \ref{ass_scalar}.
$$
\begin{aligned}
\widetilde{A}_0:=
\begin{bmatrix}
2\beta \widetilde{M}_1&0&-\widetilde{M}_2^T\\0&\widetilde{M}_3&0\\-\widetilde{M}_2&0&0\end{bmatrix},
\widetilde{A}_q:=\begin{bmatrix}
0&0&0\\0&0&(\widetilde{K}^q
)^T\\0&\widetilde{K}^q&0\end{bmatrix}, \widetilde{B}_0:=\begin{bmatrix}
    0,0,\widetilde{d},
\end{bmatrix}^T,\widetilde{B}_p:=\begin{bmatrix}
    0,\widetilde{\widehat{U}}_p,0
    \end{bmatrix}^T,
\end{aligned}
$$
for $q=1,\cdots,Q_{\widetilde{a}}$ and $p=1,\cdots,Q_{\widetilde{u}}.$
\begin{definition}\label{def_set_R}
Let $R$ be the set of all tuples of DDROM matrices $$(\widetilde{A}_0,\widetilde{A}_1,\ldots,\widetilde{A}_{Q_{\widetilde{a}}},\widetilde{B}_0,\widetilde{B}_1,\ldots,\widetilde{B}_{Q_{\widetilde{u}}},\widetilde{C}_{1},\ldots,\widetilde{C}_{Q_{\widetilde{l}}})=:(\widetilde{A}_{p},\widetilde{B}_{q},\widetilde{C}_{k}),$$ 
for $p=0,1,\cdots Q_{\widetilde{a}}$, $q=0,1,\cdots Q_{\widetilde{u}}$ and $k=1,2,\cdots Q_{\widetilde{l}}$.
Then we define the set $\mathcal{R}$ of allowable DDROM matrices as
\begin{align}\label{set R}
    \mathcal{R}=\left\{(\widetilde{A}_{p},\widetilde{B}_{q},\widetilde{C}_{k})\in R:\underset{\mu\in\Gamma}{\operatorname*{ess\;sup}}\left\|(1+\widetilde{Q}^q_a(\mu))(\widetilde{A}_0+\sum_{q=1}^{Q_{\widetilde{a}}}\widetilde{Q}_a^q(\mu)\widetilde{A}_q)^{-1}\right\|_{\mathcal{F}}<\infty \right\}.
\end{align}
\end{definition}

It is important to establish that the set $\mathcal{R}$ is open and that it forms a set of feasible DDROMs for the following analysis, the lemma gives the corresponding conclusions.
\begin{lemma}\label{lem_open}
    The set $\mathcal{R}$ \eqref{set R} in Definition \ref{def_set_R} is open. Moreover, for all $\widehat{y}$ defined by a DDROM $(\widetilde{A}_{p},\widetilde{B}_{q},\widetilde{C}_{k}) \in \mathcal{R}$, we have that $\widehat{y}$ is square-integrable.
\end{lemma}
The proof of this lemma follows the same lines as the proof of the analogous result in \cite{L2} and is omitted for brevity.

\subsection{The gradients of the $\mathcal{L}_2$ approximation error}
Based on the previous conclusions, in this section, we will derive the gradients of the $\mathcal{L}_2$ cost function \eqref{L_2 error} with respect to DDROM matrices $\widetilde{A}_{p}$, $\widetilde{B}_{q}$,
$\widetilde{C}_{k}$. These gradient formulas form the basis of our $\mathcal{L}_2$-optimal reduced-order modeling algorithm, which is detailed in next subsection.

We need to find the DDROM belonging to $\mathcal{R}$ from Definition \ref{def_set_R}, because the Lemma \ref{lem_open} ensures that the square $\mathcal{L}_2$ error is clearly defined and differentiable on $\mathcal{R}$. Therefore, We don't continue to find the gradients for $\mathcal{L}_2$ \eqref{L_2 error}, but for the following $\mathcal{L}_2$-optimization problem
\begin{equation}
\min\limits_{(\widetilde{A}_{p},\widetilde{B}_{q},\widetilde{C}_{k}) \in \mathcal{R}} \quad\mathcal{J}(\widetilde{A}_{p},\widetilde{B}_{q},\widetilde{C}_{k})=\|y-\widehat{y}\|_{\mathcal{L}_2(\Gamma, \mathbb{P})}^2.
\end{equation}

Before presenting the theorem on the explicit expression of the gradient, we first recall some preliminary results. Firstly, we employ the reduced-order dual state $\widehat{x}_d(\mu)$, which satisfies the reduced-order dual state equation 
\begin{flalign}\label{dual-state eq}
\begin{aligned}
\left\{
\begin{bmatrix}
2\beta \widetilde{M}^T_1&0&-\widetilde{M}_2^T\\0&\widetilde{M}^T_3&0\\-\widetilde{M}_2&0&0\end{bmatrix}+\sum_{q=1}^{Q_{\widetilde{a}}}\widetilde{Q}_a^q(\mu)\begin{bmatrix}
0&0&0\\0&0&\widetilde{K}^q
\\0&(\widetilde{K}^q)^T&0\end{bmatrix}\right\}
\widehat{x}_d(\mu)
=\sum_{k=1}^{Q_{\widetilde{l}}}\widetilde{Q}_l^k(\mu)(\widetilde{C}^k)^T.
\end{aligned}
\end{flalign}

For the derivation of the gradient formula in Theorem \ref{gradient_the}, we first recall the definition of the Fréchet derivative\cite{antil2018frontiers}.
\begin{definition}
Let $H$ be a Hilbert space endowed with inner product $\langle \cdot, \cdot \rangle$, and let $U \subset H$ be an open set. 

Considering a function $f: U \to \mathbb{R}$, we say that $f$ is Fréchet differentiable at $x \in U$ if there exists a bounded linear functional $Df(x) \in H^*$ such that
\begin{equation*}
f(x + h) = f(x) + Df(x)(h) + o(\|h\|), \quad \text{as } \|h\| \to 0,
\end{equation*}
where $o(\|h\|)$ satisfies
\begin{equation*}
\lim_{\|h\| \to 0} \frac{o(\|h\|)}{\|h\|} = 0.
\end{equation*}
By the Riesz representation theorem, there exists a unique element $\nabla f(x) \in H$ such that
\begin{equation*}
Df(x)(h) = \langle \nabla f(x), h \rangle, \quad \forall h \in H.
\end{equation*}
The element $\nabla f(x)$ is called the gradient of $f$ at $x$. Consequently, the first-order expansion can be written as
\begin{equation*}
f(x + h) = f(x) + \langle \nabla f(x), h \rangle + o(\|h\|).
\end{equation*}
\end{definition}

\begin{theorem}\label{gradient_the}
Let$(\widetilde{A}_{p},\widetilde{B}_{q},\widetilde{C}_{k})$ be a tuple of DDROM matrices belonging to $\mathcal{R}$ which is defined in Definition \ref{def_set_R}. Then the gradients of $\mathcal{J}$ with respect to the DDROM matrices are
\begin{align*}
\nabla_{\widetilde{A}_{0}}\mathcal{J} &=2\int_{\Gamma}\widehat{x}_{d}(\mu)\left[y(\mu)-\widehat{y}(\mu)\right]\widehat{x}^{T}(\mu)\mathrm{d}\mathbb{P}(\mu), \\
\nabla_{\widetilde{A}_{q}}\mathcal{J} &=2\int_{\Gamma}\widetilde{Q}_a^q(\mu)\widehat{x}_{d}(\mu)\left[y(\mu)-\widehat{y}(\mu)\right]\widehat{x}^{T}(\mu)\mathrm{d}\mathbb{P}(\mu),\quad&q=1,2,\ldots,Q_{\widetilde{a}}, \\
\nabla_{\widetilde{B}_{0}}\mathcal{J} &=2\int_{\Gamma}\widehat{x}_d(\mu)\left[\widehat{y}(\mu)-y(\mu)\right]\mathrm{d}\mathbb{P}(\mu), \\
\nabla_{\widetilde{B}_{p}}\mathcal{J} &=2\int_{\Gamma}\widetilde{Q}_u^p(\mu)\widehat{x}_d(\mu)\left[\widehat{y}(\mu)-y(\mu)\right]\mathrm{d}\mathbb{P}(\mu),&p=1,2,\ldots,Q_{\widetilde{u}}, \\
\nabla_{\widetilde{C}_{k}}\mathcal{J} &=2\int_{\Gamma}\widetilde{Q}_l^k(\mu)\left[\widehat{y}(\mu)-y(\mu)\right]\widehat{x}^T(\mu)\mathrm{d}\mathbb{P}(\mu), &k=1,2,\ldots,Q_{\widetilde{l}}. 
\end{align*}
\end{theorem}

\begin{proof}
To simplify the symbolic representation in the process of proof, we firstly will write \eqref{DDROM_EXT} and \eqref{ddrom_ext_int} as 
\begin{equation}\label{simple_ddrom}\begin{aligned}
    \mathbf{A}(\mu)\widehat{x}(\mu)&=\mathbf{B}(\mu),\\
    \widehat{y}(\mu)&=\mathbf{C}(\mu)\widehat{x}(\mu),
    \end{aligned}
\end{equation}
 where $$\begin{aligned}
&\mathbf{A}(\mu)=\widetilde{A}_0+\sum_{q=1}^{Q_{\widetilde{a}}}\widetilde{Q}_a^q(\mu)\widetilde{A}_q:=\sum_{q=0}^{Q_{\widetilde{a}}}\widetilde{Q}_a^q(\mu)\widetilde{A}_q,\\
&\mathbf{B}(\mu)=\widetilde{B}_0+\sum_{p=1}^{Q_{\widetilde{u}}}\widetilde{Q}_u^p(\mu)\widetilde{B}_p:=\sum_{p=0}^{Q_{\widetilde{u}}}\widetilde{Q}_u^p(\mu)\widetilde{B}_p,\\
&\mathbf{C}(\mu)=\sum_{k=1}^{Q_{\widetilde{l}}}\widetilde{Q}_l^k(\mu)\widetilde{C}_k,
\end{aligned}$$
and $\widetilde{Q}_a^0=\widetilde{Q}_u^0=1$. Substituting $\widehat{y}(\mu)=\mathbf{C}(\mu)\mathbf{A}^{-1}(\mu)\mathbf{B}(\mu)$ from \eqref{simple_ddrom} into the objective function, and using the symmetry of the inner product  $\langle y, \widehat y\rangle _{\mathcal{L}_{2}(\Gamma,\mathbb{P})} = \langle \widehat y, y\rangle_{\mathcal{L}_{2}(\Gamma,\mathbb{P})}$, we obtain
 \begin{equation}\label{rewrite obj_func}
 \begin{aligned}
 \mathcal{J}=&\|y-\widehat{y}\|_{\mathcal{L}_{2}(\Gamma,\mathbb{P})}^{2}\\
 =& \|y\|_{\mathcal{L}_{2}(\Gamma,\mathbb{P})}^{2}-2\langle y,\widehat{y}\rangle_{\mathcal{L}_{2}(\Gamma,\mathbb{P})}+\|\widehat{y}\|_{\mathcal{L}_{2}(\Gamma,\mathbb{P})}^{2} \\
 =& \|y\|_{\mathcal{L}_{2}(\Gamma,\mathbb{P})}^{2}-2\int_{\Gamma}\mathrm{tr}\left(y^T(\mu)\mathbf{C}(\mu)\mathbf{A}^{-1}(\mu)\mathbf{B}(\mu)\right)\mathrm{d}\mathbb{P}(\mu)\\
+&\int_{\Gamma}\mathrm{tr}\left(\mathbf{B}^{T}(\mu)\mathbf{A}^{-T}(\mu)\mathbf{C}^{T}(\mu)\mathbf{C}(\mu)\mathbf{A}^{-1}(\mu)\mathbf{B}(\mu)\right)\mathrm{d}\mathbb{P}(\mu).
 \end{aligned}
 \end{equation}
 The first part of the \eqref{rewrite obj_func} does not affect the solution of gradient, so it will not be considered in the future. Record the remaining two items $\mathcal{J}_{2}$ and $\mathcal{J}_{3}$ as
$$
\mathcal{J}_{2}=-2\int_{\Gamma}\mathrm{tr}\left(y^T(\mu)\mathbf{C}(\mu)\mathbf{A}^{-1}(\mu)\mathbf{B}(\mu)\right)\mathrm{d}\mathbb{P}(\mu),
$$
$$
\mathcal{J}_{3}=\int_{\Gamma}\mathrm{tr}\left(\mathbf{B}^{T}(\mu)\mathbf{A}^{-T}(\mu)\mathbf{C}^{T}(\mu)\mathbf{C}(\mu)\mathbf{A}^{-1}(\mu)\mathbf{B}(\mu)\right)\mathrm{d}\mathbb{P}(\mu),
$$
respectively. Next, we compute the gradients of the above two terms about $\widetilde{A}_q$, $\widetilde{B}_p$ and $\widetilde{C}_k$ respectively. Taking $\nabla_{\widetilde{A}_q} \mathcal{J}_2$ as an example, we describe the procedure in detail. To do so, we consider a small perturbation $\Delta \widetilde{A}_q$ and evaluate $\mathcal{J}_2(\widetilde{A}_q + \Delta \widetilde{A}_q)$. In addition, we need to use the property in \eqref{set R}, and apply the Neumann series formula.

 $$\begin{aligned}
\mathcal{J}_{2}\left(\widetilde{A}_{q}+\Delta\widetilde{A}_{q}\right)
&=-2\int_{\Gamma}\mathrm{tr}\left(y^T(\mu)\mathbf{C}(\mu)(\mathbf{A}(\mu)+\widetilde{Q}_a^q(\mu)\Delta\widetilde{A}_{q})^{-1}\mathbf{B}(\mu)\right)\mathrm{d}\mathbb{P}(\mu)\\
&=-2\int_{\Gamma}\mathrm{tr}\left(y^T(\mu)\mathbf{C}(\mu)\left(I+\widetilde{Q}_a^q(\mu)\mathbf{A}^{-1}(\mu)\Delta\widetilde{A}_{q}\right)^{-1}\mathbf{A}^{-1}(\mu)\mathbf{B}(\mu)\right)\mathrm{d}\mathbb{P}(\mu)\\
&=-2\int_{\Gamma}\mathrm{tr}\left(y^T(\mu)\mathbf{C}(\mu)\mathbf{A}^{-1}(\mu)\mathbf{B}(\mu)\right)\mathrm{d}\mathbb{P}(\mu)\\
&+2\int_{\Gamma}\mathrm{tr}\left(\widetilde{Q}_a^q(\mu)y^T(\mu)\mathbf{C}(\mu)\mathbf{A}^{-1}(\mu)\Delta\widetilde{A}_q\mathbf{A}^{-1}(\mu)\mathbf{B}(\mu)\right)\mathrm{d}\mathbb{P}(\mu)\\
&-2\int_{\Gamma}\mathrm{tr}\left(y^T(\mu)\mathbf{C}(\mu)\sum_{m=2}^\infty\left(-\widetilde{Q}_a^q(\mu)\mathbf{A}^{-1}(\mu)\Delta\widetilde{A}_q\right)^m\mathbf{A}^{-1}(\mu)\mathbf{B}(\mu)\right)\mathrm{d}\mathbb{P}(\mu)\\
&=\mathcal{J}_2(\widetilde{A}_q)+\left\langle2\int_{\Gamma}\widetilde{Q}_a^q(\mu)\mathbf{A}^{-T}(\mu)\mathbf{C}^T(\mu)y(\mu)\mathbf{B}^T(\mu)\mathbf{A}^{-T}(\mu)\mathrm{d}\mathbb{P}(\mu),\Delta\widetilde{A}_q\right\rangle_{\mathcal{F}}\\
&-2\sum_{m=2}^{\infty}\int_{\Gamma}\mathrm{tr}\left(y^T(\mu)\mathbf{C}(\mu)\left(-\widetilde{Q}_a^q(\mu)\mathbf{A}^{-1}(\mu)\Delta\widetilde{A}_q\right)^m\mathbf{A}^{-1}(\mu)\mathbf{B}(\mu)\right)\mathrm{d}\mathbb{P}(\mu).
\end{aligned}$$
To derive $\nabla_{\widetilde{A}_q}\mathcal{J}_2$ from the last equation, it is necessary to show that the second term is bounded and that the third term is a low-order infinitesimal of $\Delta\widetilde{A}_q$. In demonstrating that the second term is bounded, we will rely on the property stated in \eqref{set R}.
$$\begin{aligned}
&\left\|\int_{\Gamma}\widetilde{Q}_a^q(\mu)\mathbf{A}^{-T}(\mu)\mathbf{C}^T(\mu)y(\mu)\mathbf{B}^T(\mu)\mathbf{A}^{-T}(\mu)\mathrm{d}\mathbb{P}(\mu)\right\|_{\mathcal{F}} \\
&\leqslant\int_{\Gamma}|\widetilde{Q}_a^q(\mu)|\left\|\mathbf{A}^{-1}(\mu)\right\|_{\mathcal{F}}\left\|\mathbf{C}(\mu)\right\|_{\mathcal{F}}\left\|\mathbf{A}^{-1}(\mu)\right\|_{\mathcal{F}}\left\|\mathbf{B}(\mu)\right\|_{\mathcal{F}}\left\|y(\mu)\right\|_{\mathcal{F}}\mathrm{d}\mathbb{P}(\mu) \\
&\leqslant\|\widetilde{Q}^q_a(\mu)(\widetilde{A}_0+\sum_{q=1}^{Q_{\widetilde{a}}}\widetilde{Q}_a^q(\mu)\widetilde{A}_q)^{-1}\|_{\mathcal{L}_{\infty}(\Gamma,\mathbb{P})}\int_{\Gamma}\left\|\mathbf{C}(\mu)\right\|_{\mathcal{F}}\left\|\mathbf{A}^{-1}(\mu)\right\|_{\mathcal{F}}\left\|\mathbf{B}(\mu)\right\|_{\mathcal{F}}\left\|y(\mu)\right\|_{\mathcal{F}}\mathrm{d}\mathbb{P}(\mu)\\
&<\infty.
\end{aligned}$$
Next, we show that the last term is of low-order infinitesimal.
$$\begin{aligned}
&\left|\sum_{m=2}^{\infty}\int_{\Gamma}\mathrm{tr}\left(y^T(\mu)\mathbf{C}(\mu)\left(-\widetilde{Q}_a^q(\mu)\mathbf{A}^{-1}(\mu)\Delta\widetilde{A}_q\right)^m\mathbf{A}^{-1}(\mu)\mathbf{B}(\mu)\right)\mathrm{d}\mathbb{P}(\mu)\right| \\
&\leqslant\|y(\mu)\|_{\mathcal{L}_{2}(\Gamma,\mathbb{P})}  \sum_{m=2}^{\infty}\left\|\mathbf{C}(\mu)\left(-\widetilde{Q}_a^q(\mu)\mathbf{A}^{-1}(\mu)\Delta\widetilde{A}_q\right)^{m}\mathbf{A}^{-1}(\mu)\mathbf{B}(\mu)\right\|_{\mathcal{L}_{2}(\Gamma,\mathbb{P})} \\
&\leqslant\|y(\mu)\|_{\mathcal{L}_{2}(\Gamma,\mathbb{P})}\sum_{m=2}^{\infty}\left\|\widetilde{Q}_a^q(\mu)\mathbf{A}^{-1}(\mu)\right\|_{\mathcal{L}_{\infty}}^{m}\left\|\|\mathbf{C}(\mu)\|_{\mathcal{F}}\|\mathbf{A}^{-1}(\mu)\|_{\mathcal{F}}\|\mathbf{B}(\mu)\|_{\mathcal{F}}\right\|_{\mathcal{L}_{2}(\Gamma,\mathbb{P})}\|\Delta\widetilde{A}_{q}\|_{\mathcal{F}}^{m}\\
&=\|y(\mu)\|_{\mathcal{L}_2(\Gamma, \mathbb{P})}\| \|\mathbf{C}(\mu)\left\|_{\mathcal{F}}\right\| \mathbf{A}^{-1}(\mu)\left\|_{\mathcal{F}}\right\|\mathbf{B}(\mu)\left\|_{\mathcal{F}}\right\|_{\mathcal{L}_2(\Gamma,\mathbb{P})} \frac{\left\|\widetilde{Q}_a^q(\mu)\mathbf{A}^{-1}(\mu)\right\|_{\mathcal{L}_{\infty}}^2\left\|\Delta \widetilde{A}_q\right\|_{\mathcal{F}}^2}{1-\left\|\widetilde{Q}_a^q(\mu)\mathbf{A}^{-1}(\mu)\right\|_{\mathcal{L}_{\infty}}\left\|\Delta\widetilde{A}_q\right\|_{\mathcal{F}}} \\
& =o\left(\left\|\Delta \widetilde{A}_q\right\|_{\mathcal{F}}\right).
\end{aligned}$$
According to the above proof, we can get the expression of $\nabla_{\widetilde{A}_q} \mathcal{J}_2$ as
$$
\nabla_{\widetilde{A}_q} \mathcal{J}_2=2\int_{\Gamma}\widetilde{Q}_a^q(\mu)\mathbf{A}^{-T}(\mu)\mathbf{C}^T(\mu)y(\mu)\mathbf{B}^T(\mu)\mathbf{A}^{-T}(\mu)\mathrm{d}\mathbb{P}(\mu),
$$
especially,
$$
\nabla_{\widetilde{A}_0} \mathcal{J}_2=2\int_{\Gamma}\mathbf{A}^{-T}(\mu)\mathbf{C}^T(\mu)y(\mu)\mathbf{B}^T(\mu)\mathbf{A}^{-T}(\mu)\mathrm{d}\mathbb{P}(\mu).
$$

If we want to get the expression of $\nabla_{\widetilde{B}_p} \mathcal{J}_2$, we need to evaluate $\mathcal{J}_2\left(\widetilde{B}_p+\Delta \widetilde{B}_p\right),$
$$\begin{aligned}
\mathcal{J}_2\left(\widetilde{B}_p+\Delta \widetilde{B}_p\right) 
& =-2 \int_{\Gamma} \operatorname{tr}\left(y^T(\mu) \mathbf{C}(\mu)\mathbf{A}^{-1}(\mu)\left(\mathbf{B}(\mu)+\widetilde{Q}_u^p(\mu) \Delta \widetilde{B}_p\right)\right) \mathrm{d}\mathbb{P}(\mu) \\
&=\mathcal{J}_2(\widetilde{B}_p)-2 \int_{\Gamma} \operatorname{tr}\left(\widetilde{Q}_u^p(\mu) y^T(\mu)\mathbf{C}(\mu) \mathbf{A}^{-1} (\mu)\Delta \widetilde{B}_p\right) \mathrm{d} \mathbb{P}(\mu) \\ &=\mathcal{J}_2(\widetilde{B}_p)-2\left\langle\int_{\Gamma} \widetilde{Q}_p(\mu) \mathbf{A}^{-T}(\mu) \mathbf{C}^T(\mu) y(\mu) \mathrm{d} \mathbb{P}(\mu)\Delta \widetilde{B}_p\right\rangle_{\mathcal{F}}.
\end{aligned}$$
Then we can get 
$$
\nabla_{\widetilde{B}_p} \mathcal{J}_2=-2 \int_{\Gamma} \widetilde{Q}_u^p(\mu) \mathbf{A}^{-T} (\mu)\mathbf{C}^T(\mu) y(\mu) \mathrm{~d} \mathbb{P}(\mu),
$$
especially,
$$
\nabla_{\widetilde{B}_0} \mathcal{J}_2=-2 \int_{\Gamma} \mathbf{A}^{-T} (\mu)\mathbf{C}^T(\mu) y(\mu) \mathrm{~d} \mathbb{P}(\mu),
$$
Similarly, we can also get 
$$
\nabla_{\widetilde{C}_k} \mathcal{J}_2=-2 \int_{\Gamma} \widetilde{Q}_l^k(\mu)y(\mu) \mathbf{B}^T(\mu)\mathbf{A}^{-T}(\mu) \mathrm{~d} \mathbb{P}(\mu).
$$

The preceding derivation gives the gradient of the second term $\mathcal{J}_2$. Following analogous steps, the gradient of $\mathcal{J}_3$—and subsequently the full gradient of \eqref{rewrite obj_func}—can be obtained without detailed repetition. The resulting gradient is presented below.
$$\begin{aligned}
& \nabla_{\widetilde{A}_a^q} \mathcal{J}=2 \int_{\Gamma} \widetilde{Q}_q(\mu)\mathbf{A}^{-T}(\mu)\mathbf{C}^T(\mu)\left(y(\mu)-\mathbf{C}(\mu) \mathbf{A}^{-1}(\mu)\mathbf{B}(\mu)\right) \mathbf{B}^T(\mu)\mathbf{A}^{-T}(\mu) \mathrm{~d} \mathbb{P}(\mu), \\
& \nabla_{\widetilde{B}_p} \mathcal{J}=2 \int_{\Gamma} \widetilde{Q}_u^p(\mu) \mathbf{A}^{-T}(\mu) \mathbf{C}^T(\mu)\left(\mathbf{C}(\mu) \mathbf{A}^{-1} (\mu)\mathbf{B}(\mu)-y(\mu)\right) \mathrm{d} \mathbb{P}(\mu), \\
& \nabla_{\widetilde{C}_k} \mathcal{J}=2 \int_{\Gamma} \widetilde{Q}^k_l(\mu)\left(\mathbf{C}(\mu) \mathbf{A}^{-1} (\mu)\mathbf{B}(\mu)-y(\mu)\right) \mathbf{B}^T(\mu)\mathbf{A}^{-T}(\mu) \mathrm{~d} \mathbb{P}(\mu).
\end{aligned}$$
Bring in $\widehat{y}(\mu)=\mathbf{C}(\mu)\mathbf{A}^{-1}(\mu)\mathbf{B}(\mu)$ and $\widehat{x}_d(\mu)=\mathbf{A}^{-1}(\mu)\mathbf{C}(\mu)$, we get the conclusion in the Theorem \ref{gradient_the}.
\end{proof}

The gradients derived in Theorem \ref{gradient_the} do not require access to the full-order matrices $M_i$, $K^q$, $\widehat{U}_p$, and $C_k$, nor the full-order state $x(\mu)$; instead, they are computed directly from the output $y(\mu)$ of the FOM. Through data-driven access to these gradient evaluations, we can design various optimization algorithms to build $\mathcal{L}_2$-optimal DDROM for different situations, which only relys on the output $y(\mu)$. We will discuss the detail of the algorithm in the next section.

\subsection{The data-driven, gradient-based algorithm}\label{sec_L2_alg}
In this section, we presents the proposed $\mathcal{L}_2$-optimal reduced-order modeling algorithm in separable parameters form, denoted as $\mathcal{L}_2$-Opt-PSF. The pseudocode is given in Algorithm \ref{alg_L2}.

The proposed algorithm formulates the reduced-order modeling problem as an optimization problem over the system matrix parameters. In contrast to traditional projection-based reduction methods, the optimization is performed directly in the parameter space of the ROM, thereby avoiding explicit dependence on the high-dimensional FOM.

\begin{algorithm}[H]
    \caption{$\mathcal{L}_2$-Opt-PSF}
    \label{alg_L2}
    \renewcommand{\algorithmicrequire}{\textbf{Input:}}
    \renewcommand{\algorithmicensure}{\textbf{Output:}}
    \begin{algorithmic}[1]
        \REQUIRE Input-to-output mapping $y$, initial guess for DDROM $(\widetilde{A}_q,\widetilde{B}_p, \widetilde{C}_k)$, maximum number of iterations $\mathbf{maxit}$, 
    tolerance $\mathbf{tol}>0.$ 
        \STATE  Set $\widehat{y}^{(0)}$ as the output of the initial DDROM.
        \FOR{$i$ in $1,2,\cdots,\mathbf{maxit}$}
            \STATE Compute a new DDROM with output $\widehat{y}^{(i)}$ using a step of a gradient-based optimization method, with the squared $\mathcal{L}_2$ error in \eqref{L_2 error} as the objective function and gradients based on Theorem \ref{gradient_the}.
            \IF{$\frac{\|\widehat{y}^{(i-1)}-\widehat{y}^{(i)}\|_{\mathcal{L}_2(\Gamma,\mathbb{P})}}{\|\widehat{y}^{(i)}\|_{\mathcal{L}_2(\Gamma,\mathbb{P})}}\leq\mathbf{tol}$}
                \STATE Exit the \textbf{for} loop.
            \ENDIF
        \ENDFOR     
        \RETURN the last computed DDROM.
    \ENSURE DDROM $(\widetilde{A}_q,\widetilde{B}_p, \widetilde{C}_k)$.  
    \end{algorithmic}
\end{algorithm}

We give some details about the algorithm. In the initialization stage, suitable initial parameters for the ROM must be specified. In particular, the full-order matrices $M_i$, $K^q$, $d$, $\widehat{U}_p$, and $C_k$ are utilized to construct the initial DDROM parameters $(\widetilde{A}_q, \widetilde{B}_p, \widetilde{C}_k)$. The choice of initialization plays a critical role in both the convergence behavior and the final approximation accuracy. To this end, the RB method is employed to generate the initial model parameters. By projecting the full-order system onto a low-dimensional subspace spanned by carefully selected basis functions, the RB approach effectively captures the dominant system dynamics and provides a high-quality initial guess for the subsequent $\mathcal{L}_2$-based optimization.

Furthermore, the proposed algorithm relies solely on input–output data and does not require access to the system matrices or state variables of the high-fidelity model, rendering it inherently non-intrusive. The stopping criterion is defined based on the relative change between successive outputs, rather than the error with respect to a high-dimensional reference solution. This design further emphasizes the data-driven nature of the proposed framework. Finally, Algorithm \ref{alg_L2} does not impose a specific optimization solver. Since the resulting problem is generally non-convex, the selection of an appropriate optimization algorithm should be guided by the specific structure and properties of the problem under consideration.

\section{Generalized multiscale finite element method}\label{sec_GMsFEM}
In the section \ref{sec_L2_alg}, we give the $\mathcal{L}_2$-optimal reduced-order modeling algorithm. We need to provide an initial value for the algorithm to get the final DDROM. Given the full order matrix in the FOM, the ROM can be obtained by POD method or RB method. When utilizing these projection-based methods to construct ROMs, it is necessary to calculate \eqref{Galerkin} for many training samples to get the snapshot space.

Generally speaking, the number of training samples is very large, and it needs high calculation cost. Besides, the constraint in the optimization problem may be a partial differential equation with multiscale structure. As a local model reduction method, the generalized multiscale finite element method (GMsFEM) can effectively deal with these problems. One of the most important characteristics of GMsFEM is that multiscale basis functions can be calculated in offline stage and can be used repeatedly for models with different source terms, boundary conditions and similar multiscale structural coefficients. In this section, we will briefly outline the local model reduction using GMsFEM.

Before giving the detailed procedures of GMsFEM, we firstly set the scene of the numerical discretization, as demonstrated in Figure \ref{mesh_GMs}. We assume that the computational domain $\Omega$ is partitioned uniformly by a coarse mesh $\mathcal{T}_H$ and a fine mesh $\mathcal{T}_h$, with mesh size $H$ and $h$, respectively. Moreover, $\mathcal{T}_h$ is obtained by refining the coarse mesh $\mathcal{T}_H$. The nodes of the the coarse mesh are denoted by $\left\{N_i\right\}_{i=1}^{\mathcal{N}_c},i=1,\cdots,\mathcal{N}_c$. The neighborhood $\omega_i$ of the node $N_i$ consists of all the coarse mesh elements for which node $N_i$ is a vertex, i.e.,
$$
\omega_i=\cup \left\{K_s\in \mathcal{T}_H|x_i\in\overline{K}_s\right\},
$$
each coarse grid is a unit $K$.
\begin{figure}[H]
    \centering   
\includegraphics[width=0.5\textwidth]{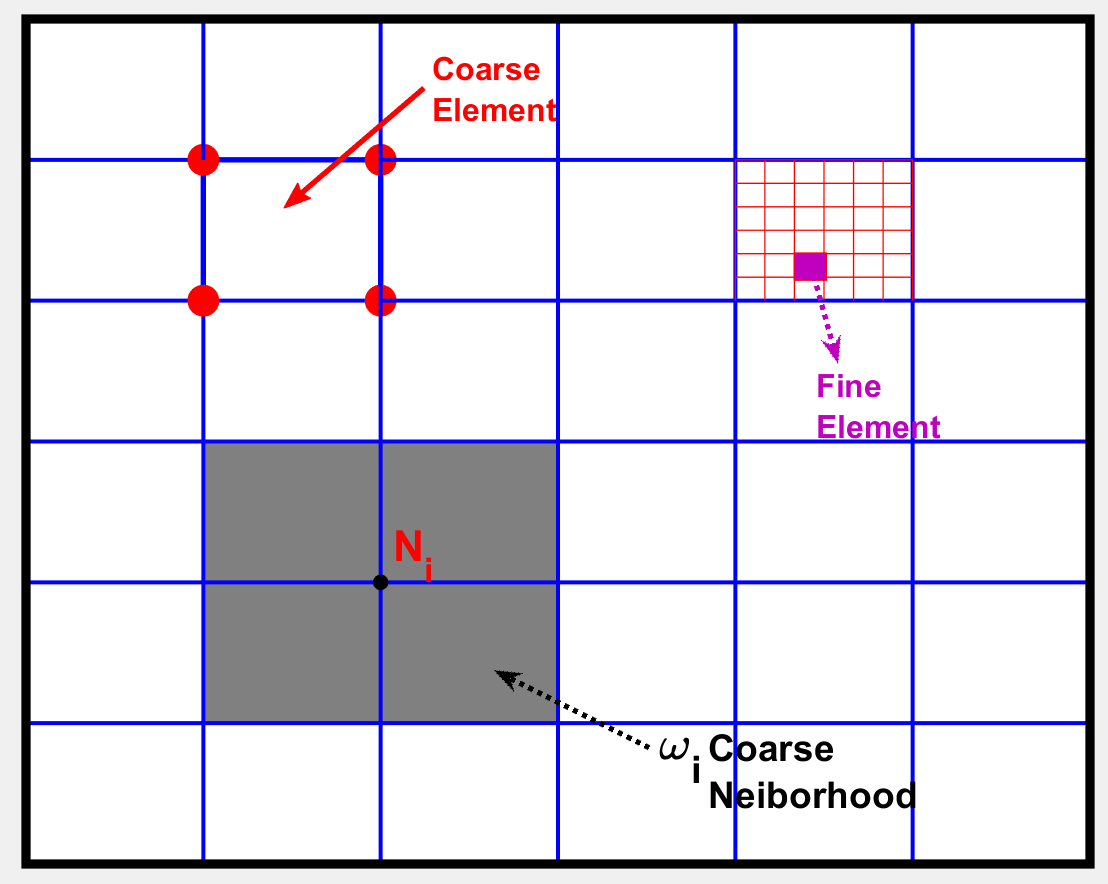}
    \caption{Illustration of the discretization configuration.}
    \label{mesh_GMs}
\end{figure}

Generalized multiscale finite element method can be splitted into two stages: offline stage and online stage. At the offline stage of GMsFEM, we firstly construct the space of snapshots, $V^{(i)}_{\text{snap}}$, and then introduce the construction of local reduced basis functions. For the snapshot space, we can construct it by various ways: (1) all fine-grid functions; (2) harmonic snapshots; (3) oversampling harmonic snapshots; and (4) forced-based snapshots. In the following, we will introduce the first way to form the snapshot spaces.

Let $\mathcal{L}$ denote the differential operator associated with the optimality system \eqref{opt_sys}, which consists of the state equation—a diffusion equation with spatially varying coefficient $\kappa(x)$—and its corresponding adjoint equation. Firstly, we solve the eigenvalue problem on each subdomain $\omega_i$,
\begin{equation}
\label{eigen_problem}
\left\{
    \begin{aligned}
   \mathcal{L}(\phi_{i,\ell})&=\lambda_{i,\ell}\widetilde{\kappa(x)}\phi_{i,\ell}\quad in \;\omega_i,\\
    \kappa(x)\cdot\nabla\phi_{i,\ell}\cdot n&=0\quad on\;\partial\omega_i.
    \end{aligned}
\right.
\end{equation}
This gives a set of snapshot functions $\left\{\phi_{i,\ell}\right\}$. Here $\widetilde{\kappa(x)}=\kappa(x)\displaystyle\sum_{i=1}^{N_c}H^2|\nabla\chi_i|^2$ and $\left\{\chi
_i\right\}$ is a set of partition of unity functions.

Let $\left\{L_m\right\}$ denotes the set of basis functions which are defined in the fine grid. After applying the FEM to discretize the snapshot equation \eqref{eigen_problem}, we obtain the matrix eigenvalue problem.
\begin{equation}\label{eigen_pro_matrix}
    K\phi_{i,\ell}=\lambda S\phi_{i,\ell},
\end{equation}
where $$(K)_{m,n}=(\kappa(x)\nabla L_m,\nabla L_n)_{L^2(\omega_i)},\quad(S)_{m,n}=(\widetilde{\kappa(x)}L_m,L_n)_{L^2(\omega_i)}.$$
Then we choose the $M_i$ lowermost eigenvalues and the corresponding eigenvectors of the eigenvalue problem \eqref{eigen_pro_matrix}, and denote the eigenvalues and eigenvectors by $\left\{\lambda_{\ell}^{\omega_i}\right\}$ and $\left\{\varphi_{\ell}^{\omega_i}\right\}$, respectively. The local snapshot space $V_{\text{snap}}(\omega_i):=\{\varphi_\ell^{\omega_i},1\leq\ell\leq M_i\}.$ Let $M=\sum_{i=1}^{N_c}M_i$ be the total number of eigenvectors for snapshots. We use the partition of unity functions to paste the snapshot functions and get the multiscale basis function space 
$$V_H(\Omega):=\operatorname{span}\{\varphi_k:1\leq k\leq M\}=\operatorname{span}\{\varphi_{i,\ell}:\varphi_{i,\ell}=\chi_i\varphi_\ell^{\omega_i},1\leq i\leq N_c,1\leq\ell\leq M_i\}.$$ 
The multiscale basis functions can be repeatedly used online. We can use a matrix $R^G$ to store these basis functions, i.e.,
\begin{equation}
    \label{matrix R}
    (R^G)^T=[\varphi_1,\varphi_1,,\dots,\varphi_M],
\end{equation}
where each $\varphi_j$ denotes a column vector. We look at the GMsFEM closer in terms of matrix-vector multiplication. Let $K_g$ and $F_g$ be the stiffness matrix and the load vector on coarse grid, and $u_H$ represents the coarse solution. The discrete formulation of state equation in \eqref{opt_sys} on coarse grid is as follows,
\begin{equation}
    a(u_H,v;\mu)=(f,v;\mu),\;\forall v\in V_H,
\end{equation}
This leads to a matrix form
\begin{equation}\label{mul_matrix}
    K_g\overline{u}_H=F_g,
\end{equation}
where $K_g=R^GK(R^G)^T,\; F_g=R^Gb$
($b$ represents the load-vector on fine grid). Then we downscale the coarse scale solution to fine scale solution by using multiscale basis functions,
\begin{equation}
    u_H=(R^G)^T\overline{u}_H.
\end{equation}

At the online stage, the multiscale basis functions can be repeatedly used, so that Compared with direct numerical simulation on fine grid, GMsFEM can significantly improve the computation efficiency. Then, we use the Algorithm \ref{alg_GMsFEM} to give an overview of GMsFEM.

\begin{algorithm}[h]
\caption{Generalized multiscale finite element method (GMsFEM)}
\label{alg_GMsFEM}
\renewcommand{\algorithmicrequire}{\textbf{Input:}}
\renewcommand{\algorithmicensure}{\textbf{Output:}}

\begin{algorithmic}[1]

\REQUIRE  The variational problem $a(u, v) = (f, v)$ defined on $\Omega$;
        coefficient $\kappa(x)$, partition of unity $\{\chi_i\}$.
\ENSURE $u_H$.

\vspace{0.3em}

\textbf{Offline Stage:}
\STATE Generate fine mesh $\mathcal{T}_h$ and coarse mesh $\mathcal{T}_H$;
\STATE Define coarse neighborhoods $\{\omega_i\}$ for each coarse node $N_i$;

\FOR{each coarse neighborhood $\omega_i$}
\STATE Solve the eigenvalue problem \eqref{eigen_pro_matrix} and select  eigenvectors $\{\varphi_\ell^{\omega_i}\}$ 
corresponding to the first $M_i$ smallest eigenvalues;
\STATE Construct local snapshot space 
        $V_{\mathrm{snap}}(\omega_i) = \mathrm{span}\{\varphi_\ell^{\omega_i}\}_{\ell=1}^{M_i}$;

\ENDFOR

\STATE Construct multiscale basis function space 
$V_H =\mathrm{span}\{\varphi_{i,\ell}:\varphi_{i,\ell}=\chi_i\varphi_\ell^{\omega_i}\}_{\ell=1}^{M_i}$ and use the matrix $R^G$ to store these basis functions.

\vspace{0.3em}

\textbf{Online Stage:}

\STATE Assemble fine-scale stiffness matrix $A$ and load vector $b$, then project to coarse scale: $K = R^G A (R^G)^T$, $F = R^G b$;
\STATE Solve the coarse system \eqref{mul_matrix} to obtain $\bar{u}_H$;
\STATE Reconstruct fine-scale solution $u_H = (R^G)^T \bar{u}_H$.

\end{algorithmic}
\end{algorithm}

Next, we apply GMsFEM into the optimality system \eqref{Galerkin} i.e., finding $(u_H,f_H,\lambda_H)\in\mathscr{V}_0^H(\Omega)\times\mathcal{M}_h(\Omega)\times\mathscr{V}_0^H(\Omega)$ such that

\begin{equation}
\label{Gms_gar}
\left\{
\begin{aligned}
a(u_H,\widetilde{u}_H;\mu)&=(f_H,\widetilde{u}_H;\mu),\;\forall\widetilde{u}_H\in\mathscr{V}_0^H(\Omega),\\
a(\lambda_H,\widetilde{\lambda}_H;\mu)&=-(u_H-\hat{u},\widetilde{\lambda}_H;\mu),\;\forall\widetilde{\lambda}_H\in\mathscr{V}_0^H(\Omega),\\
2\beta(f_H,\widetilde{f}_H;\mu)&=(\widetilde{f}_H,\lambda_H;\mu),\;\forall\widetilde{f}_H\in\mathcal{M}_h(\Omega),
\end{aligned}
\right.
\end{equation}
where $\mathscr{V}_0^H(\Omega):=V_H(\Omega)\otimes L^2(\Gamma)\subset\mathscr{V}_0^h(\Omega).$ The reduced optimality matrix corresponding to \eqref{Gms_gar} is
\begin{equation}
\label{gms_matrix_form}
\begin{aligned}
&\underbrace{
\begin{bmatrix}
2\beta M_{1}&0&-M_{2}^TR^G\\0&(R^G)^TM_{3}R^G&(R^G)^TK^T(\mu)R^G\\-(R^G)^TM_{2}&(R^G)^TK(\mu)R^G&0\end{bmatrix}}_{\mathscr{A}_H(\mu)\in\mathbb{R}^{(2M+N_e)\times(2M+N_e)}}
\begin{bmatrix}
\overline{F}_H(\mu)\\\overline{U}_H(\mu)\\\overline{\Lambda}_H(\mu)
\end{bmatrix}
=\begin{bmatrix}0\\(R^G)^T\overline{\widehat{U}}(\mu)\\(R^G)^Td\end{bmatrix},
\end{aligned}
\end{equation}
If we set 
$$
R=
\begin{bmatrix}
I & 0 & 0 \\
0 & (R^G)^T & 0 \\
0 & 0 & (R^G)^T
\end{bmatrix}
\in\mathbb{R}^{(2M+N_e)\times(2M+N_e)},
$$
we can get $\mathscr{A}_H(\mu)=R\mathscr{A}_h(\mu)R^T,$ and then 
$$
\begin{aligned}
    F_H(\mu)&=(R^G)^T\overline{F}_H(\mu),\\
    U_H(\mu)&=(R^G)^T\overline{U}_H(\mu),\\
    \Lambda_H(\mu)&=(R^G)^T\overline{\Lambda}_H(\mu).
\end{aligned}
$$
\eqref{gms_matrix_form} can  be transformed into the form of variable separation, so it can be converted into the forms similar to \eqref{FOM_EXT} and \eqref{fom_ext_int} by introducing the interest quantity, so we get the FOM using GMsFEM.

\section{Numerical examples}
In this section, we validate the effectiveness of the proposed $\mathcal{L}_2$-Opt-PSF algorithm (Algorithm \ref{alg_L2}) for solving stochastic optimal control problems through two representative numerical examples. The experiments are conducted under various parameter settings for the PDE constraints, objective functionals, and control constraints, with particular attention to the performance of the algorithm in the one-dimensional parameter case.

To construct the initial reduced-order model for the $\mathcal{L}_2$-Opt-PSF algorithm, we employ the RB method combined with a greedy algorithm in the numerical experiments. Let the parameter space be $\mathcal{P}=[0.1,10]$, and let the training set $\mathcal{P}_{\text{train}}$ consist of 100 uniformly sampled points generated by the MATLAB \texttt{linspace} function to 2ensure sufficient coverage of the parameter space. The parameters are thus uniformly distributed over the interval. The greedy algorithm adaptively selects the optimal sample points through a posteriori error estimation and successively enriches the basis function space. At the $n$-th iteration, the full-order solution corresponding to the parameter point with the largest error in the current reduced-order model is added as a new basis function, and this process repeats until the prescribed reduced dimension $r$ is reached. The reduced-order model constructed by this strategy possesses global optimality and serves as a good initial guess for the subsequent $\mathcal{L}_2$-Opt-PSF optimization algorithm.

The gradient-based optimization in Algorithm \ref{alg_L2} employs the BFGS algorithm with a strong Wolfe line search to accelerate convergence, where the maximum number of iterations $\mathbf{maxit}$ is set to $1000$ and the tolerance $\mathbf{tol}$ is set to $10^{-16}$.

In the following numerical experiments, we first construct the full-order model based on the standard finite element method to verify the basic feasibility of the proposed algorithm and to present the corresponding reference solutions. However, when the mesh is further refined or the size of the training set increases, directly using the finite element method to generate the snapshot space incurs significant computational cost. To alleviate this issue, we further introduce the generalized multiscale finite element method described in Section \ref{sec_GMsFEM}. This method effectively captures fine-scale features on a coarse grid, thereby reducing the overall computational complexity while maintaining accuracy.

\subsection{Optimal control for stochastic diffusion equation}
We start with the following
stochastic optimal control problem defined on the two-dimensional unit square $\Omega=[0,1]^2$ :
\begin{equation}
\left\{  
    \begin{aligned}
     \min_{u,f}J(u,f)&=\frac{1}{2}\|u(x,\mu)-\hat{u}(x,\mu)\|_{\mathscr{L}^{2}(\Omega)}^{2}+\beta\|f(x,\mu)\|_{\mathscr{L}^{2}(\Omega)}^{2},\\
     s.t.-div(&\kappa(x,\mu)\nabla u(x,\mu))=f(x,\mu),\; in\;\Omega,\\
     &u(x,\mu) = 0,\;on\;\partial\Omega,
    \end{aligned}
\right.  
\end{equation}
where $\beta=10^{-3}$, the diffusion coefficient $\kappa(x, \mu)$ and the desired state function $\hat{u}(x,\mu)$ are
$$
\left\{  
    \begin{aligned}
     \kappa(x,\mu)&=\kappa_1(x)+\mu(1-x_1)\\
     \hat{u}(x,\mu)&=x_1x_2(1-x_1)(1-x_2)+\mu x_1^2x_2^2(1-x_1)(1-x_2)
    \end{aligned}
\right.  
$$
Here $x=(x_1,x_2)\in\Omega$, 
 $\kappa_1(x)$ is a high-contrast function and its map is depicted in Figure \ref{fig_high_contrast}.
\begin{figure}[H]
    \centering   
\includegraphics[width=0.6\linewidth]{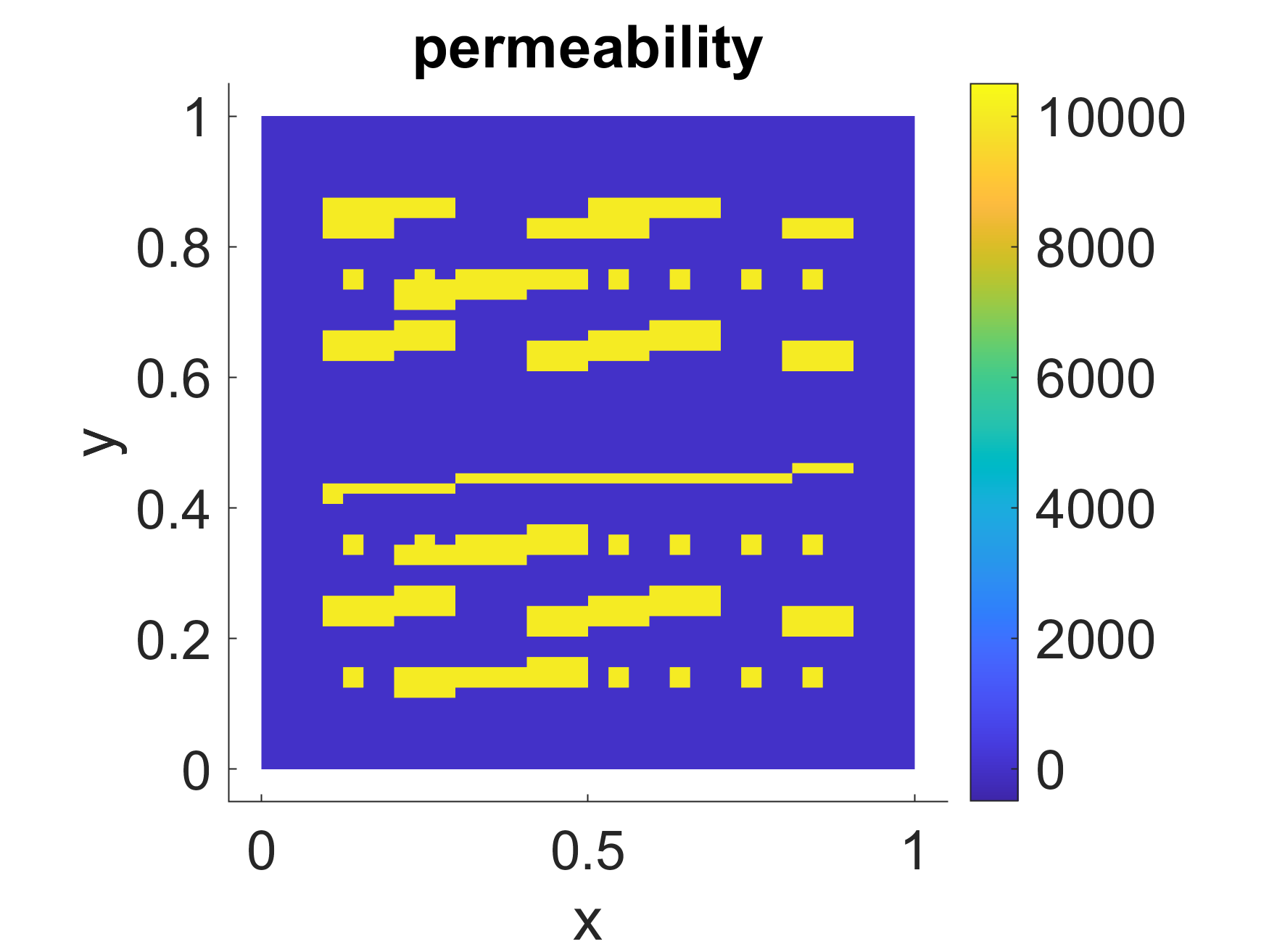}
    \caption{High-contrast coefficient $\kappa_1$}
    \label{fig_high_contrast}
\end{figure}
This section presents two sets of numerical experiments. In the first part, the finite element method is used to generate the full-order data to validate the proposed algorithm. In the second part, the generalized multiscale finite element method is employed to generate the full-order data for optimization. The details of these experiments are elaborated in the following two subsections.

\subsubsection{Comparison between the FEM-based FOM and the DDROM}
First, for the FOM, we employ the finite element method on a $64\times64$ uniform fine grid to compute the reference solutions $(F_{\text{ref}}, U_{\text{ref}}, \Lambda_{\text{ref}})$ and to construct the snapshot space. Subsequently, the RB method is used to select the snapshot matrix and construct a reduced-order model with reduced dimension $r=3$, which serves as the initial guess for the optimization algorithm. The solution of this ROM is denoted by $(F_{\text{rom}}, U_{\text{rom}}, \Lambda_{\text{rom}})$. Based on this, the DDROM and its solution $(F_{\text{ddrom}}, U_{\text{ddrom}}, \Lambda_{\text{ddrom}})$ are further obtained according to Algorithm\ref{alg_L2}.
For the output quantity of interest, we define
$$
y(\mu)=
\left\{
\begin{bmatrix}0\\0\\d\end{bmatrix}+\mu\begin{bmatrix}0\\\widehat{U}_p\\0\end{bmatrix}
\right\}x(\mu).
$$ 
Under the general definition of $y(\mu)$ given above, the absolute and relative errors of the ROM and DDROM are presented in Figure \ref{fig_E1_FEM_y}. In this figure, the blue curve shows the error of the reduced-order model that serves as the initial guess for the algorithm, while the red curve represents the error of the DDROM obtained after optimization by the $\mathcal{L}_2$-Opt-PSF algorithm. Here, $\widehat{y}(\mu)$ denotes the output of the respective reduced-order model, and the reference solution $y(\mu)$ is obtained from the finite element method.
\begin{figure}[H]
    \centering
    \begin{subfigure}[b]{0.48\textwidth}
        \includegraphics[width=\textwidth, height=0.35\textheight, keepaspectratio]{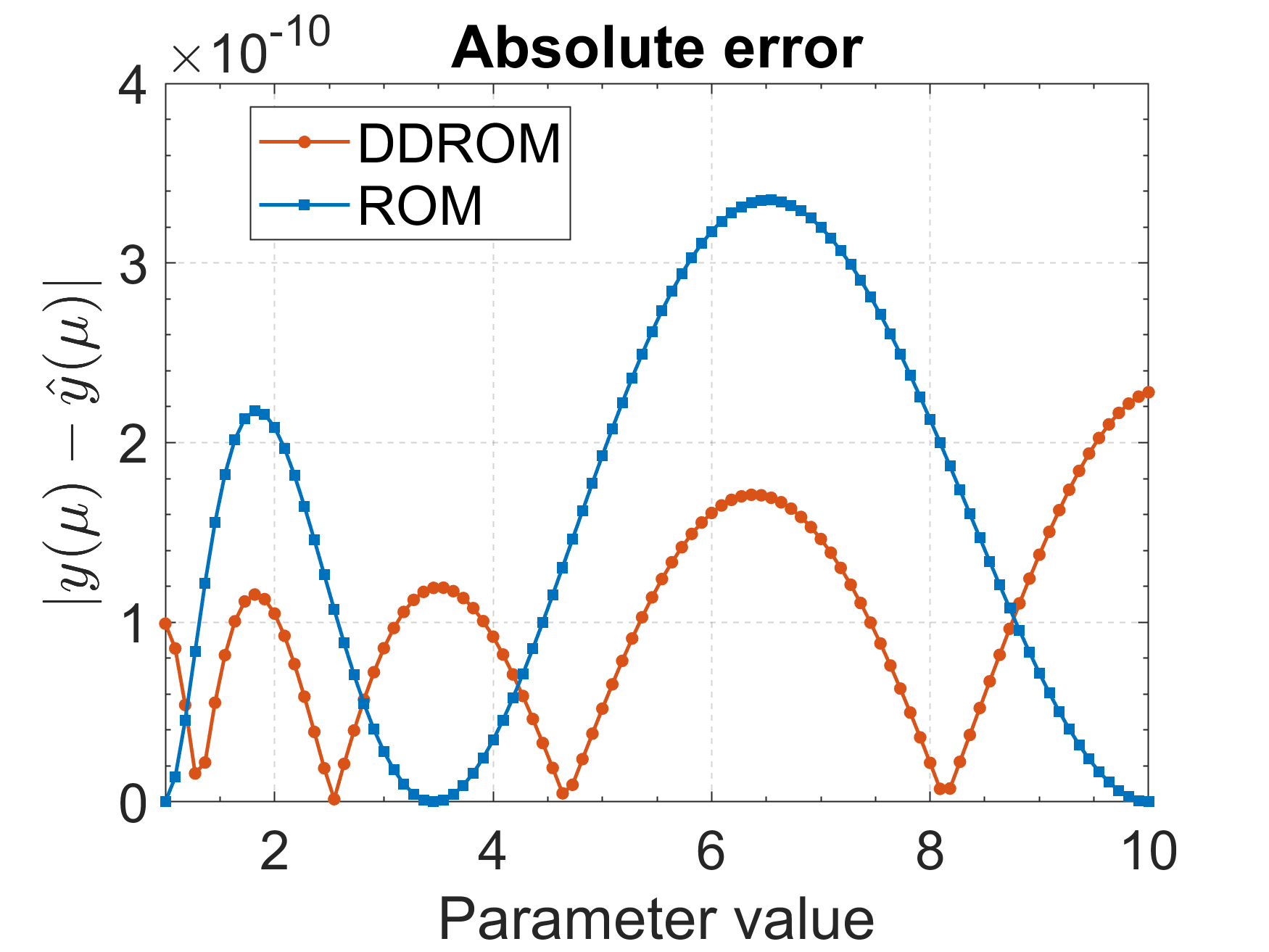}
    \end{subfigure}
    \hfill
    \begin{subfigure}[b]{0.48\textwidth}
        \includegraphics[width=\textwidth, height=0.35\textheight, keepaspectratio]{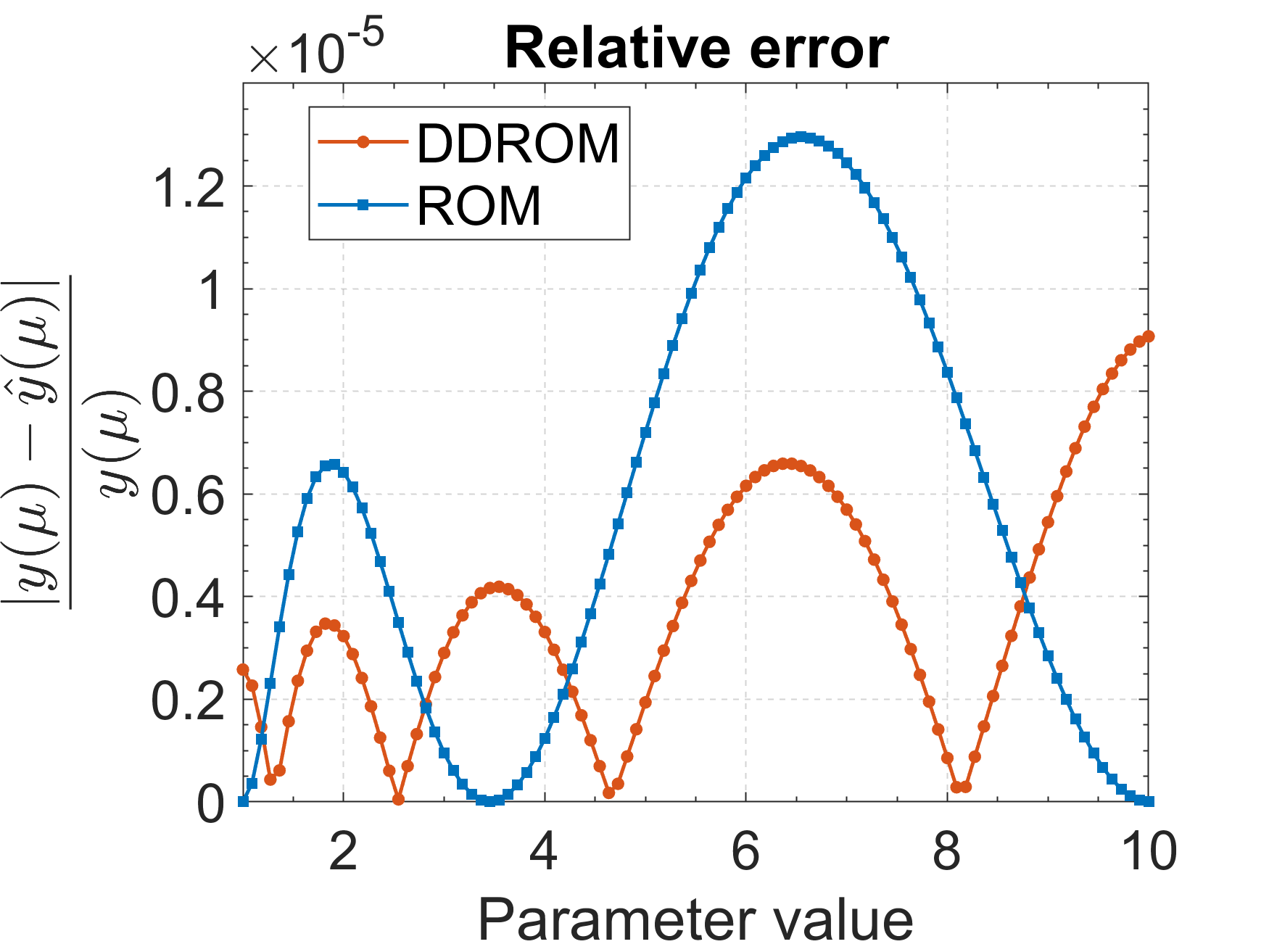}
    \end{subfigure}
    \captionsetup{justification=centering, font=small, skip=6pt}
    \caption{Error plots of the output quantity $y$.}
    \label{fig_E1_FEM_y}
\end{figure}

Since the optimal control problem pays more attention to $U$ and $F$ in \eqref{alge_form}, we need to adjust the matrix 
$C$ in \eqref{int_dis} so that the output $y(\mu)$ is only related to $U$ or $F$. Obviously, we only need to let $y(\mu)$ have the following form respectively,
$$
y(\mu)=
\left\{
\begin{bmatrix}0\\I\\0\end{bmatrix}+\mu\begin{bmatrix}0\\0\\0\end{bmatrix}
\right\}x(\mu),
$$ 
$$
y(\mu)=
\left\{
\begin{bmatrix}I\\0\\0\end{bmatrix}+\mu\begin{bmatrix}0\\0\\0\end{bmatrix}
\right\}x(\mu).
$$ 
The relative error plots are shown in the following figure.
\begin{figure}[H]
    \centering
    \begin{subfigure}[b]{0.48\textwidth}
        \includegraphics[width=\textwidth]{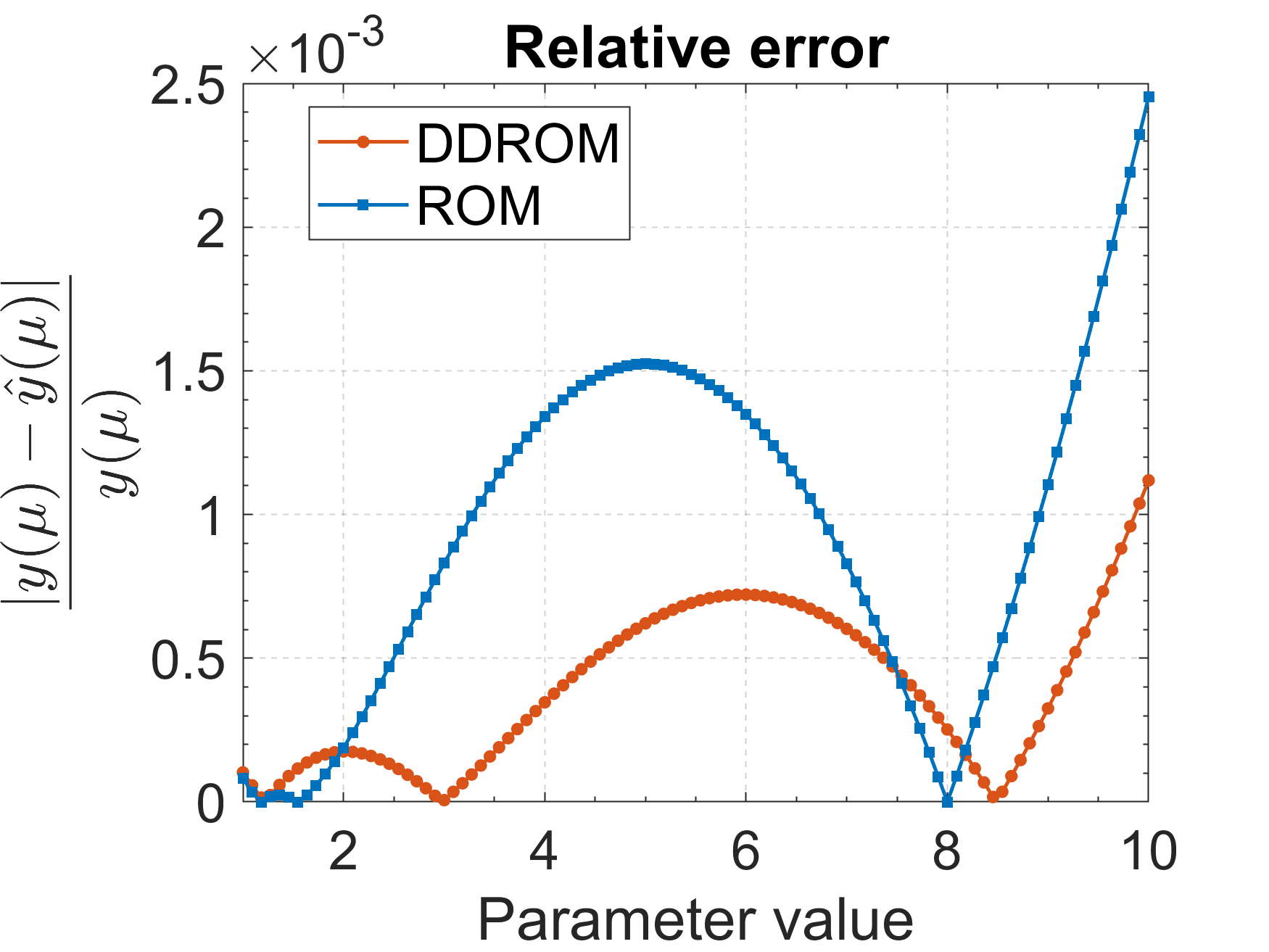}
        \centering (a) Output $y$ related only to $U$
    \end{subfigure}
    \hfill
    \begin{subfigure}[b]{0.48\textwidth}
        \includegraphics[width=\textwidth]{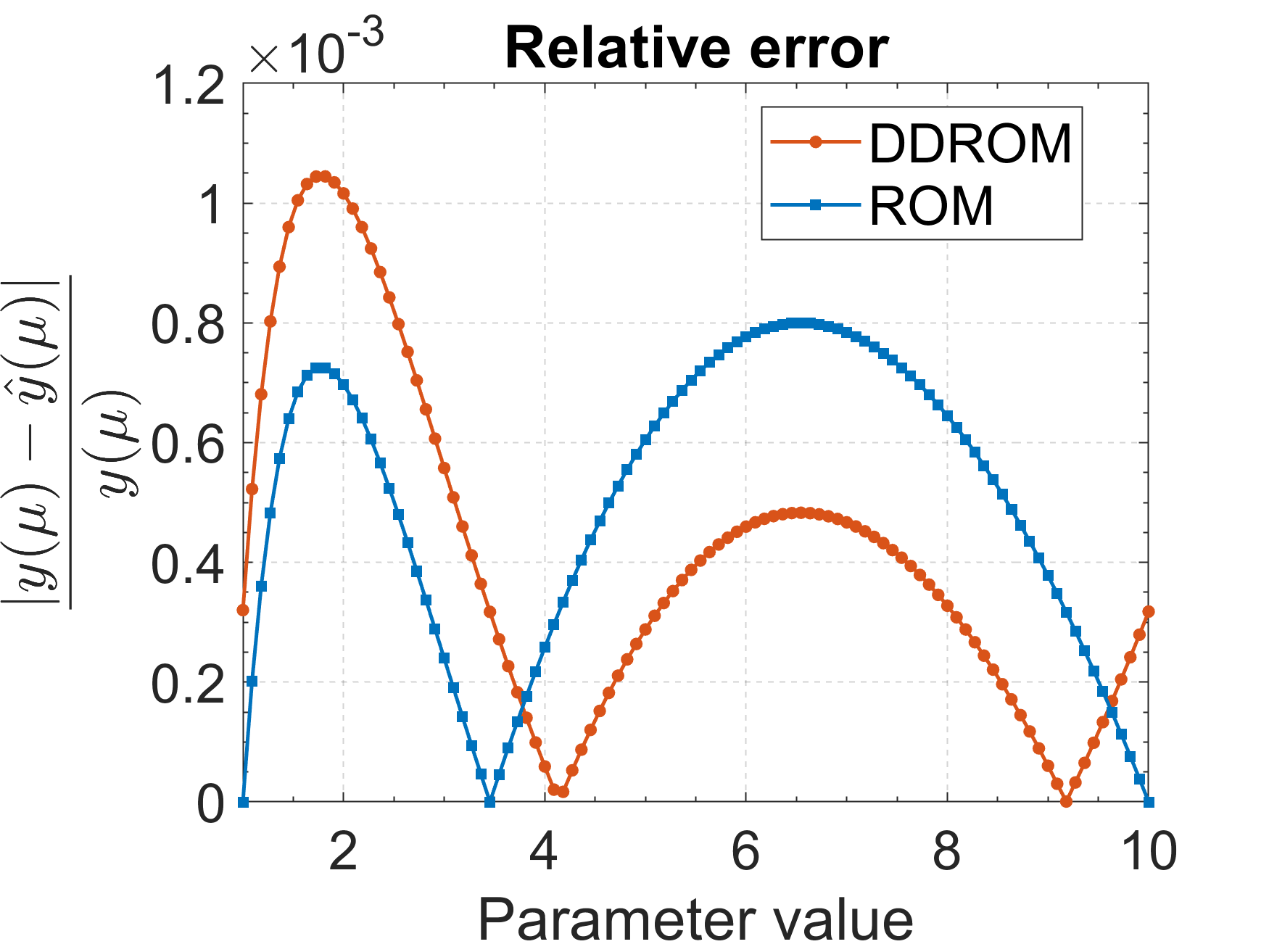}
        \centering (b) Output $y$ related only to $F$
    \end{subfigure}
    \captionsetup{justification=centering, font=small, skip=6pt}
    \caption{Relative error plots for the quantity of interest $y$ related only to $U$ or $F$.}
    \label{fig_E1_FEM_U&F}
\end{figure}

As observed in Figure \ref{fig_E1_FEM_y}, the DDROM output $\widehat{y}(\mu)$ exhibits high consistency with the reference solution $y(\mu)$ over the entire parameter interval, with the relative error remaining at a low level overall. This indicates that the proposed method possesses good approximation capability at the output level. Moreover, the error variation is relatively stable across different parameter values, demonstrating a certain degree of robustness in the parameter space.

Furthermore, when the output depends solely on the state variable $U$ or the control variable $F$ (see Figure \ref{fig_E1_FEM_U&F}), the DDROM still maintains high accuracy. This shows that the proposed method is not only effective in the sense of overall output approximation but also performs well in approximating specific physical variables.

\begin{table}[h]
    \centering
    \caption{\small The $L_2$ relative errors with different methods for the state variable $u$ and the control variable $f$.}
    \begin{tabular}{l l l l}
        \toprule
               & RB & POD & $\mathcal{L}_2$-Opt-PSF \\
        \midrule
        \(e_{u}\) & 1.47e-06 & 1.36e-07 & 1.15e-07\\
        \(e_{f}\)&6.44e-07 & 1.05e-07  &2.30e-07 \\
        \bottomrule
    \end{tabular}
    \label{tab_E1_FEM_method}
\end{table}
To quantitatively evaluate the model performance, we compare the proposed method with two classical model reduction approaches, the RB method and the POD method. The $L_2$ relative errors for the state variable $u$ and the control variable $f$ are respectively defined as
\begin{equation}
    \begin{aligned}
     e_u = \frac{1}{N} \sum_{i=1}^{N} \frac{\| U_{\text{rom}}(\mu_i) - U_{\text{ref}}(\mu_i) \|_{L_2(\Omega)}}{\| U_{\text{ref}}(\mu_i) \|_{L_2(\Omega)}}, \\
    e_f = \frac{1}{N} \sum_{i=1}^{N} \frac{\| F_{\text{rom}}(\mu_i) - F_{\text{ref}}(\mu_i) \|_{L_2(\Omega)}}{\| F_{\text{ref}}(\mu_i) \|_{L_2(\Omega)}}.
    \end{aligned}
\end{equation}

As can be seen from the Table \ref{tab_E1_FEM_method}, the $\mathcal{L}_2$-Opt-PSF method outperforms both the RB and POD methods in approximating the state variable $u$, while for the control variable $f$, it achieves an accuracy comparable to that of the POD method and slightly outperforms the RB method. Overall, the proposed method exhibits a well-balanced approximation capability across different variables, which is consistent with its design philosophy of directly optimizing the output error.

In summary, the proposed $\mathcal{L}_2$-Opt-PSF method outperforms both classical methods in approximating the state and control variables. These results confirm the effectiveness of the method in achieving model reduction by directly minimizing the output error, aligning with its design objective.

\subsubsection{Comparison between the GMsFEM-based FOM and the DDROM}
\label{sec_E1_GMsFEM}
In the numerical experiments of this section, the Generalized Multiscale Finite Element Method is employed 
directly as the full order model solver. First, the computational domain is uniformly discretized into a $128\times128$ fine mesh, upon which an $8\times8$  uniform coarse mesh is constructed. The solution obtained on this coarse mesh serves as the FOM reference solution $(F_{\text{fom}},U_{\text{fom}},\Lambda_{\text{fom}})$ relied upon in this experiment.We then select a snapshot matrix to derive the ROM with $r=3$ by RB method, which serves as the initial value for the algorithm. For the output in the model, we let 
$$
y(\mu)=
\left\{
\begin{bmatrix}0\\0\\d\end{bmatrix}+\mu\begin{bmatrix}0\\\widehat{U}_p\\0\end{bmatrix}
\right\}x(\mu).
$$ 
\begin{figure}[H]
    \centering
    \begin{subfigure}[b]{0.48\textwidth}
        \includegraphics[width=\textwidth]{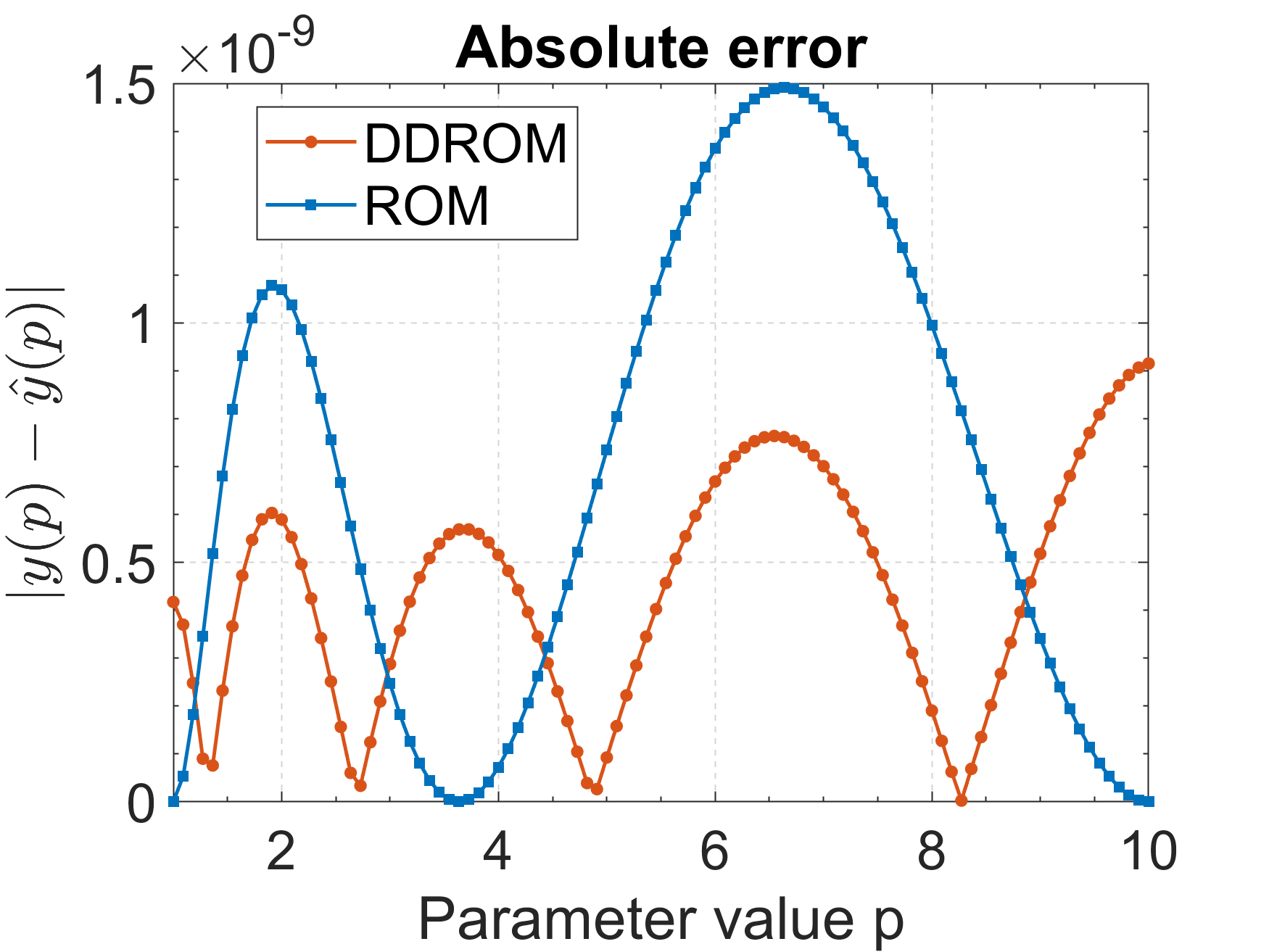}
    \end{subfigure}
    \hfill
    \begin{subfigure}[b]{0.48\textwidth}
        \includegraphics[width=\textwidth]{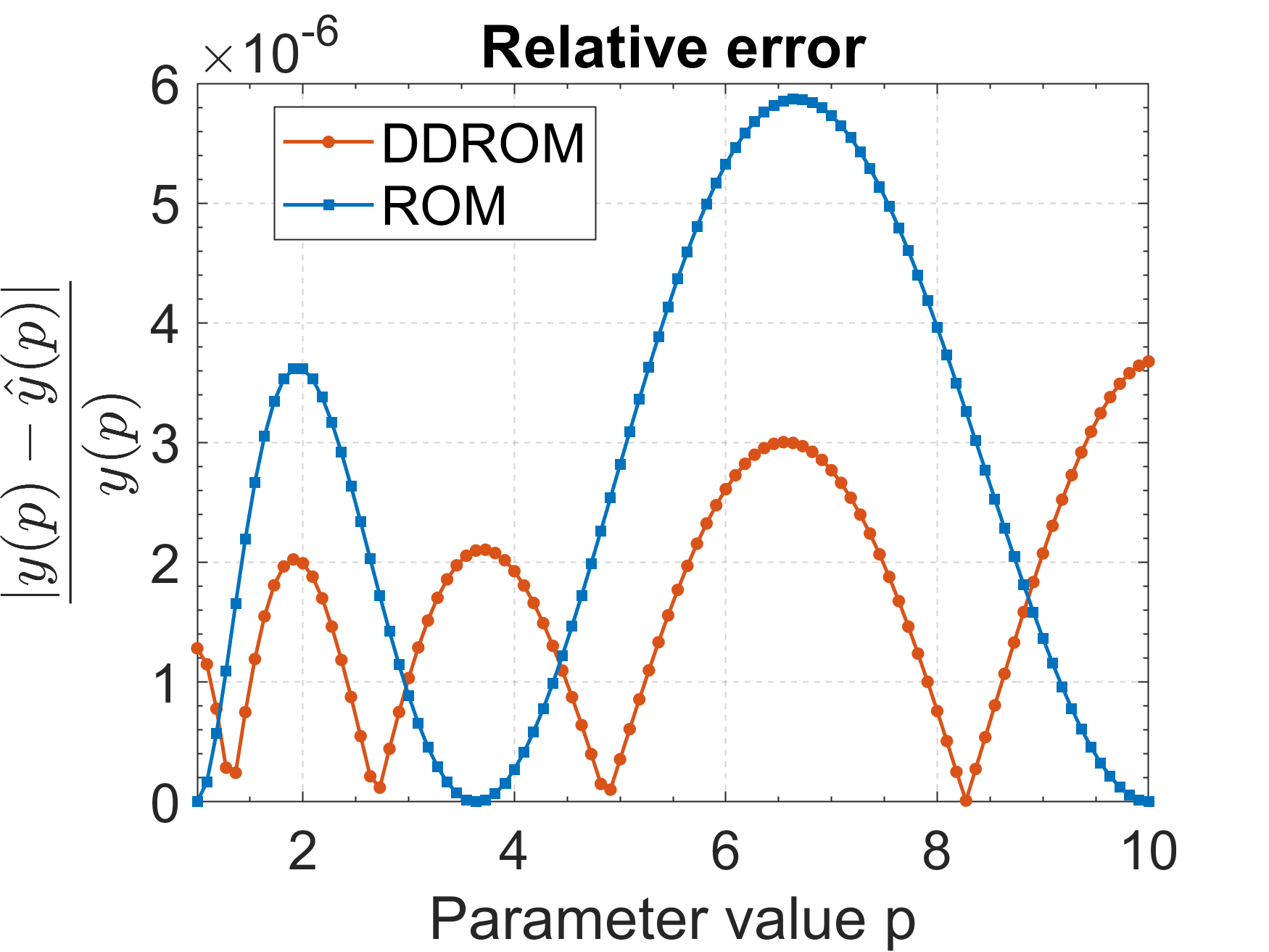}
    \end{subfigure}
    \captionsetup{justification=centering, font=small, skip=6pt}
    \caption{Error plots of the output quantity $y$.}
    \label{fig_E1_GMsFEM_err_128_8}
\end{figure}
As can be observed from Figure \ref{fig_E1_GMsFEM_err_128_8}, the consistency between the DDROM output $\widehat{y}(\mu)$, optimized by the $\mathcal{L}_2$-Opt-PSF algorithm, and the reference solution $y(\mu)$ is significantly better than that of the initial reduced-order model over the entire parameter interval $p \in [0.1,10]$. This demonstrates that the proposed $\mathcal{L}_2$-Opt-PSF algorithm effectively enhances the output approximation accuracy of the reduced-order model by directly minimizing the output error.

Figure \ref{fig_E1_U_128_8} presents the contour plots of the target solution $U$ averaged over the 100 test parameter points, and Figure \ref{fig_E1_F_128_8} presents the contour plots of the target solution $F$ averaged over the same set of parameters. The three columns of subfigures correspond, from left to right, to the reference solution, the initial reduced-order model solution constructed by the RB method, and the data-driven reduced-order model solution obtained after optimization by the $\mathcal{L}_2$-Opt-PSF algorithm.

\begin{figure}[H]
    \centering
    \begin{subfigure}[htb]{0.33\textwidth}
        \includegraphics[width=\textwidth]{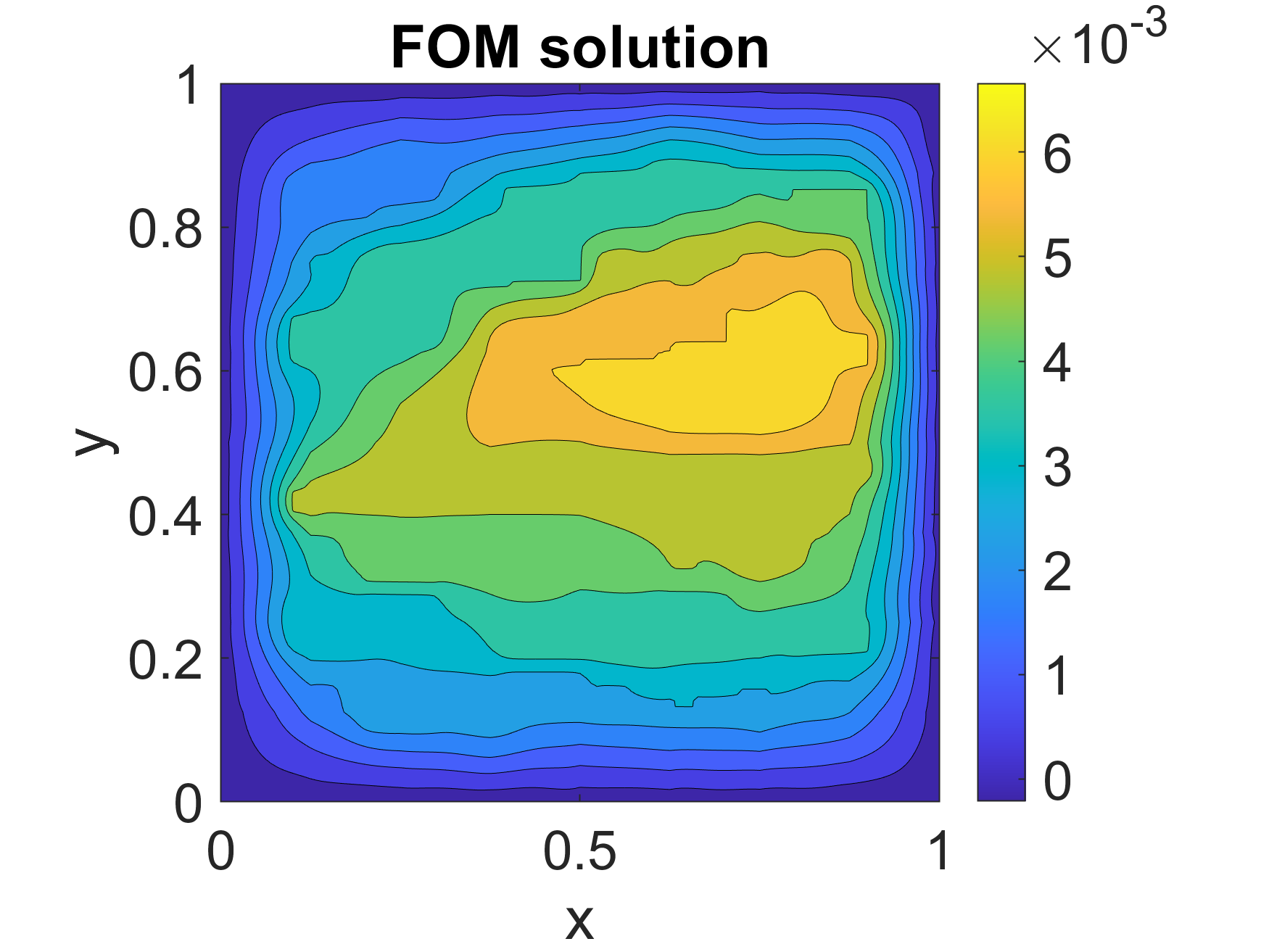}

    \end{subfigure}
    \hspace{-8pt}
    \begin{subfigure}[htb]{0.33\textwidth}
        \includegraphics[width=\textwidth]{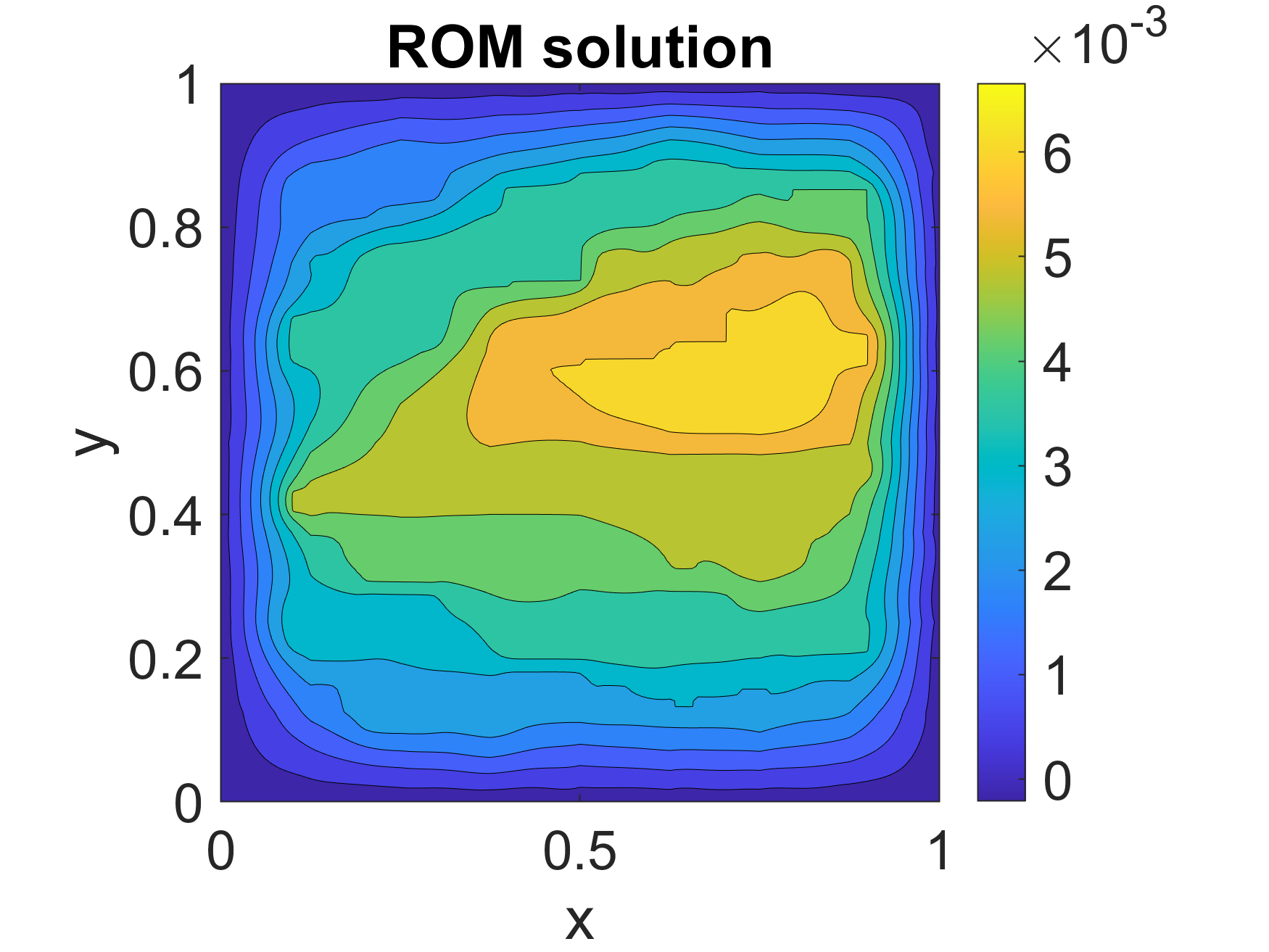}
    \end{subfigure}
    \hspace{-8pt}
    \begin{subfigure}[htb]{0.33\textwidth}
        \includegraphics[width=\textwidth]{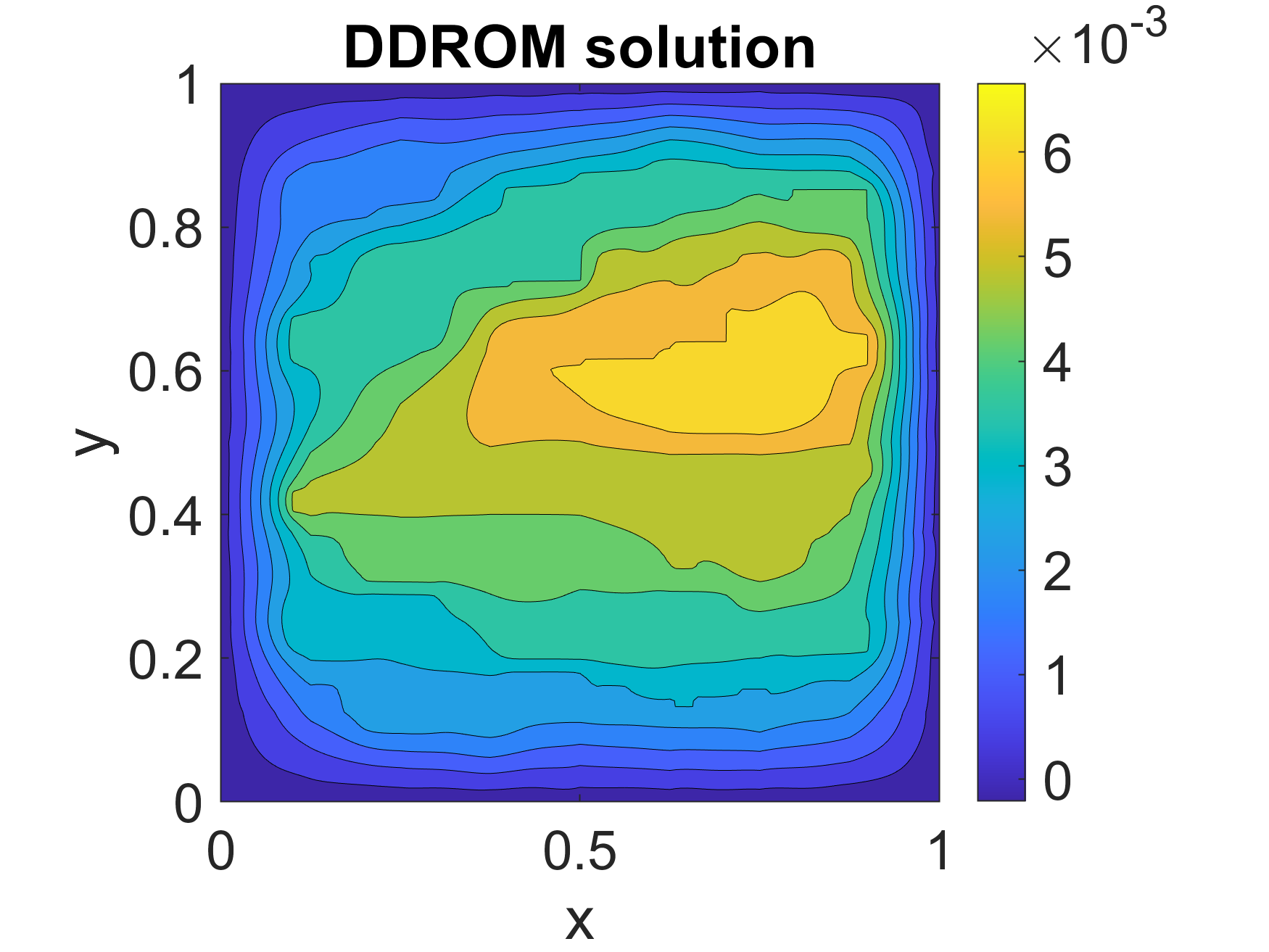}
    \end{subfigure}
    \captionsetup{justification=centering, font=small, skip=6pt}
    \caption{Averaged contour plots of solution $U$ over 100 test parameter points.}
    \label{fig_E1_U_128_8}
\end{figure}

\begin{figure}[H]
    \centering
    \begin{subfigure}[htb]{0.33\textwidth}
        \includegraphics[width=\textwidth]{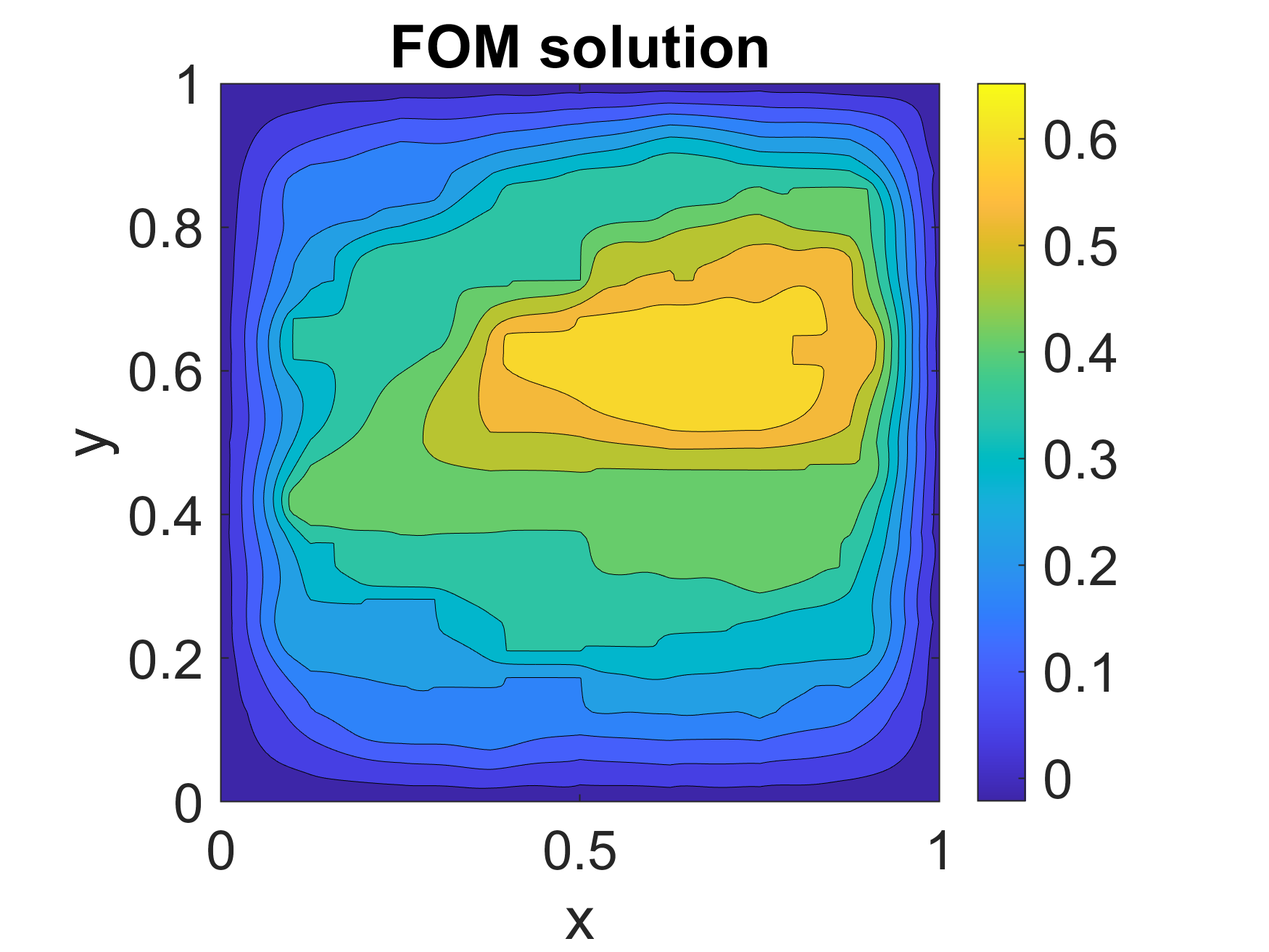}

    \end{subfigure}
    \hspace{-8pt}
    \begin{subfigure}[htb]{0.33\textwidth}
        \includegraphics[width=\textwidth]{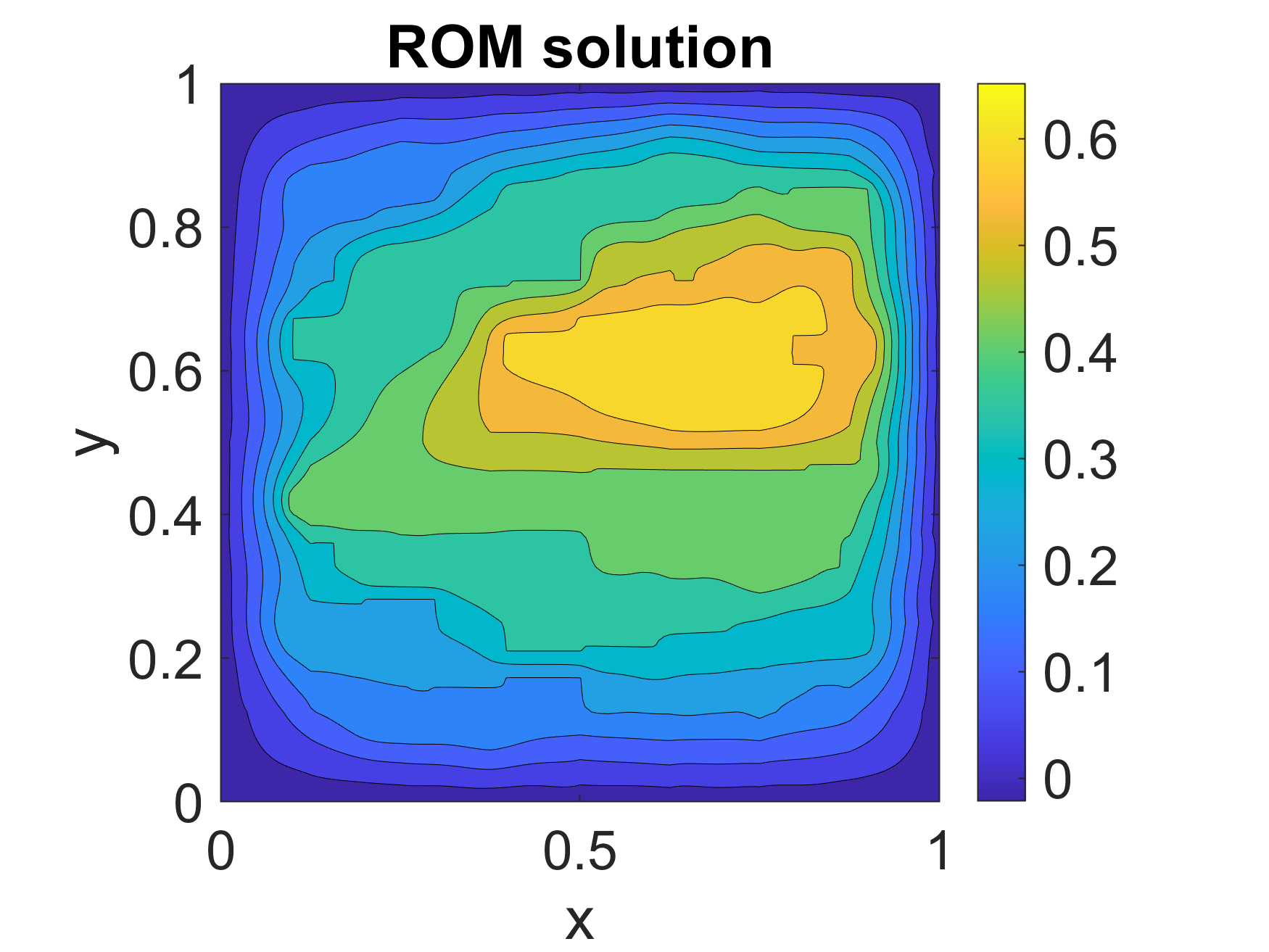}

    \end{subfigure}
    \hspace{-8pt}
    \begin{subfigure}[htb]{0.33\textwidth}
        \includegraphics[width=\textwidth]{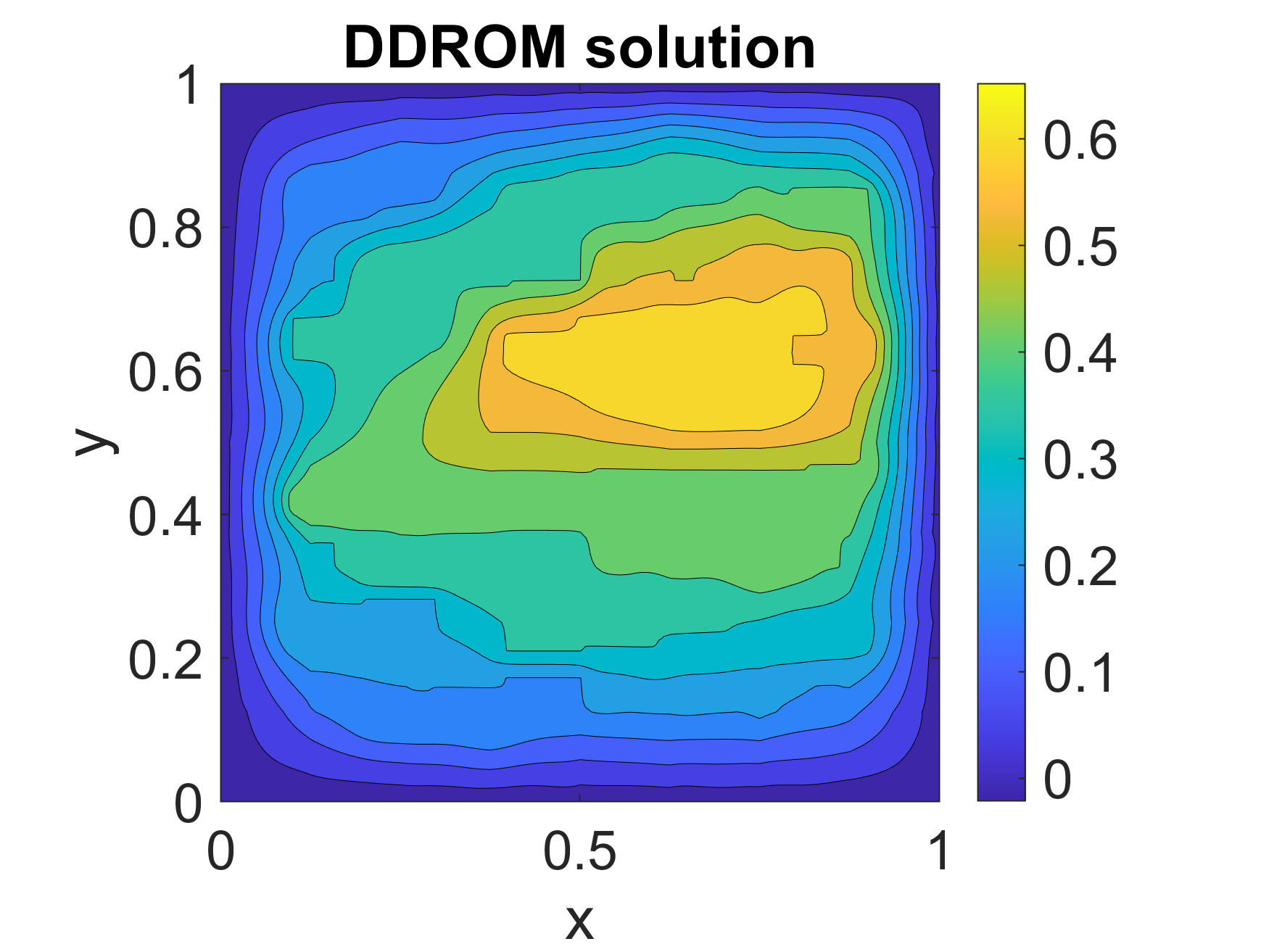}

    \end{subfigure}
    \captionsetup{justification=centering, font=small, skip=6pt}
    \caption{Averaged contour plots of solution $F$ over 100 test parameter points.}
    \label{fig_E1_F_128_8}
\end{figure}

As can be seen from Figure \ref{fig_E1_U_128_8}, the initial ROM solution already exhibits high consistency with the reference solution in terms of the overall distribution, indicating that the reduced-order model constructed by the RB method inherently possesses good approximation capability. In comparison, the DDROM solution remains highly consistent with the reference solution visually, but the improvement over the initial ROM is not particularly pronounced.

This phenomenon can be explained from two perspectives. First, according to the quantitative error results, the relative error of the initial ROM for the state variable $U$ has already reached a low level, representing a relatively high accuracy. On this basis, the further optimization space for the $\mathcal{L}_2$-Opt-PSF algorithm is limited, making it difficult to observe significant differences in the visual results. Second, the optimization objective of the proposed method is the $\mathcal{L}_2$ error of the output quantity $y(\mu)$, rather than the direct approximation error of the state variable $U$ or the control variable $F$. Therefore, the improvement is primarily reflected at the output level rather than in the intuitive changes of the solution fields.

Similarly, as observed in Figure \ref{fig_E1_F_128_8}, the contour distribution of the control variable $F$ exhibits a consistent pattern: the initial ROM solution already agrees well with the reference solution, while the DDROM solution shows no significant visual difference compared to the initial ROM, which is consistent with the above analysis.

To verify the convergence of the proposed method and determine the optimal reduced dimension, we investigate the behavior of the output error $e_y$ for reduced dimensions $r$ ranging from $1$ to $5$. The test set consists of $100$ uniformly sampled parameter points over the interval $[0.1,10]$. All other parameters are fixed, and only $r$ is varied. The error is defined as the average $\mathcal{L}_2$ relative error:
\begin{equation}
		e_y = \frac{1}{N} \sum_{i=1}^{N} \frac{\lVert y_{\text{rom}}(\mu_i) - y_{\text{ref}}(\mu_i) \rVert_{L_2(\Omega)}}{\lVert y_{\text{ref}}(\mu_i) \rVert_{L_2(\Omega)}}, 
\end{equation}
which is consistent with the $\mathcal{L}_2$ optimization objective of this paper and effectively reflects the overall approximation performance of the reduced-order model over the parameter space.

\begin{table}[htbp]
	\centering
	\caption{Average $\mathcal{L}_2$ relative errors of the output quantity $y$ with different reduced dimensions $r$.}
	\label{tab_E1_error_vs_r}
	\begin{tabular}{cccc}
		\toprule
		$r$ & ROM & DDROM  & Improvement \\
		\midrule
		1 & 2.367e-01 & 1.736e-01 & 26.66\% \\
		2 & 1.200e-03 & 8.259e-04 & 31.18\% \\
		3 & 3.927e-06 & 3.061e-06 & 22.05\% \\
		4 & 3.512e-08 & 2.660e-08 & 24.26\% \\
		5 & 4.468e-10 & 3.588e-10 & 19.70\% \\
		\bottomrule
	\end{tabular}
\end{table}

As shown in Table \ref{tab_E1_error_vs_r}, as $r$ increases from $1$ to $5$, the output errors of both the ROM and DDROM decrease exponentially, from the order of $10^{-1}$ down to $10^{-10}$, fully validating the convergence of the method. For $r \geq 3$, the error already falls below the order of $10^{-6}$, achieving a satisfactorily high accuracy. Therefore, we choose $r = 3$ as the reduced dimension in our numerical experiments to strike a balance between accuracy and computational efficiency. Moreover, for all tested values of $r$, the DDROM consistently yields lower errors than the ROM, with improvement ratios ranging from $20\%$ to $30\%$. This demonstrates that the $\mathcal{L}_2$-Opt-PSF optimization strategy steadily enhances the approximation accuracy across different model complexities.

To evaluate the computational efficiency of the proposed method, we compare the CPU times of the FOM, ROM and DDROM under the same hardware and software environment. The offline time is reported as the total seconds required for construction. The average online time is computed over 10000 test parameter points. Based on these, we further report the total time for 10000 consecutive queries and the corresponding average total time per query. The results are summarized in Table~\ref{tab_E1_cpu}.

\begin{table}[htbp]
    \centering
    \caption{The CPU time for different models.}
    \label{tab_E1_cpu}
    \begin{tabular}{l p{2.5cm} p{2.5cm} p{2.5cm}}
        \toprule
        CPU(s) &FOM  &ROM  &DDROM \\
        \midrule
        Offline time       & 9.103E-01 & 9.104E-01 & 9.384E-01 \\
        Average online time & 1.062E-02 & 2.131E-05 & 1.179E-05 \\
        Total time       & 1.071E+02 & 1.123E+00 & 1.056E+00 \\
        Average total time & 1.071E-02 & 1.123E-04 & 1.056E-04 \\
        \bottomrule
    \end{tabular}
\end{table}

Based on Table~\ref{tab_E1_cpu}, the offline times for all three models are comparable, as the dominant cost (GMsFEM basis construction and matrix projection) is shared across all models. In terms of online efficiency, the DDROM achieves an average online time of $1.179\times10^{-5}$ seconds, which is approximately 900 times faster than the FOM and about 1.8 times faster than the ROM. For 10000 consecutive queries, the DDROM reduces the total computational cost by 101 times compared to the FOM. These results demonstrate that the proposed $\mathcal{L}_2$-Opt-PSF method is highly efficient for this stochastic optimal control problem, particularly for many query applications such as real-time control and uncertainty quantification.

In summary, the proposed $\mathcal{L}_2$-Opt-PSF method demonstrates strong performance in convergence, accuracy, and efficiency for the stochastic diffusion optimal control problem. With a reduced dimension of $r=3$, it attains an output error of order $10^{-6}$. In terms of computational efficiency, for 10000 consecutive queries, the DDROM reduces the total computational cost by 101 times compared to the FOM. These results confirm the effectiveness and practicality of the proposed method for optimal control problems.

\subsection{Optimal control for stochastic advection-diffusion equation}
In this example, we consider a stochastic advection-diffusion equation as the governing constraint. Specifically, we study the following stochastic optimal control problem defined on the two-dimensional unit square $\Omega = [0,1]^2$:
\begin{equation}
\left\{  
    \begin{aligned}
     \min_{u,f}J(u,f)&=\frac{1}{2}\|u(x,\mu)-\hat{u}(x,\mu)\|_{\mathscr{L}^{2}(\Omega)}^{2}+\beta\|f(x,\mu)\|_{\mathscr{L}^{2}(\Omega)}^{2},\\
     s.t.-div(&\kappa^{ad}(x,\mu)\nabla u(x,\mu))+\delta (x)\cdot\nabla u(x,\mu)=f(x,\mu),\; \text{in}\;\Omega,\\
     &u(x,\mu) = 0,\;\text{on}\;\partial\Omega,
    \end{aligned}
\right.  
\end{equation}
where the desired state function $\hat{u}(x,\mu)$, the diffusion coefficient $\kappa^{ad}(x,\mu)$, and the advection field $\delta(x) = [\delta_1(x), \delta_2(x)]^\top \in \mathbb{R}^2$ are defined as follows:

$$
\left\{  
    \begin{aligned}
    \hat{u}(x,\mu)&=x_1x_2(1-x_1)(1-x_2)+\mu x_1^2x_2^2(1-x_1)(1-x_2),\\
     \kappa^{ad}(x,\mu)&=\kappa_1(x)+\mu(1+x_1),\\
     \delta_1(x)&=1+x_1+x_1^2,\\
     \delta_2(x)&=1+x_2+x_2^2.
    \end{aligned}
\right. 
$$
Here, $\kappa_1(x)$ denotes the same high-contrast coefficient shown in Figure \ref{fig_high_contrast}. 

In the above problem, the diffusion coefficient $\kappa^{ad}(x,\mu)$ consists of a high-contrast function $\kappa_1(x)$ and a parameter-dependent term, endowing the problem with multiscale features in space. Due to the presence of the advection term $\boldsymbol{\delta}(x)\cdot\nabla u(x,\mu)$, the solution exhibits pronounced directional variations in space. Moreover, the diffusion coefficient may become small in certain regions, making the advection effect dominant there. The parameter $\mu$ further increases the complexity of the problem, introducing a strongly nonlinear dependence of the solution on the parameter. Consequently, the problem simultaneously exhibits multiscale behavior, locally advection-dominated characteristics, and parameter dependence, imposing high demands on the accuracy and stability of the model reduction method.

To reduce the high offline computational cost associated with constructing snapshots for the full-order model (FOM), we continue to incorporate the generalized multiscale finite element method (GMsFEM) into the overall computational framework. In this setup, the domain is discretized into a uniform fine mesh of size $128 \times 128$ to adequately resolve the multiscale structure. By constructing multiscale basis functions, the computation is transferred to a coarse mesh of size $8 \times 8$, significantly reducing the computational complexity. Snapshot matrices are constructed from the GMsFEM solutions on the coarse mesh, from which a reduced-order model (ROM) of dimension $r$ is generated. This ROM then serves as the initial guess for the $\mathcal{L}_2$-Opt-PSF optimization, ultimately yielding a data-driven reduced-order model (DDROM).

As demonstrated by the numerical experiments in Section \ref{sec_E1_GMsFEM}, the choice of the reduced dimension $r$ directly affects the trade-off between approximation accuracy and computational efficiency when using the $\mathcal{L}_2$-Opt-PSF algorithm. If $r$ is too small, the reduced-order model fails to capture the essential features of the original problem, leading to large approximation errors. If $r$ is too large, the online computational burden increases, and redundant basis functions may be introduced, potentially causing overfitting. Therefore, it is necessary to determine an appropriate $r$ that balances accuracy against efficiency in this experiment.

To investigate the convergence behavior of $e_y$ with respect to $r$ and to determine the optimal reduced dimension, we fix all other parameters and vary $r$ from $1$ to $5$. A test set consisting of $100$ uniformly sampled parameter points in the interval $[0.1,10]$ is used to compute the corresponding average output error $e_y$ for each $r$. The results are presented in Table \ref{tab_E2_error_vs_r}.

\begin{table}[htbp]
	\centering
	\caption{Average $\mathcal{L}_2$ relative errors of the output quantity $y$ for the stochastic advection-diffusion problem with different reduced dimensions $r$.}
	\label{tab_E2_error_vs_r}
	\begin{tabular}{c c c c}
		\toprule
		$r$ & ROM & DDROM & Improvement \\
		\midrule
		1 & 2.282e-01 & 1.733e-01 & 24.06\% \\
		2 & 5.083e-01 & 1.629e-01 & 67.95\% \\
		3 & 9.130e-02 & 5.550e-02 & 39.21\% \\
		4 & 1.100e-03 & 8.004e-04 & 27.24\% \\
		5 & 7.933e-05 & 5.820e-05 & 26.63\% \\
		\bottomrule
	\end{tabular}
\end{table}

The results in Table \ref{tab_E2_error_vs_r} demonstrate that the DDROM error decays exponentially as $r$ increases, thereby confirming the method's convergence. When $r$ is further increased from $4$ to $5$, the error continues to decrease to $5.82\times10^{-5}$, but the reduction rate slows down significantly, indicating that the model has converged. Balancing both accuracy and computational efficiency, $r=4$ is a suitable choice for the reduced dimension for this problem. For scenarios that demand extremely fast computation, $r=3$ may serve as an alternative.

After determining the reduced dimension $r=4$, we further evaluate the approximation performance of the proposed method over the parameter space. A set of sample points is selected in the test parameter interval $[0.1,10]$, and the absolute and relative errors of both the ROM and DDROM are computed with respect to the reference solution. The resulting error curves as functions of the parameter are shown in Figure \ref{fig_E2_GMsFEM_err}.

\begin{figure}[H]
    \centering
    \begin{subfigure}[b]{0.48\textwidth}
        \includegraphics[width=\textwidth]{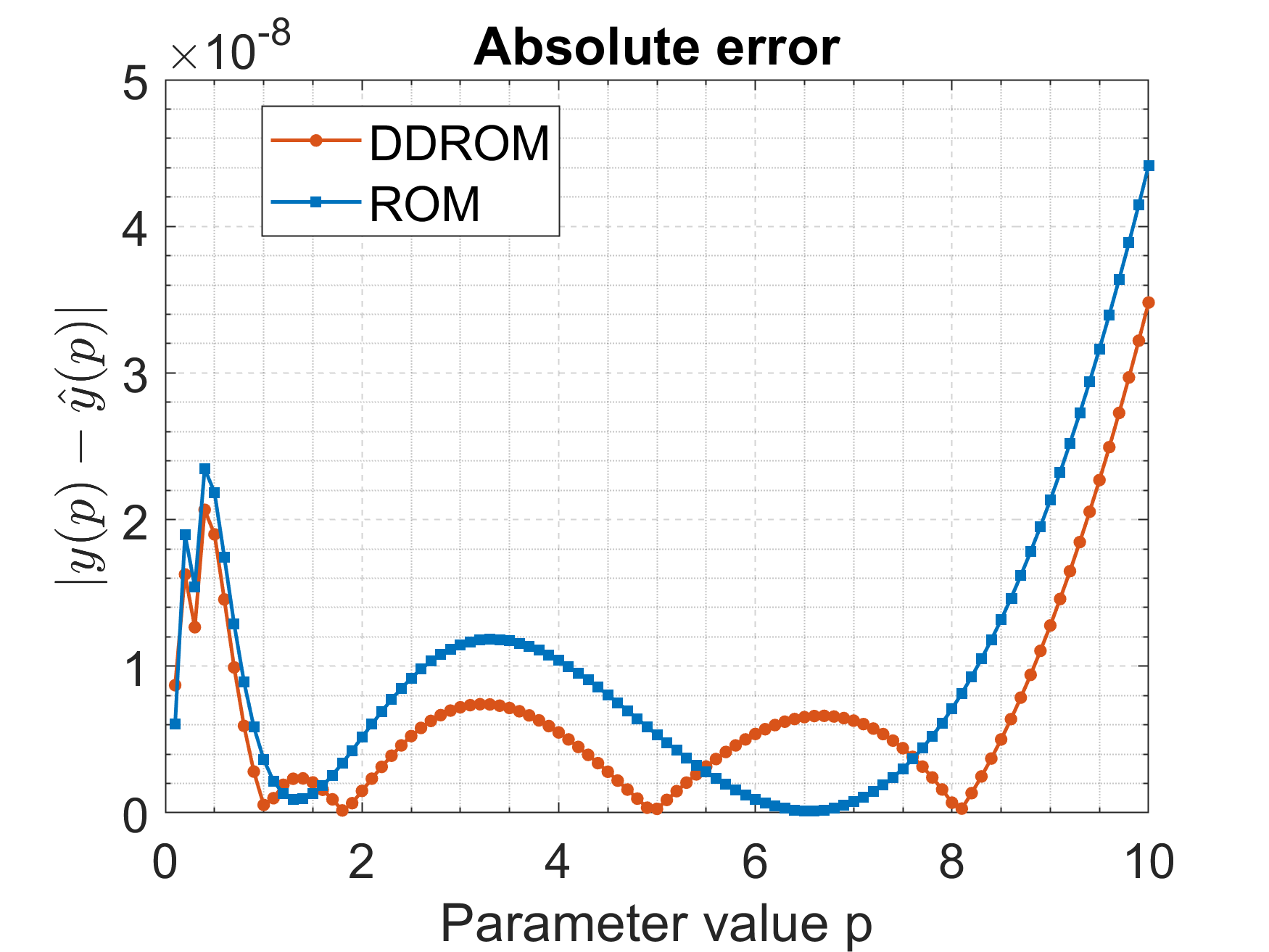}
    \end{subfigure}
    \hfill
    \begin{subfigure}[b]{0.48\textwidth}
        \includegraphics[width=\textwidth]{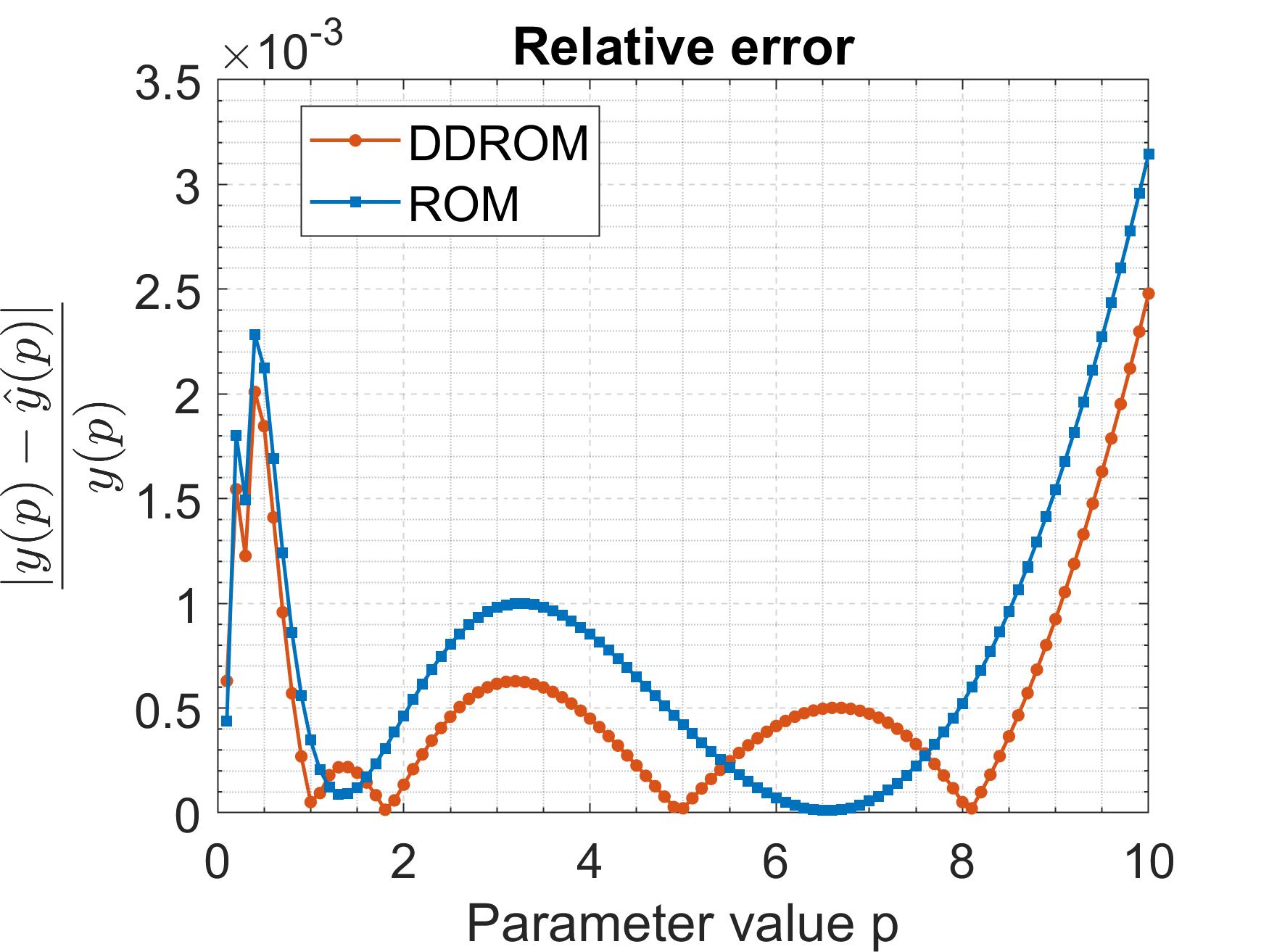}
    \end{subfigure}
    \captionsetup{justification=centering, font=small, skip=6pt}
    \caption{Error plots for the quantity of interest $y$.}
    \label{fig_E2_GMsFEM_err}
\end{figure}

As shown in the figure, both types of errors remain at low levels over the entire parameter interval---the absolute error is on the order of $10^{-10}$ and the relative error on the order of $10^{-5}$---indicating that the constructed reduced-order model achieves high approximation accuracy. A further comparison reveals that the DDROM yields significantly lower errors than the initial ROM in most parameter regions, especially for larger parameter values (e.g., $\mu \to 10$), where the optimization leads to more pronounced error reduction. This demonstrates that the $\mathcal{L}_2$-Opt-PSF method can effectively mitigate the error accumulation issues present in the initial reduced-order model.

Moreover, the error curves vary smoothly over the parameter space without exhibiting noticeable oscillations or local instabilities, confirming the good stability and robustness of the proposed method across different parameter values. In summary, Figure \ref{fig_E2_GMsFEM_err} quantitatively verifies the advantages of the proposed two-stage method in terms of both accuracy and stability. These results indicate that combining multiscale approximation with data-driven optimization can effectively enhance the performance of model reduction for complex parameter-dependent problems.

To further validate the approximation performance of the proposed method for the advection-diffusion optimal control problem, we select a representative parameter value $\overline{\mu}$ and perform a visual analysis of the spatial distributions of the state variable $U$ and the control variable $F$. Figures \ref{fig_E2_u_128_8} and \ref{fig_E2_f_128_8} present the corresponding contour plots, while Figures \ref{fig_E2_U} and \ref{fig_E2_F} show the three-dimensional surface comparisons.

\begin{figure}[H]
    \centering
    \begin{subfigure}[htb]{0.33\textwidth}
        \includegraphics[width=\textwidth]{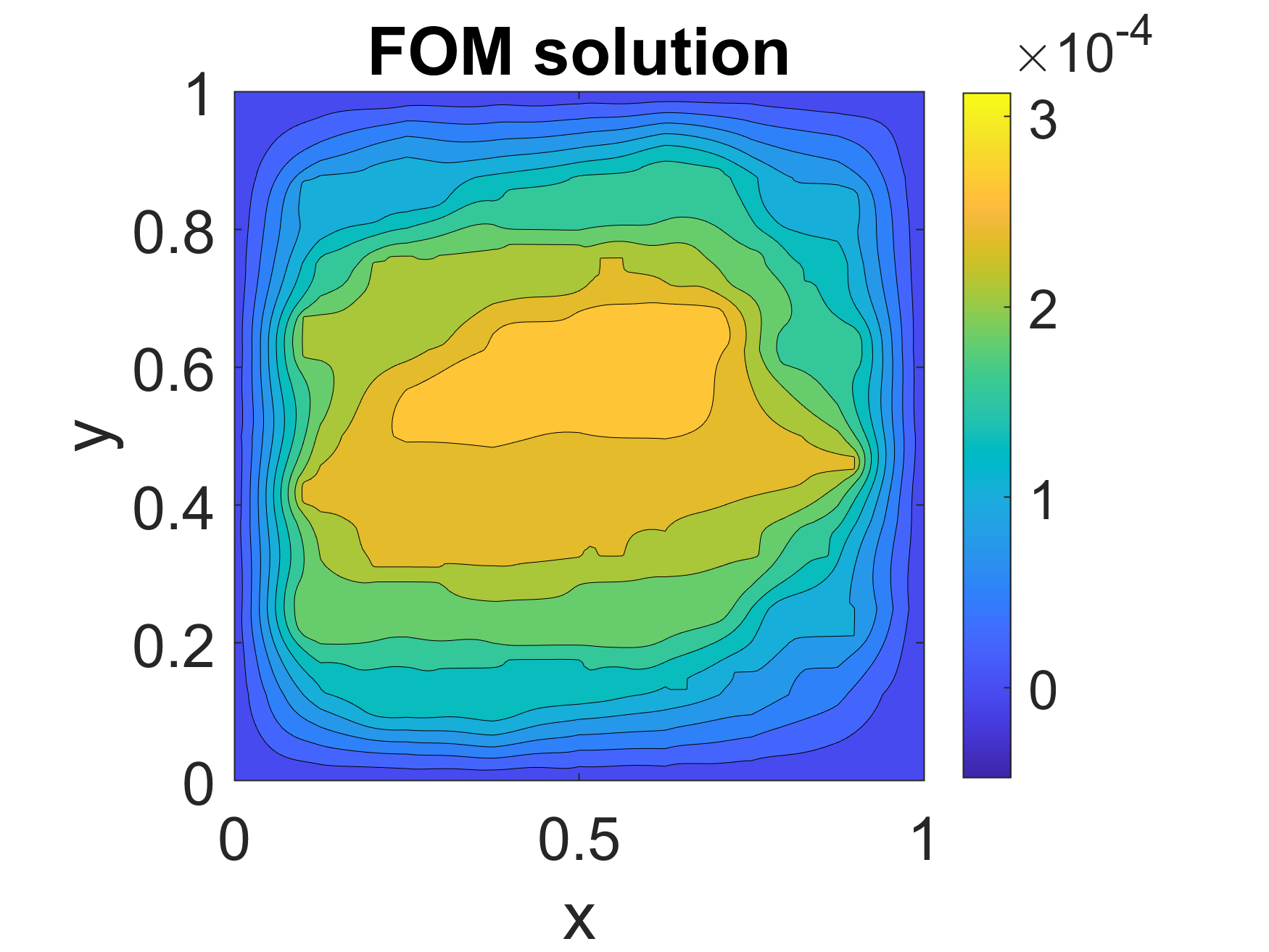}
    \end{subfigure}
    \hspace{-8pt}
    \begin{subfigure}[htb]{0.33\textwidth}
        \includegraphics[width=\textwidth]{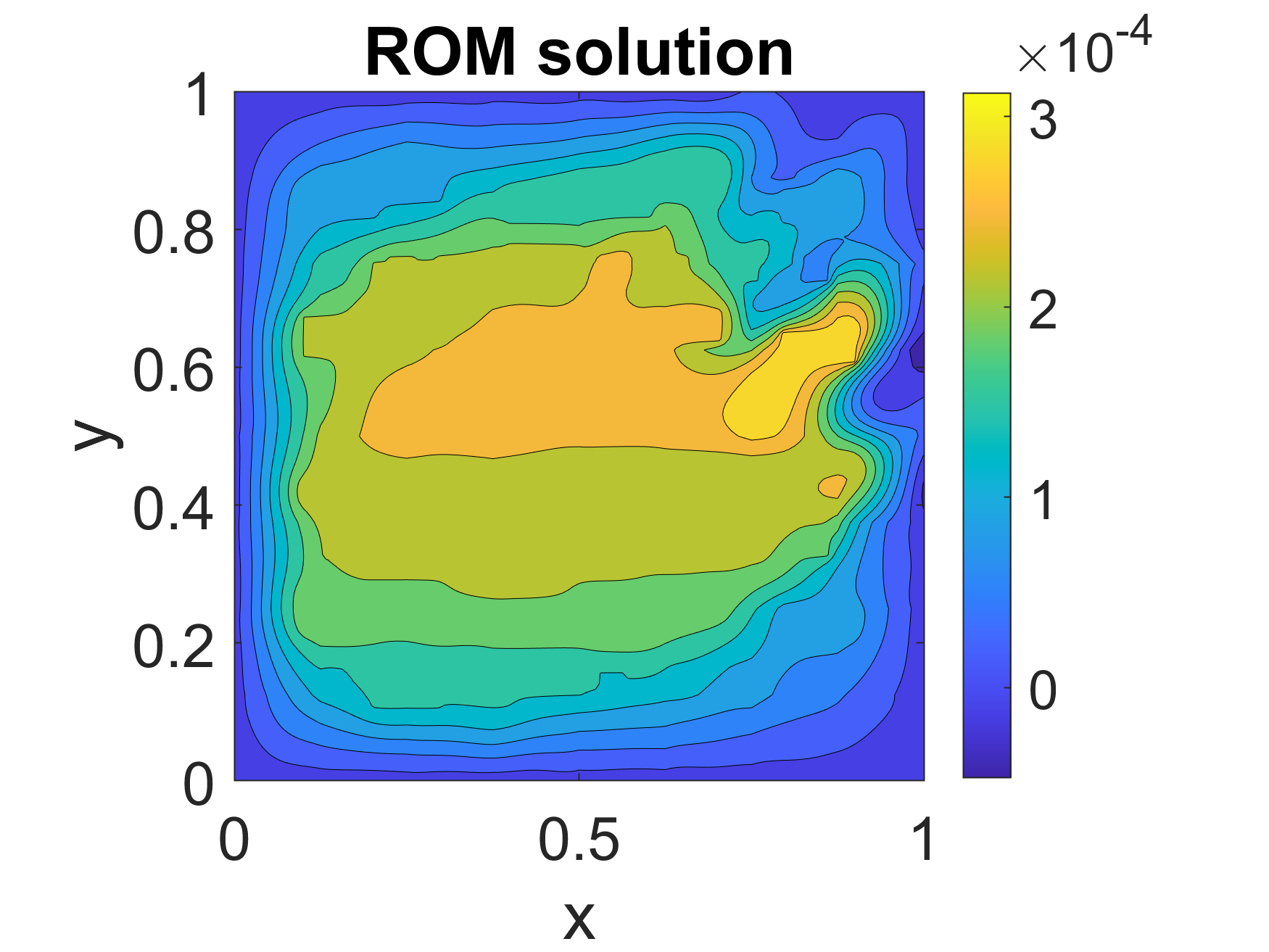}
    \end{subfigure}
    \hspace{-8pt}
    \begin{subfigure}[htb]{0.33\textwidth}
        \includegraphics[width=\textwidth]{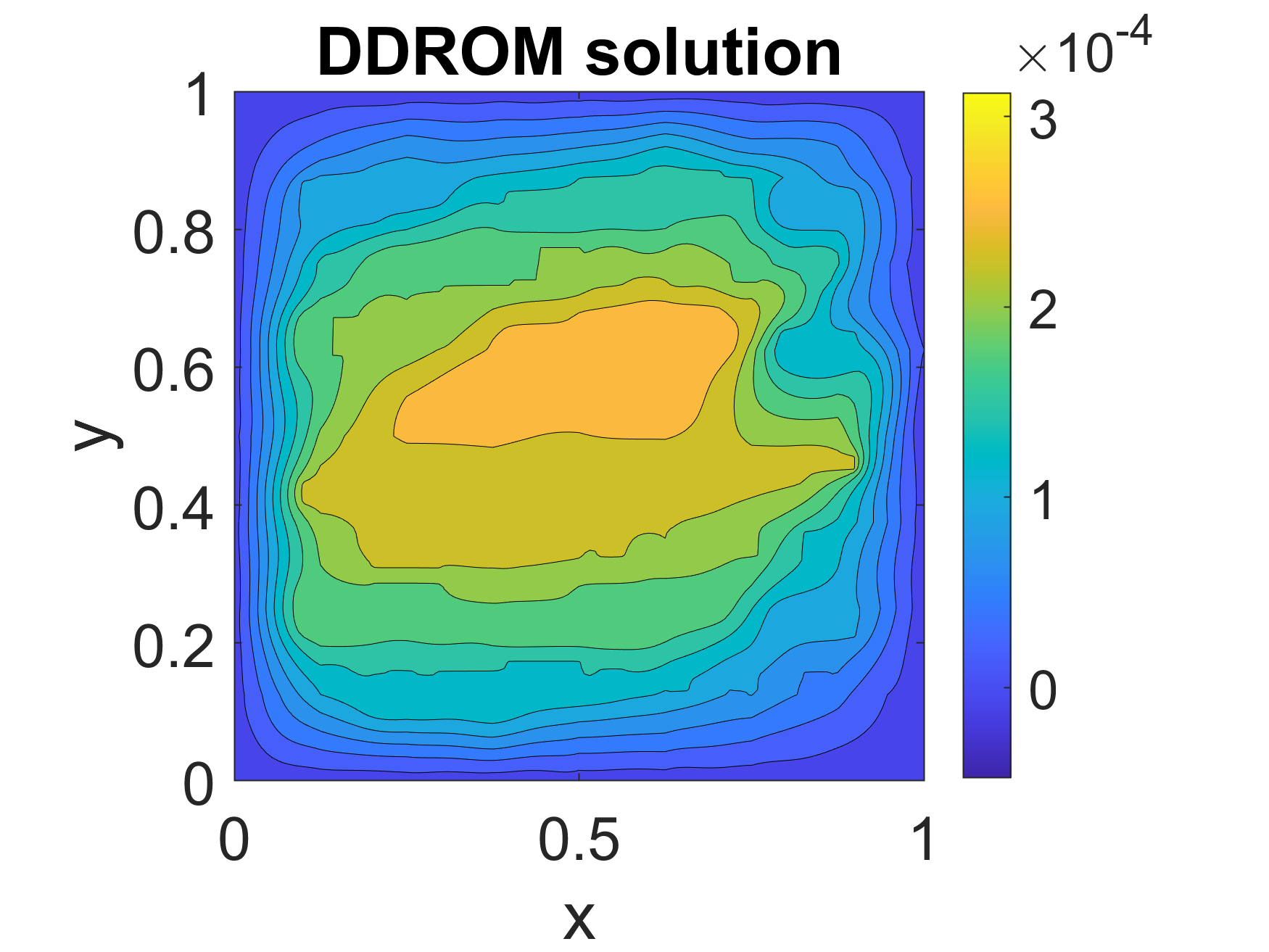}
    \end{subfigure}
    \captionsetup{justification=centering, font=small, skip=6pt}
    \caption{Images of solution $U$ at $\overline{\mu}$.}
    \label{fig_E2_u_128_8}
\end{figure}

\begin{figure}[H]
    \centering
    \begin{subfigure}[htb]{0.33\textwidth}
    \includegraphics[width=\textwidth]{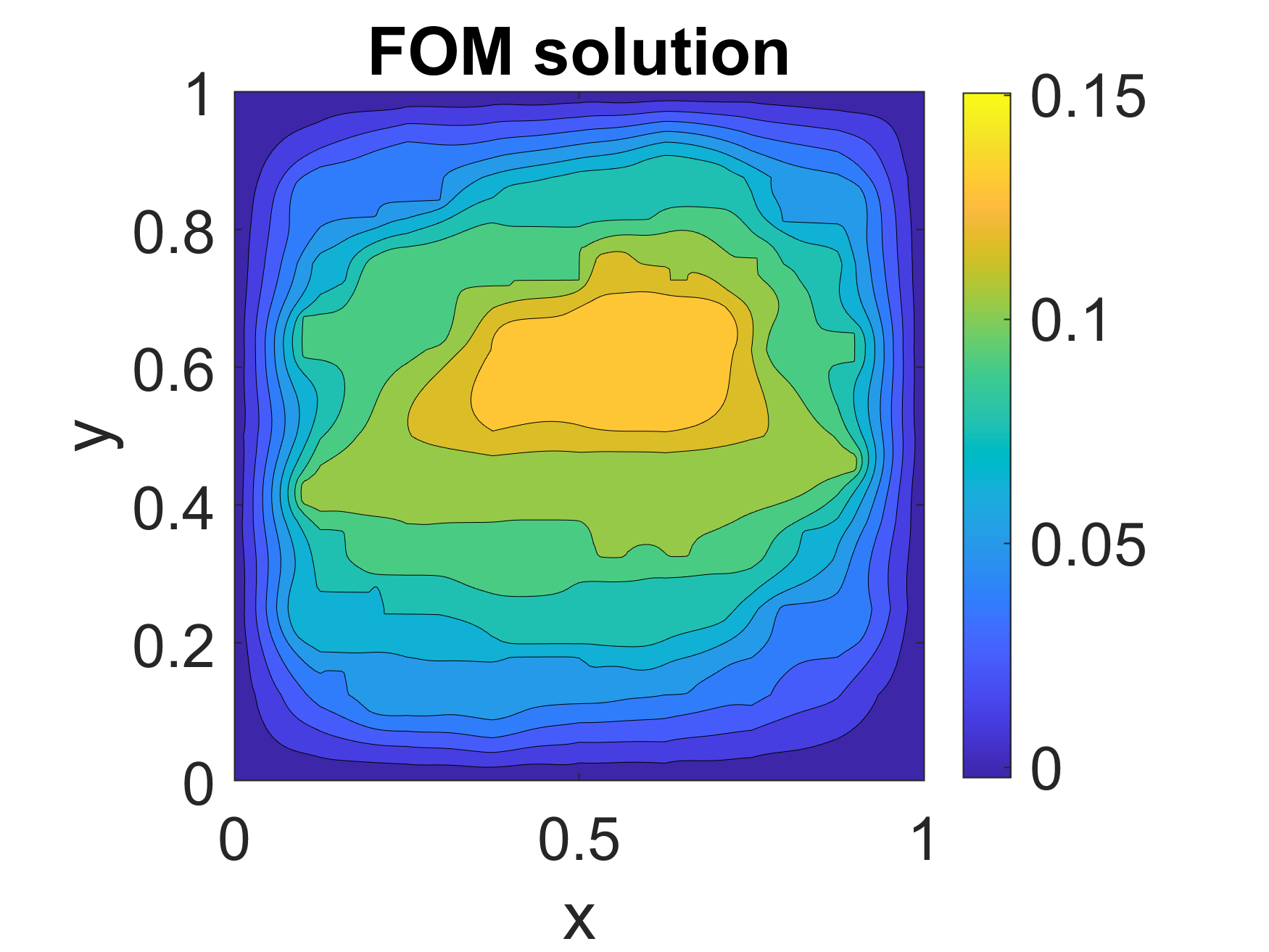}
    \end{subfigure}
    \hspace{-8pt}
    \begin{subfigure}[htb]{0.33\textwidth}
    \includegraphics[width=\textwidth]{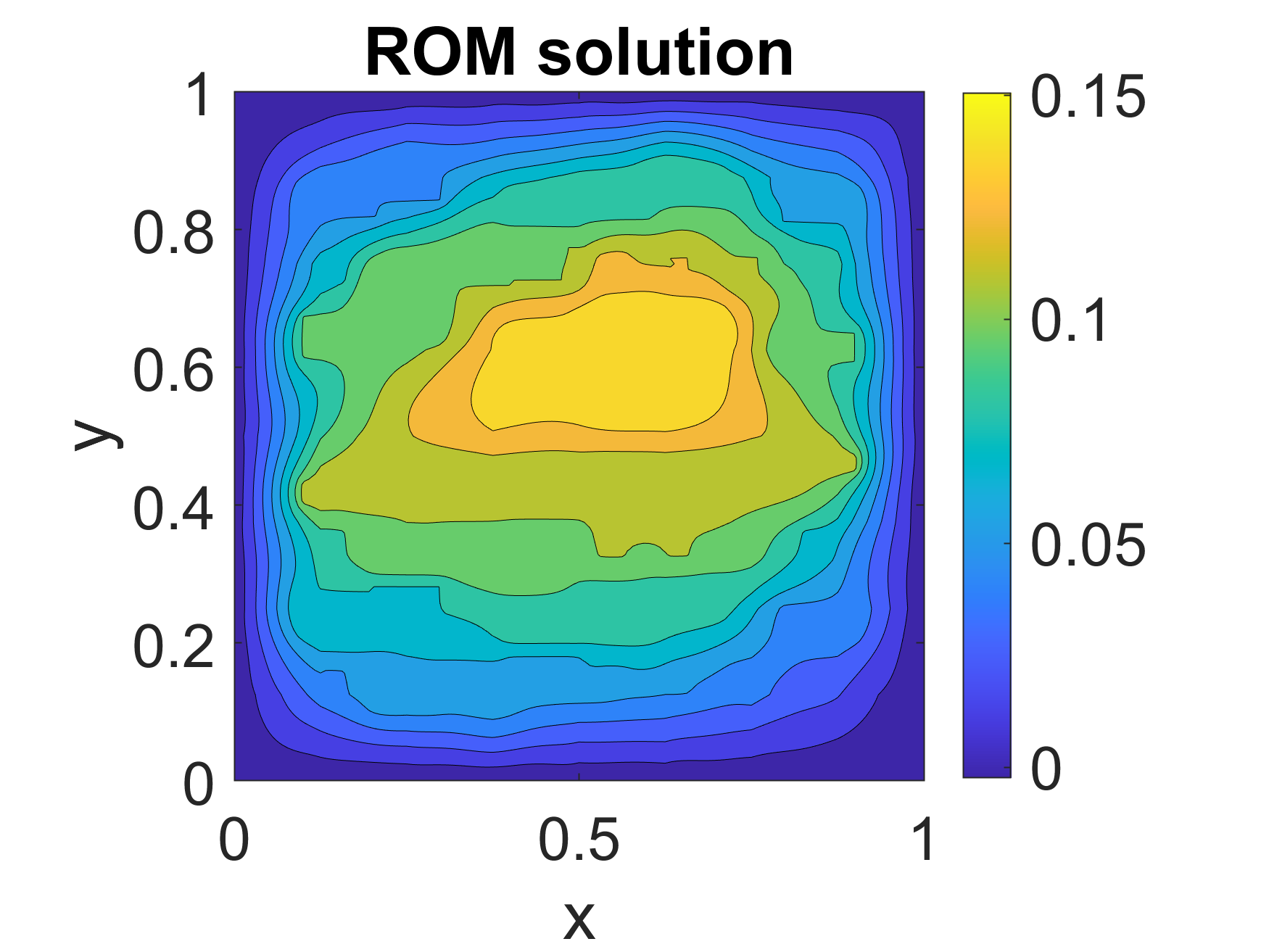}
    \end{subfigure}
    \hspace{-8pt}
    \begin{subfigure}[htb]{0.33\textwidth}
    \includegraphics[width=\textwidth]{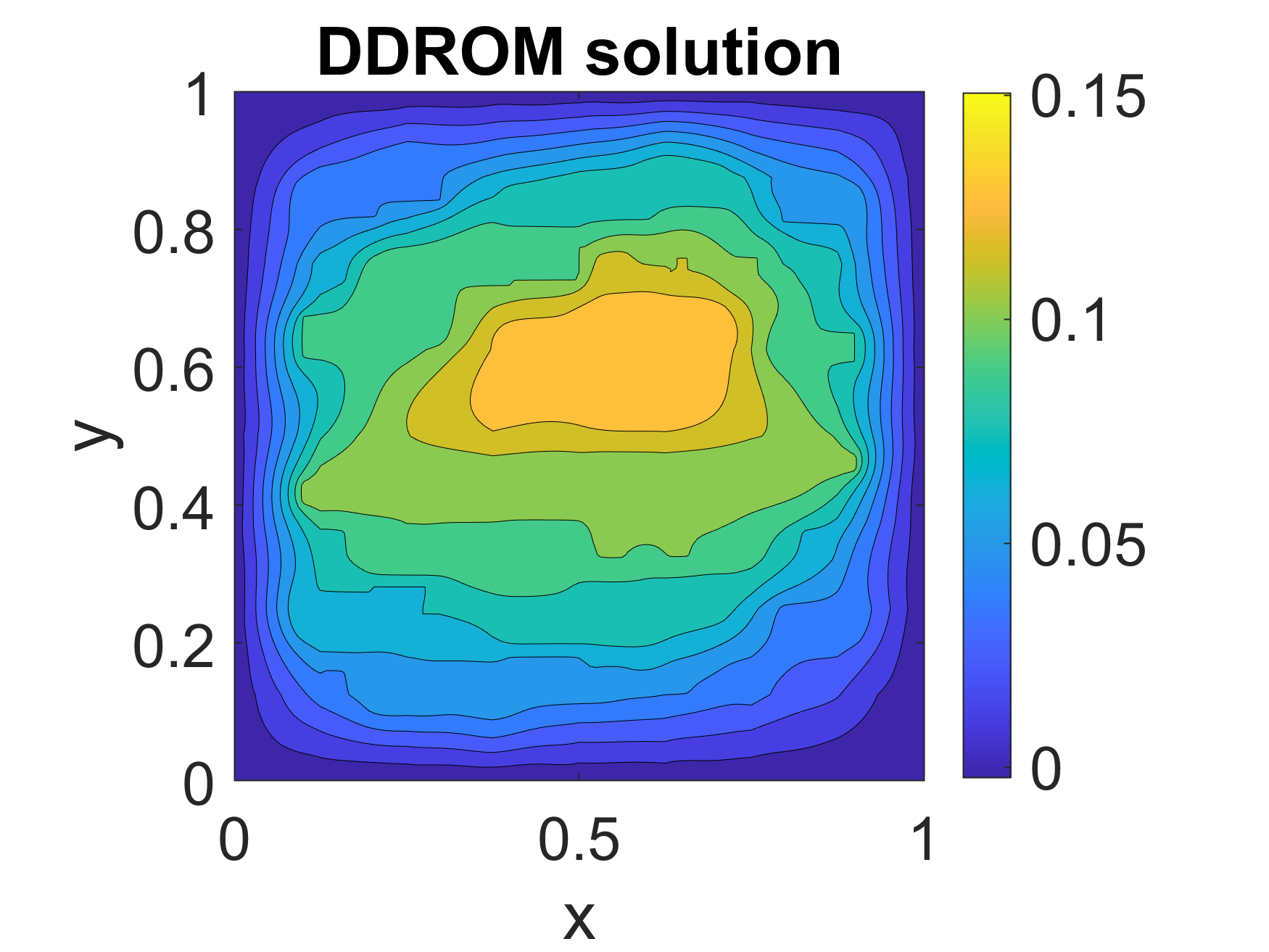}
    \end{subfigure}
    \captionsetup{justification=centering, font=small, skip=6pt}
    \caption{Images of solution $F$ at $\overline{\mu}$.}
    \label{fig_E2_f_128_8}
\end{figure}

As shown in Figure \ref{fig_E2_u_128_8}, the reference solution of the state variable $U$ exhibits a pronounced asymmetric spatial distribution due to the dominant advection effect. The initial ROM solution captures the overall structure reasonably well, but shows minor deviations in regions with sharp gradients. In contrast, the DDROM solution better matches the reference solution in these critical areas, demonstrating that the optimization improves local approximation while preserving the global structure. Similar observations are made for the control variable $F$ in Figure \ref{fig_E2_f_128_8}, where the DDROM solution achieves higher consistency than the initial ROM. This confirms that the $\mathcal{L}_2$-Opt-PSF method is effective for both state and control variables.

\begin{figure}[H]
    \centering   
    \includegraphics[width=1\linewidth]{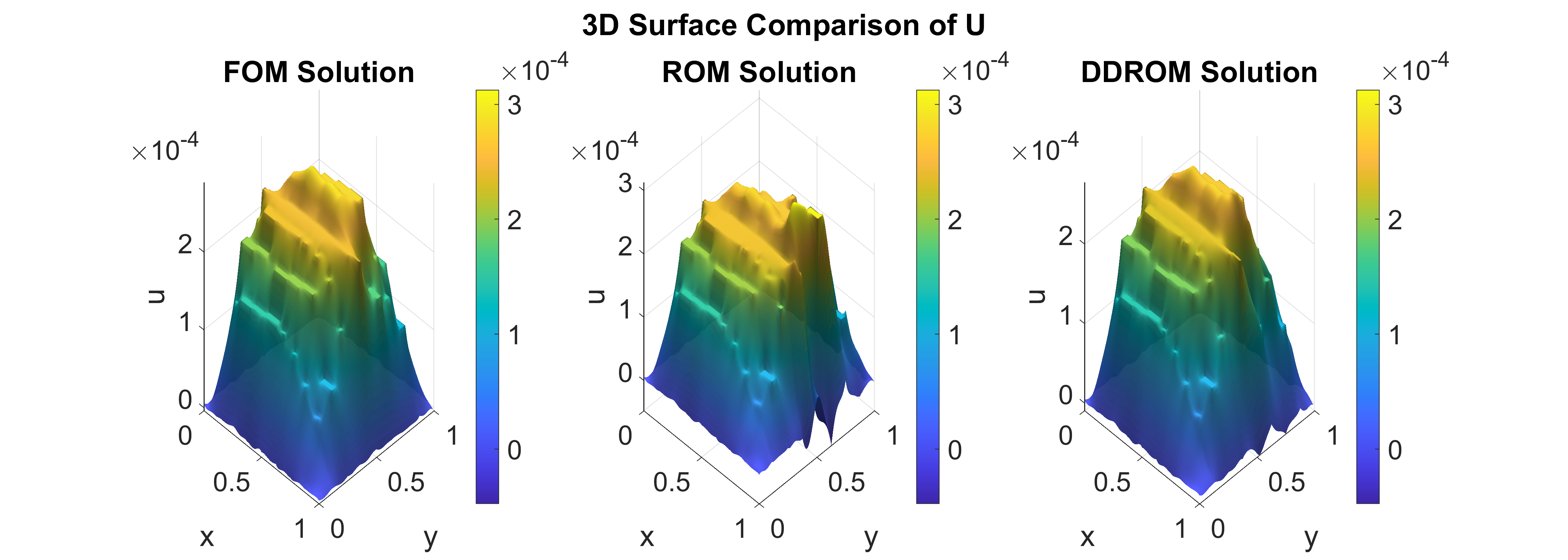}
    \caption{Three-dimensional surface comparison of $U$.}
    \label{fig_E2_U}
\end{figure}

\begin{figure}[H]
    \centering   
    \includegraphics[width=1\linewidth]{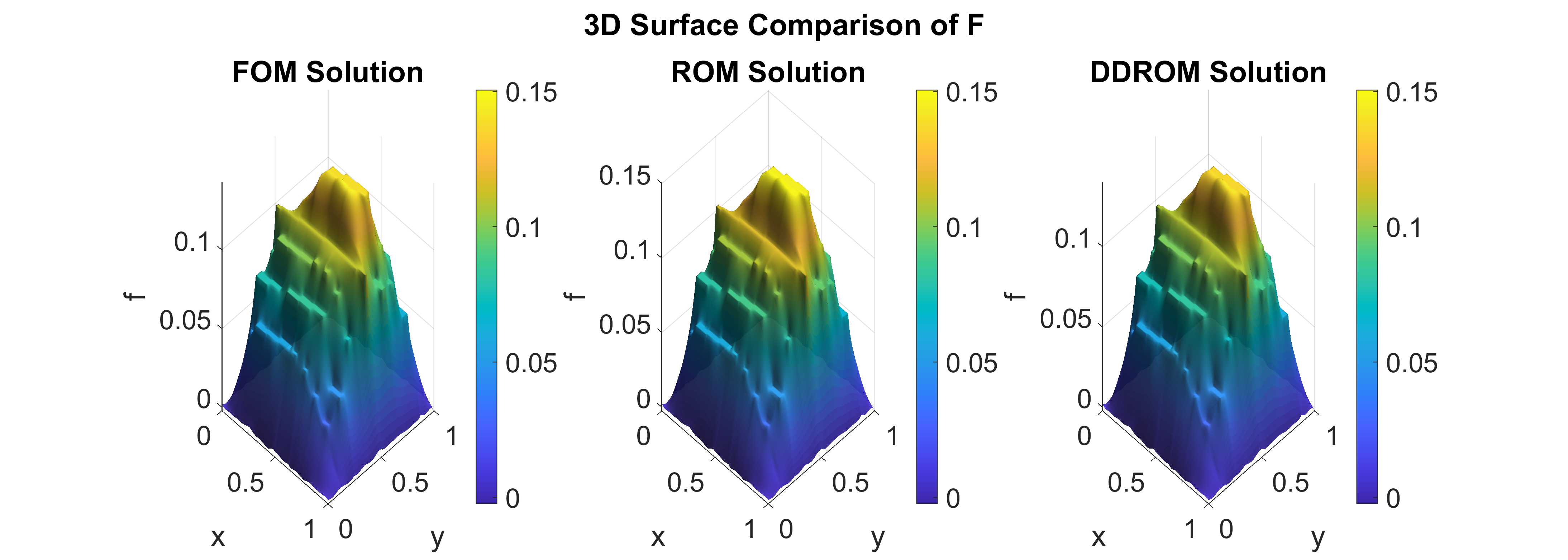}
    \caption{Three-dimensional surface comparison of $F$.}
    \label{fig_E2_F}
\end{figure}

A closer look at the three-dimensional surface plots reveals that the reference solution exhibits a pronounced peak structure and directional variation, reflecting the complex solution behavior under coupled advection-diffusion effects. The initial ROM solution captures the overall trend reasonably well, but still shows some discrepancies in peak height and local fluctuations. In contrast, the DDROM solution achieves better agreement in the peak and steep-gradient regions, with a surface shape that more closely matches the reference solution. This indicates that the optimization process effectively reduces local errors and improves the resolution of fine-scale features.

It should be noted that, while the overall visual differences between the DDROM and the initial ROM are not dramatic, this observation is reasonable. On one hand, the initial ROM already achieves high accuracy, providing a good starting point for the subsequent optimization. On the other hand, the optimization objective of the $\mathcal{L}_2$-Opt-PSF method is the $\mathcal{L}_2$ error of the output quantity, so the improvement is primarily reflected at the output level rather than in the visual appearance of the solution fields. This is consistent with the numerical results presented in Section \ref{sec_E1_GMsFEM}.

To evaluate the practical computational efficiency of the reduced-order models, we compare the CPU of the FOM, ROM, and DDROM. All tests are conducted under the same hardware and software environment. The offline time is reported as the total seconds required for construction. The average online time is computed over 10000 test parameter points. Based on these, we further report the total time for 10000 consecutive queries and the corresponding average total time per query. The results are summarized in Table \ref{tab_E2_cpu}.

\begin{table}[htbp]
    \centering
    \caption{The CPU time for different models.}
    \label{tab_E2_cpu}
    \begin{tabular}{l p{2.5cm} p{2.5cm} p{2.5cm}}
        \toprule
        CPU(s) &FOM  &ROM  &DDROM \\
        \midrule
        Offline time       & 2.686E+00 & 2.686E+00 & 2.694E+00 \\
        Average online time & 6.384E-02 & 1.778E-05 & 1.463E-05 \\
        Total time       & 6.411E+02 & 2.864E+00 & 2.841E+00 \\
        Average total time & 6.411E-02 & 2.864E-04 & 2.841E-04 \\
        \bottomrule
    \end{tabular}
\end{table}

Based on Table~\ref{tab_E2_cpu}, the offline times for all three models are comparable, as the dominant cost is shared. In terms of online efficiency, the DDROM achieves an average online time of $1.463\times10^{-5}$ seconds, which is 4360 times faster than the FOM and 1.22 times faster than the ROM. For 10000 consecutive queries, the DDROM reduces the total computational cost by 226 times compared to the FOM. These results demonstrate that the proposed method is highly advantageous for many-query applications such as real-time control and optimization under uncertainty.

In this subsection, the $\mathcal{L}_2$-Opt-PSF method is applied to the optimal control problem constrained by the stochastic advection-diffusion equation. Through convergence analysis, the optimal reduced dimension is determined to be $r=4$, under which the DDROM achieves an output error of $e_y = 8.00\times10^{-4}$. The error plots, contour plots and surface plots demonstrate that the DDROM solution is highly consistent with the reference solution.  For many-query scenarios, the DDROM substantially reduces the total computational cost compared to the FOM, highlighting its practical value for real-time control and large-scale parameter evaluation.

\section{Conlusion}
In this paper, we develop a non-intrusive, data-driven model reduction framework, termed $\mathcal{L}_2$-Opt-PSF, for solving multiscale stochastic optimal control problems governed by parameterized PDEs. Our approach leverages gradient optimization techniques combined with GMsFEM, providing an effective solution strategy for such optimal control problems. The method innovatively employs a parameter-separable form to handle dependence on stochastic parameters and directly minimizes the $\mathcal{L}_2$ norm of output errors via gradient optimization. Due to the multiscale nature of the original model, relying on the full model for gradient optimization would lead to a significant computational burden and high memory requirements. To address this, GMsFEM is employed solely as an offline solver to generate high-fidelity realizations of the optimal control for sampled random parameters. This significantly enhances computational efficiency for solving stochastic optimal control problems. Several numerical examples were carefully implemented for different stochastic optimal control problems, with the stochastic diffusion equation and the stochastic advection-diffusion equation serving as constraints. The numerical results demonstrate the efficacy of the proposed model reduction method and its promising applicability to stochastic optimal control problems governed by complex models.

\bibliographystyle{unsrt} 
\bibliography{references}

\end{document}